\documentclass[11pt]{amsart}

\usepackage{amsmath, amscd,hyperref,amssymb}
\usepackage{amsfonts}
\usepackage{amsthm, upref}
\usepackage{graphicx}
\usepackage{enumerate}
\usepackage[dvipsnames]{xcolor}
\usepackage{tikz}
\usepackage{pgf}
\usetikzlibrary{matrix,arrows, decorations.markings}
\usepackage{accents}
\usepackage{bm}
\usepackage{mathtools,centernot}
\usepackage{color,soul}

\input xy
\xyoption{all}

\numberwithin{equation}{section}

\def\E{{\mathbb{E}}}
\def\R{{\mathbb{R}}}
\def\SS{{\mathcal{S}}}
\def\N{{\mathbb{N}}}
\def\PP{{\mathbb{P}}}

\def\P{{\mathcal{P}}}

\def\B{{\mathcal{B}}}

\def\K{{\mathcal{K}}}

\def\LL{{\mathcal{L}}}
\def\Law{{\mathrm {Law}}}

\def\J{{\mathcal{J}}}

\def\D{{\mathcal{D}}}
\def\W{{\mathcal{W}}}
\def\Z{{\mathbb{Z}}}

\def\F{{\mathcal{F}}}
\def\FF{{\mathbb{F}}}
\def\cH{{\mathcal{H}}}

\def\Var{\mathrm{Var}}

\def\Re{\mathrm{Re}}
\def\Im{\mathrm{Im}}

\def\dd{\mathrm{d}}

\DeclareMathOperator*{\esssup}{ess\,sup}

\newcommand{\hxi}{\widehat{\xi}}

\newcommand{\wh}{\widehat}
\newcommand{\cC}{\mathcal{C}}
\newcommand{\nll}{\centernot{\ll}}

\theoremstyle{plain}
\newtheorem{theorem}{Theorem}[section]
\newtheorem{lemma}[theorem]{Lemma}

\newtheorem{proposition}[theorem]{Proposition}

\theoremstyle{definition}
\newtheorem{example}[theorem]{Example}
\newtheorem{definition}[theorem]{Definition}
\newtheorem{assumption}[theorem]{Assumption}
\newtheorem{remark}[theorem]{Remark}

\title[Universal Central Limit Theorem for non-exchangeable  diffusions]{Universal Central Limit Theorem for non-exchangeable interacting diffusions}
\author{Mykhaylo Shkolnikov}
\address{Department of Mathematical Sciences and Center for Nonlinear Analysis, Carnegie Mellon University, Pittsburgh, PA 15213}
\email{mshkolni@gmail.com}
\author{Lane Chun Yeung}
\address{Department of Applied Mathematics, Illinois Institute of Technology, Chicago, IL 60616}
\email{lyeung2@illinoistech.edu}
\thanks{M. Shkolnikov is partially supported by the National Science Foundation grant DMS-2342349.}

\begin{document}

\begin{abstract}
	We study non-exchangeable interacting diffusions with pairwise interaction strengths encoded by a sequence of matrices. Under suitable structural and denseness conditions on these matrices, we prove a universal Central Limit Theorem for the global fluctuation field. As the number of particles $n$ becomes large, it converges in distribution to the unique solution of a stochastic partial differential equation (SPDE), the same Gaussian limit as in the exchangeable mean field case. The result applies, for instance, to scaled adjacency matrices of $m_n$-regular graphs when $m_n/\sqrt{n}\to\infty$. A spatial interaction model shows that the $n^{-1/2}$ denseness threshold is sharp. The proof proceeds with an analysis in negative Sobolev spaces, building on  sharp quantitative propagation of chaos results together with functional inequalities.
\end{abstract}
	\maketitle
	
	\section{Introduction}
	
	We consider a system of interacting diffusions $(X^1,\dots,X^n)$ on $\R^d$. For $i=1,\dots,n$,
	\begin{align} \label{eq:X}
		\dd X^i_t = \Big(b_0(t, X^i_t) + \sum_{j=1}^{n} \xi_{ij} b(t, X^i_t, X^j_t)\Big) \, \dd t + \sigma  \, \dd B^i_t, \quad X^i_0 \stackrel{\text{i.i.d.}}{\sim} \mu_0.
	\end{align}
	Here $B^1,\dots,B^n$ are independent standard Brownian motions and $\sigma>0$ is constant. The coefficient $b_0$ describes the \emph{self-interaction} of each particle, while $b$ describes the \emph{pairwise interaction}. The matrix $\xi=(\xi_{ij})_{i,j=1}^n$ specifies the interaction strengths between particles. We assume that $\xi$ has nonnegative entries and zero diagonal. Thus $\xi_{ij}$ quantifies the influence of particle $j$ on particle $i$.
	
	\subsection{The exchangeable mean field case}
	
	A central question for systems of the form \eqref{eq:X} is to understand their large-$n$ behavior. At the Law of Large Numbers level, one studies the limiting behavior of a typical particle. When the interaction strengths are homogeneous, namely $\xi_{ij}=1/(n-1)$ for all $i\neq j$, the interactions in \eqref{eq:X} are of \emph{mean field} type. In this case, the particle system $(X^1_t,\dots,X^n_t)$ is exchangeable. Under suitable regularity assumptions on $b_0$ and $b$, the empirical measure $\mu^n_t := \frac{1}{n} \sum_{i=1}^{n} \delta_{X^i_t}$ is expected to converge to a deterministic limiting probability measure $\mu_t$, characterized by the \emph{McKean-Vlasov equation}
	\begin{align} \label{eq:MKV}
		\dd Y_t =
		\Big(
		b_0(t, Y_t)
		+
		\int_{\R^d} b(t, Y_t, y) \, \mu_t(\dd y)
		\Big)\dd t
		+ \sigma \, \dd B_t,
		\quad
		\mu_t = \Law(Y_t),
		\quad
		Y_0 \sim \mu_0.
	\end{align}
	This phenomenon is known as \emph{propagation of chaos}. Since the seminal works of McKean and Sznitman \cite{mckean1967propagation}, \cite{sznitman2006topics}, propagation of chaos has been established for many mean field models. We refer to the surveys \cite{chaintron2021propagation,chaintron2022propagation} for detailed accounts.
	
	The focus of this paper is the next-order behavior beyond propagation of chaos. Once the empirical measure is known to converge to $\mu_t$, it is natural to study the fluctuations around it. This leads to the global fluctuation field
	\begin{align} \label{eq:def.fluc.process}
		\eta^n_t := \sqrt{n}\big(\mu^n_t-\mu_t\big)
		= \frac{1}{\sqrt{n}}\sum_{i=1}^{n}(\delta_{X^i_t}-\mu_t),
		\qquad t\in[0,T].
	\end{align}
	In the exchangeable mean field case, one expects $\eta^n$ to converge to a Gaussian random field $\eta$, formally characterized as the solution of the linear stochastic partial differential equation (SPDE)
	\begin{align} \label{eq:SPDE.intro}
		\begin{split}
			\partial_t \eta_t(x)
			&= \frac{\sigma^2}{2}\Delta \eta_t(x)
			- \nabla\cdot\big(b_0(t,x)\eta_t(x)\big)
			- \nabla\cdot\Big(\eta_t(x)\int_{\R^d} b(t,x,y)\,\mu_t(\dd y)\Big) \\
			&\quad
			- \nabla\cdot\Big(\mu_t(x)\int_{\R^d} b(t,x,y)\,\eta_t(\dd y)\Big)
			- \sigma\nabla\cdot\big(\sqrt{\mu_t(x)}\,\dot{W}(t,x)\big).
		\end{split}
	\end{align}
	Here $\dot{W}$ denotes the $\R^d$-valued space-time white noise. Fluctuation results for mean field interacting diffusions go back to  \cite{tanaka1981central}, \cite{tanaka1984limit}, \cite{hitsuda1986tightness}, \cite{fernandez1997hilbertian}, and we refer to Subsection \ref{sec:literature} for a more comprehensive review.
\subsection{The non-exchangeable setting}

In many applications, the homogeneous interaction assumption $\xi_{ij}=1/(n-1)$ is too crude. A richer description is obtained by placing the particles on a weighted graph, whose edge weights specify the interaction strengths. Such network-based models arise in opinion and collective dynamics \cite{ayi2021mean}, \cite{duteil2021mean}, \cite{ben2024mean}, systemic risk, interbank markets, and financial contagion \cite{allen2000financial}, \cite{boss2004network}, \cite{cont2013network}, \cite{elliott2014financial}, and neuroscience \cite{jabin2025mean}, \cite{brunel2000dynamics}. This motivates the general framework in \eqref{eq:X}, where the matrix $\xi$ encodes the underlying connection structure. 

A general interaction matrix $\xi$ breaks the exchangeability of the particle system, and such heterogeneous systems have been studied intensively over the past decade. A natural question is one of \emph{universality}. At the level of propagation of chaos, one asks for which sequences of $n\times n$ interaction matrices $\xi$ the empirical measure of the particle system converges to the usual McKean-Vlasov limit \eqref{eq:MKV} as $n\to\infty$. Early results in this direction include \cite{bhamidi2019weakly}, \cite{delattre2016note}, \cite{coppini2020law}, which treat Erdős--Rényi graphs and related random graph models, with the interaction matrix normalized by the expected degree. The work \cite{oliveira2019interacting} reaches the optimal denseness condition $np_n\to\infty$ for the hydrodynamic limit in the Erdős--Rényi setting, in an \emph{annealed} sense. In contrast, \cite{coppini2023central} proves a \emph{quenched} Law of Large Numbers for the global empirical measure under the same optimal condition. More generally, when $\xi$ is sufficiently dense and has row sums close to $1$, one still expects the empirical measure to converge to the usual McKean--Vlasov limit. When the row sums equal $1$, the work \cite{lacker2024quantitative} quantifies how close \eqref{eq:X} is to this limit in terms of the interaction matrix $\xi$. When $\xi$ is dense but its row sums are not close to $1$, \eqref{eq:MKV} is not the correct large-$n$ limit, and alternative limits have been derived using the theory of dense graph limits \cite{lovasz2012large}, for example in \cite{medvedev2014nonlinear}, \cite{chiba2016mean}, \cite{medvedev2018continuum}, \cite{bayraktar2022stationarity}, \cite{bayraktar2023graphon}, \cite{bet2024weakly}, \cite{jabin2025mean}. 

The same universality question can be asked at the level of fluctuations. Namely, for which sequences of interaction matrices $\xi$ does the fluctuation field \eqref{eq:def.fluc.process} converge to the mean field SPDE limit \eqref{eq:SPDE.intro}?
	
	\subsection{Summary of main result}
	Our main result, Theorem \ref{thm:fluctuation.clt} below, establishes a universal Central Limit Theorem for the interacting particle system \eqref{eq:X}, valid for a class of interaction matrices $\xi$ satisfying suitable structural and denseness conditions. Here, universality means that for every $\xi$ in the class, the fluctuation field \eqref{eq:def.fluc.process} converges to the same Gaussian limit, the solution of the SPDE \eqref{eq:SPDE.intro} that arises in the exchangeable mean field case.
	
Our assumptions are deterministic and make no use of the randomness of an underlying graph. As discussed in Examples \ref{ex:graphs} and \ref{ex:LS}, they cover the scaled adjacency matrix of any $m_n$-regular graph when $\sqrt{n}/m_n\to0$, the Erd\H{o}s--R\'enyi graph $G(n,p_n)$ under the analogous threshold $\sqrt{n}/(np_n)\to0$, and the spatial interaction models of \cite{lucon2016transition} in the subcritical regime and in the absence of disorder variables.
	In these examples, the condition amounts to the requirement that the largest entry of $\xi$ be $o(n^{-1/2})$, and this $n^{-1/2}$ threshold is sharp for the spatial interaction models, where the rescaled fluctuations in the supercritical regime converge to a deterministic limit (see Example \ref{ex:LS}).
	
\subsection{Overview of proof strategy}
Our approach for proving Theorem \ref{thm:fluctuation.clt} is the tightness-limit-uniqueness route. We first establish tightness of the sequence of fluctuation measures $(\eta^n)_{n \in \N}$, then identify any subsequential limit point as a solution of the SPDE \eqref{eq:SPDE.intro}, and finally prove uniqueness for the SPDE.

The challenge in proving tightness is the simultaneous presence of the heterogeneous interactions and the $\sqrt{n}$ scaling. We briefly illustrate how the sharp quantitative propagation of chaos result in \cite{lacker2024quantitative} and functional inequalities can be used to overcome this challenge. For any nice test function $\varphi : \R^d \to \R$, we observe the decomposition
\begin{align*}
	\E\big[|\big\langle \eta_t^n,\varphi\big\rangle|^2\big]
	= 
	\E\bigg[\bigg(\frac{1}{\sqrt{n}}\sum_{i=1}^n\big(\varphi(X_t^i)-\E[\varphi(X_t^i)]\big)\bigg)^{2}\bigg]
	+
	\bigg| \frac{1}{\sqrt{n}} \sum_{i=1}^{n} \big\langle \Law(X^i_t) - \mu_t,\varphi\big\rangle \bigg|^{2}.
\end{align*}
The first expectation is $n \Var\left(\frac{1}{n}\sum_{i=1}^{n}\varphi(X_t^i)\right)$. The diffusion $(X_t^1,\dots,X_t^n)$ satisfies a \emph{Poincar\'e inequality} with a constant independent of $t$ and $n$ (see Proposition \ref{lem:preliminaries}\eqref{lem:preliminaries:iii}), so this variance is $O(1/n)$, hence the first term is $O(1)$. For the second term, the Cauchy-Schwarz inequality, Kantorovich duality, and the transport inequality for $\mu_t$ (see Proposition \ref{lem:preliminaries}\eqref{lem:preliminaries:v}) give
\begin{align*}
	\bigg| \frac{1}{\sqrt{n}} \sum_{i=1}^{n} \big\langle \Law(X^i_t) - \mu_t,\varphi\big\rangle \bigg|^{2}
	&\le \sum_{i=1}^{n}\big|\big\langle \Law(X^i_t) - \mu_t,\varphi\big\rangle\big|^2 \\
	&\lesssim \sum_{i=1}^{n} \W_2^2 (\Law(X^i_t), \mu_t)
	\lesssim \sum_{i=1}^{n} H(\Law(X^i_t) \, \vert \, \mu_t),
\end{align*}
where $\W_2$ denotes the $2$-Wasserstein distance and $H$ the relative entropy. Throughout the paper, we write $A \lesssim B$ to mean that there exists a constant $C>0$, independent of $n$, such that $A \le C B$. Using the sharp propagation of chaos estimates in \cite{lacker2024quantitative} (see Lemma \ref{lem:max.ent.POC} and Remark \ref{rem:data.processing}), this entropy sum is of the order $\sum_{i=1}^{n} \big(\sum_{j=1}^{n} (\xi_{ij}^2 + \xi_{ji}^2)\big)^2$.
It is instructive to consider scaled adjacency matrices of simple undirected graphs (see Example \ref{ex:graphs}). For an $m_n$-regular graph, $\xi_{ij}=\mathbf{1}_{\{i\sim j\}}/m_n$, the entropy sum is of order $O(n/m_n^2)$, so $\liminf_{n \to \infty} m_n /\sqrt{n} > 0$ is sufficient for tightness, with an analogous threshold $\liminf_{n\to\infty} np_n/\sqrt{n}>0$ for the Erd\H{o}s--R\'enyi graph $G(n,p_n)$.   These bounds first control $\eta_t^n$ in the negative Sobolev space $\cH^{-k}$, which gives tightness of the time marginals, and the Aldous criterion then upgrades this to tightness of $(\eta^n)_{n \in \N}$ in $C([0,T];\cH^{-k})$.

After tightness has been obtained, it remains to identify the subsequential limits and to prove uniqueness. For the identification, we use the semimartingale decomposition for $\langle \eta_t^n,\varphi\rangle$ and pass to the limit in each term. The martingale part is identified through its quadratic variation, which gives the Gaussian noise term in \eqref{eq:SPDE.intro}. The finite-variation part contains the linearized mean field drift and additional error terms caused by the heterogeneity of $\xi$. This mean field drift, the same as in the classical exchangeable case, converges to the linear drift in \eqref{eq:SPDE.intro}, but the noncompact state space makes this step delicate. The functions appearing in the drift are bounded but need not decay at infinity, whereas $\eta^n$ converges only in the unweighted space $C([0,T];\cH^{-k})$. We localize each such function with a cutoff, pass to the limit on the compact part, and control the tail by uniform weighted moment estimates. The remaining heterogeneous error terms are shown to vanish using functional inequalities and the quantitative propagation of chaos result of \cite{lacker2024quantitative}. Every subsequential limit is therefore a weak solution (both in the probabilistic and distributional sense)  of \eqref{eq:SPDE.intro}.

For uniqueness, we adapt the energy inequality argument of \cite[Sections 2.5 and 3.4]{rozovsky2018stochastic}, used also in \cite{kurtz2004stochastic},\cite{delarue2019masterclt}. If $\eta$ and $\eta'$ are two solutions driven by the same noise and with the same initial condition, then the noise cancels in the difference $\delta_t:=\eta_t-\eta'_t$, and $\delta$ solves a deterministic linear equation. Estimating $\delta$ in a weighted negative Sobolev norm and applying Gronwall's inequality gives $\delta\equiv 0$. Hence, the limiting SPDE enjoys pathwise uniqueness, and therefore uniqueness in law.

	\subsection{Related literature}\label{sec:literature}
	
	\subsubsection{Fluctuations for mean field systems}
	Fluctuation theory for mean field systems has a long history. In deterministic settings, Vlasov-type limits and the associated fluctuations have been studied in \cite{braun1977vlasov}, \cite{lancellotti2009fluctuations}. For stochastic interacting mean field diffusions, Central Limit Theorems have been established in  \cite{tanaka1981central}, \cite{ito1983distribution}, \cite{tanaka1984limit}, \cite{fouque1984convergence}, \cite{sznitman1985fluctuation}, \cite{shiga1985central}, \cite{hitsuda1986tightness}, \cite{mitoma1985infinite}, \cite{meleard1996asymptotic}, \cite{fernandez1997hilbertian}.
	Fluctuation results have also been obtained for moderately interacting diffusions \cite{oelschlager1987fluctuation,jourdain1998propagation}, particle representations of nonlinear  SPDEs \cite{kurtz2004stochastic}, rank-based models \cite{kolli2018spde}, and mean field games \cite{delarue2019masterclt}.
	
	More recent work addresses quantitative estimates and variants of the fluctuation problem. In \cite{jourdain2021central}, a Central Limit Theorem is proved for nonlinear observables of the empirical measure, while \cite{cecchin2025convergence} obtains convergence rates for fluctuations of mean field diffusions. Finite-dimensional and finite-state analogues for nonlinear Markov and mean field game models have been obtained in \cite{kolokoltsov2010nonlinear,kolokoltsov2011markov,cecchin2019convergence}. Uniform-in-time quantitative CLTs for mean field fluctuation processes have recently been proven in \cite{bourguin2026uniform}. See also \cite[Theorem 1.4]{bernou2026uniform} for the Langevin and overdamped settings.
	
	There has also been substantial progress for singular interactions. Gaussian fluctuation limits for such systems have been proved in \cite{wang2023gaussian}, with common-noise variants in \cite{shao2025fluctuation}, \cite{nikolaev2025fluctuation}. Related equilibrium fluctuation results have been obtained in \cite{bodineau1999stationary}, \cite{grotto2020central}, \cite{geldhauser2021limit}, \cite{leble2018fluctuations}, \cite{serfaty2023gaussian}, \cite{gu2024quantitative}.
	
	\subsubsection{Fluctuations for non-exchangeable systems}
	
	Fluctuations for non-exchangeable systems are much less studied. Finite-type and multi-species models are treated in \cite{budhiraja2016fluctuation}, \cite{chen2016fluctuation}. In \cite{budhiraja2016fluctuation}, the particles are divided into finitely many populations. The full system is not exchangeable, but particles within the same population are. Without a common stochastic factor, the centered path-space fluctuation fields converge to a centered Gaussian field. When a common factor is present, the corresponding randomly centered fluctuations have a Gaussian-mixture limit. The work \cite{chen2016fluctuation} treats a two-species system of reflected diffusions interacting through partial annihilation near an interface.
	
	The work \cite{lucon2016transition} treats spatially extended mean field diffusions, in which each particle $\theta^i$ is attached to a site $x_i$ of a periodic one-dimensional lattice and the interaction strength between $\theta^i$ and $\theta^j$ decays like $|x_i-x_j|^{-\alpha}$ for $\alpha\in[0,1)$. Building on the Law of Large Numbers of \cite{lucon2014mean}, they exhibit a phase transition governed by the decay exponent $\alpha$. In the subcritical regime $\alpha\in[0,1/2)$, they prove a Central Limit Theorem. In the supercritical regime $\alpha\in(1/2,1)$, the relevant scaling is instead $n^{1-\alpha}$, and the rescaled fluctuations converge to a deterministic limit. In the absence of disorder variables, the Central Limit Theorem in the subcritical regime is recovered as a special case of Theorem \ref{thm:fluctuation.clt}, while the supercritical regime shows that the $n^{-1/2}$ threshold in our denseness assumption on the matrix \eqref{eq:sufficient.matrix.max} is sharp. We refer to Example \ref{ex:LS} for a more detailed discussion.
	
The same $\sqrt{n}$ denseness threshold also appears in static models. For Ising models on approximately $m_n$-regular graphs with a spectral gap, \cite{deb2023fluctuations} proves that the fluctuations of the magnetization are universal and match the Curie-Weiss model when $\sqrt{n}/m_n\to0$. This threshold is exactly the condition in our regular graph example (Example \ref{ex:graphs}\eqref{rem.b}), and the authors show it to be tight outside the high-temperature regime.
	
	Closer to the present work are fluctuation results on random graphs. For a possibly time-varying, dense inhomogeneous random graph model, \cite[Theorem 4.2]{bhamidi2019weakly} establishes a path-space empirical-measure Central Limit Theorem, under a nondegeneracy condition that keeps the edge probabilities bounded away from zero uniformly in time. For a time-independent Erd\H{o}s--R\'enyi graph $G(n,p_n)$, this requires $\liminf_{n\to\infty}p_n>0$, so the graph must remain dense.
	
	The closest result to the present work is \cite{coppini2023central}, which proves Central Limit Theorems for both the global and the local empirical measures of diffusions on the one-dimensional torus interacting through a possibly sparse Erd\H{o}s--R\'enyi graph $G(n,p_n)$, with the adjacency matrix normalized by the expected degree, that is $\xi_{ij} = \mathbf{1}_{\{i\sim j\}}/(np_n)$. For the global fluctuations, under the denseness condition $np_n^4\to\infty$, they recover the universal mean field fluctuation SPDE when the initial data, which may be non-i.i.d., are independent of the graph and the corresponding quenched or annealed initial fluctuation assumptions hold. When the initial data are allowed to depend on the graph, they exhibit an example for which the limiting field is no longer given by this universal SPDE. Their proof relies on the independence of the edges and extensions of Grothendieck's inequality. By comparison, Theorem \ref{thm:fluctuation.clt} produces the same universal mean field SPDE. When the adjacency matrix is normalized by the degree, $\xi_{ij}=\mathbf{1}_{\{i\sim j\}}/\mathrm{deg}(i)$, it gives a quenched Central Limit Theorem under the weaker threshold $\liminf_{n\to\infty} np_n/\sqrt{n}=\infty$, as explained in Example \ref{ex:graphs}\eqref{rem.c}. We expect this threshold to be sharp, as it is exactly the condition \eqref{eq:sufficient.matrix.max}, where $\max_{i,j}\xi_{ij}\sim(np_n)^{-1}$, and, in the absence of disorder variables, this $n^{-1/2}$ scaling is already sharp in the spatial interaction model of \cite{lucon2016transition} (Example \ref{ex:LS}). More fundamentally, our assumptions on $\xi$ are deterministic and make no use of the randomness of the graph, and instead our proof builds on the quantitative propagation of chaos shown in  \cite{lacker2024quantitative}.
	
	Finally, a special non-exchangeable setting, that of \emph{sequential} interacting diffusions, is initiated in \cite{du2023sequential} and studied further in \cite{wang2026quantitative}. Here the interaction matrix $\xi$ is lower triangular, so particle $i$ is influenced only by a weighted average of its predecessors $j<i$, a structure motivated by computational considerations. A notable special case is the uniform one, $\xi_{ij}=\mathbf{1}_{\{j<i\}}/(i-1)$. The paper \cite{du2023sequential} proves quantitative propagation of chaos estimates, and \cite{wang2026quantitative} establishes a Gaussian fluctuation theorem. Our results do not cover this model, as the column sums diverge, so \eqref{eq:column.sum.assump} of Assumption \ref{assum.main} fails.
	
	\subsection{Organization of the paper}
	
	The rest of the paper is structured as follows. Section \ref{sec:assump.result}  fixes our notation and states our assumptions and main result, Theorem \ref{thm:fluctuation.clt}. Section \ref{sec:prelim} introduces further notation for the proofs along with preliminaries used throughout the paper. Sections \ref{sec:tightness}, \ref{sec:conv}, and \ref{sec:unique} establish tightness, convergence, and uniqueness, respectively. Section  \ref{sec:proof.main} combines these to prove Theorem \ref{thm:fluctuation.clt}. The appendices contain some auxiliary results, along with several  deferred proofs.
	\section{Assumptions and main result} \label{sec:assump.result}
	\subsection{Notations}
	
	We fix throughout the paper a dimension $d \in \N$ and a time horizon $T > 0$. For $n \in \N$, write $[n] = \{1, \dots, n\}$. Given any topological space $E$, let $\P(E)$ denote the space of Borel probability measures on $E$. For $\mu\in\P(E)$ and a Borel measurable function $\phi$ on $E$, we write $\langle \mu,\phi\rangle$ for the integral $\int_{E}\phi\,\dd \mu$ when it is well-defined. For a real-valued function $f$ on a set $E$, we write $\Vert f \Vert_\infty := \sup_{x \in E} |f(x)|$.
	
	We use standard multi-index notation. For $\bm{\alpha}=(\alpha_1,\dots,\alpha_d)\in\N^d$, let $|\bm{\alpha}|:=\alpha_1+\cdots+\alpha_d$ and
	\begin{align*}
	D^{\bm{\alpha}} f
	:=
	\partial_{x_1}^{\alpha_1}\cdots\partial_{x_d}^{\alpha_d} f.
	\end{align*}
	For $k\in\N$, let $C^k(\R^d)$ be the space of functions $f:\R^d\to\R$ such that $D^{\bm{\alpha}} f$ exists and is continuous for every multi-index $\bm{\alpha}$ with $0 \le |\bm{\alpha}|\le k$. We write $C_b^k(\R^d)\subset C^k(\R^d)$ for the subspace of functions for which all derivatives up to order $k$ are bounded, and we use the norm
	\begin{align*}
	\Vert f\Vert_{C_b^k}
	:=
	\sum_{0 \le |\bm{\alpha}|\le k}\Vert D^{\bm{\alpha}} f\Vert_\infty.
	\end{align*}
	The same derivative notation is used componentwise for vector-valued functions.
	
	\subsubsection{Distances between probability measures}
	For any $\mu,\nu\in \P(E)$, the relative entropy is defined as usual by
	\begin{align*}
	H(\nu \, |\, \mu) := \int_E \frac{\dd \nu}{\dd \mu } \log \frac{\dd \nu}{\dd \mu } \, \dd \mu \text{ if } \nu \ll \mu, \quad H(\nu \, |\, \mu) =\infty \ \text{if } \nu \nll \mu.
	\end{align*}
	For $E = \R^k$, the quadratic Wasserstein distance is defined by
	\begin{align*}
		\W_2(\mu, \nu) := \inf_\pi \left(\int_{\R^k \times \R^k} |x - y|^2 \pi (\dd x, \dd y)\right) ^{1/2},
	\end{align*}
	where the infimum is taken over all $\pi \in \P(\R^k \times \R^k)$ with marginals $\mu$ and $\nu$.  We denote the total variation norm by $\Vert\cdot\Vert_{\mathrm{TV}}$.

	\subsubsection{Operators}
	For each $t \in [0,T]$, $\nu \in \P(\R^d)$, and each test function $f \in C^\infty_c(\R^d)$, define the operator $\LL_{t,\nu}$ by
	\begin{align}\label{def.generator}
		\begin{split}
			\LL_{t,\nu} f(x)
			:=& \,
			\frac{\sigma^2}{2}\Delta f(x)
			+ b_0(t,x)\cdot \nabla f(x) 
			+ \int_{\R^d} b(t,x,x')\cdot \nabla f(x)\,\nu(\dd x') \\
			&
			+ \int_{\R^d} b(t,x',x)\cdot \nabla f(x')\,\mu_t(\dd x').
		\end{split}
	\end{align}

	\subsubsection{Sobolev spaces} \label{sec:sob.space}
	For each $k \in \N$, let $\cH^k$ be the (unweighted) Sobolev space on $\R^d$, defined as the completion of $C_c^\infty(\R^d)$ with respect to the norm
	\begin{align*}
		\|f\|_{\cH^k}^2
		:=
		\sum_{0 \le |\bm{\alpha}|\le k}\int_{\R^d} |D^{\bm{\alpha}} f(x)|^2\,\dd x.
	\end{align*}
	We denote by $\cH^{-k}$ its dual space, and by $\langle \cdot, \cdot \rangle_{\cH^{-k}, \cH^k}$ the dual pairing. We also introduce the weighted spaces. We set
	\begin{align*}
		\lambda_d = \lfloor d/2\rfloor + 1.
	\end{align*}
	Define the weight
	\begin{align} \label{eq:w.def}
		w(x)=1+|x|^{2\lambda_d},\qquad x\in\R^d.
	\end{align}
	For each $k\in\N$, define the weighted Sobolev space $\cH^{k}_w$ as the completion of $C^\infty_c(\R^d)$ with respect to the norm
	\begin{align*} 
		\|f\|_{\cH^{k}_w}^2
		:=
		\sum_{0 \le |\bm{\alpha}|\le k}\int_{\R^d} \frac{|D^{\bm{\alpha}} f(x)|^2}{w(x)}\,\dd x.
	\end{align*}
	We denote by $\cH^{-k}_w$ its dual space, and by $\langle \cdot, \cdot\rangle_{\cH^{-k}_w,\cH^{k}_w}$ the duality bracket between $\cH^{k}_w$ and $\cH^{-k}_w$.
	
	\subsection{Assumptions}
	In our main assumption, we make use of the following notation. For each $n \in \N$, let $\hxi$ be the $n \times n$ matrix with entries
	\begin{align}\label{hatxi.def}
		\hxi_{ij} := (n-1)\xi_{ij} - 1, \qquad i,j = 1, \dots, n.
	\end{align}
	This matrix captures the deviation of the interaction matrix $\xi $ from the mean field case, in which $\xi_{ij} = \mathbf{1}_{\{i \ne j\}}/(n-1)$ and $\hxi_{ij} = - \mathbf{1}_{\{i = j\}}$.
	\begin{assumption} \label{assum.main}
		\begin{enumerate}[(i)]
			\item 			\label{assum:mat} The matrix $\xi = (\xi_{ij})_{i,j = 1}^n$ has nonnegative entries, zero diagonal entries, row sums equal to one: 
			\begin{align} \label{eq:row.sum.assump}
				\sum_{j=1}^{n} \xi_{ij} =  1, \qquad \text{for all} 
				\quad i = 1, \dots, n,\qquad \tag{rows}
			\end{align}
			and bounded column sums: There exists $C < \infty$ independent of $n$ such that
			\begin{align} \label{eq:column.sum.assump}
				\max_{1 \le j \le n} \sum_{i=1}^{n} \xi_{ij} \le C, \quad \text{and} \quad \limsup_{n\to \infty} \frac{1}{n} \sum_{j=1}^{n} \Big(\sum_{i=1}^{n} \xi_{ij} \Big)^2 \le 1. \tag{columns}
			\end{align}
			As $n \to \infty$,  we  assume that
			\begin{align}
				\frac{1}{n^{3/2}}\Big(\sum_{i,j=1}^n|\hxi_{ij}|\Big)\max_{i,j\in[n]}\xi_{ij} \to 0 \label{assum:mat.3},
			\end{align} 
			as well as 
			\begin{align}
				\limsup_{n\to\infty}\sum_{i=1}^n\Big(\sum_{j=1}^n(\xi_{ij}^2+\xi_{ji}^2)\Big)^2 < \infty.
				\label{eq:xi_row_sq}
			\end{align}
			
			\item \label{assum.drift} Assume $ b_0 \in L^\infty([0,T]; C_b^{k+1}(\R^d;\R^d)) $, i.e.,
			\begin{align*}
				\max_{0 \le |\bm{\alpha}|\le k+1}\sup_{t\in[0,T]}\sup_{x\in\R^d}\big|D^{\bm{\alpha}}b_0(t,x)\big|<\infty,
			\end{align*}
			and assume that $ b \in L^\infty([0,T]; C_b^{k+1}(\R^d \times \R^d;\R^d)) $.
			
			\item \label{assum.init.transport}The initial distribution $\mu_0$ admits finite second moments, and the following transport inequality holds: There exists $0 \le \gamma_0 < \infty$ such that
			\begin{align} \label{eq:init.transport}
				\W_2^2(\nu, \mu_0) \le \gamma_0 H(\nu \, |\, \mu_0), \quad   \nu \in \P(\R^d).
			\end{align}
		\end{enumerate}
	\end{assumption}
	
	\begin{remark}\label{remark.sufficient.matrix}
		First, if $\xi$ is symmetric, \eqref{eq:column.sum.assump} follows from \eqref{eq:row.sum.assump}. Second, without assuming symmetry, \eqref{assum:mat.3} follows from the nonnegativity of $\xi$, \eqref{eq:row.sum.assump}, and
		\begin{align}
			\max_{i,j\in[n]} \xi_{ij}=o(n^{-1/2}),
			\label{eq:sufficient.matrix.max}
		\end{align}
		while \eqref{eq:xi_row_sq} follows from the nonnegativity of $\xi$, \eqref{eq:row.sum.assump}, the first assertion in \eqref{eq:column.sum.assump}, and \eqref{eq:sufficient.matrix.max}. In fact, under these assumptions, the left-hand side of \eqref{eq:xi_row_sq} tends to zero. See Appendix \ref{sec:proof.remark} for a proof. We explain in Example \ref{ex:LS} that the $n^{-1/2}$ threshold in \eqref{eq:sufficient.matrix.max} is sharp, as can be seen from a class of spatial interaction models.
	\end{remark}
	Concrete examples of interaction matrices satisfying Assumption \ref{assum.main}, including regular and Erd\H{o}s--R\'enyi graphs, are given in Subsection \ref{sec:examples}.

	\subsection{Main result}
	\begin{definition}
		\label{def:SPDE.limit}
		Let $k \ge \lambda_d+2$ be an integer. On a filtered probability space $(\Omega, \F, \FF = (\F_t)_{t \in [0, T]}, \PP)$, let $\zeta_0$ be an $\cH^{-k}$-valued $\F_0$-measurable random variable. A pair $ (\eta,W)\in C([0,T];\cH^{-k})\times C([0,T];\cH^{-k-1})$ is called a solution of the fluctuation SPDE with initial condition $\zeta_0$ if
		\begin{enumerate}[(i)]
\item \label{SPDE.def.integrability} $\eta\in L^2([0,T];\cH^{-(k-1)}_w)$ $\PP$-a.s. By Lemma \ref{lem:sob}\eqref{embed.k1.k2}, this implies that $\eta\in L^2([0,T];\cH^{-k}_w)$ $\PP$-a.s. Therefore, the drift term  in \eqref{eq: SPDE.integral} is well-defined.
			\item \label{SPDE.def.dym} $W$ is a continuous $\FF$-adapted $\cH^{-k-1}$-valued centered Gaussian process with covariance given, for all $f_1,f_2 \in \cH^{k+1}$, by
			\begin{align} \label{def.SPDE.cov}
				\E\big[	\langle W_s,f_1\rangle_{\cH^{-k-1},\cH^{k+1}}\,\langle W_t,f_2\rangle_{\cH^{-k-1},\cH^{k+1}}\big]
				=
				\sigma^2\int_0^{s\wedge t}\big\langle \mu_u,\,\nabla f_1\cdot\nabla f_2\big\rangle\,\dd u.
			\end{align}
			\item \label{SPDE.def.weak} for every $f\in C_c^\infty(\R^d)$ and every $t\in[0,T]$,
			\begin{align}  \label{eq: SPDE.integral}
				\langle \eta_t,f\rangle_{\cH^{-k},\cH^{k}}
				&=
				\langle \zeta_0,f\rangle_{\cH^{-k},\cH^{k}}
				+
				\int_0^t \big\langle \eta_s, \LL_{s,\mu_s} f\big\rangle_{\cH^{-k}_w,\cH^{k}_w}\,\dd s
				+
				\langle W_t,f\rangle_{\cH^{-k-1},\cH^{k+1}}.
			\end{align}
		\end{enumerate}
	\end{definition}

	We can now state the main result. 
	
	\begin{theorem}
		\label{thm:fluctuation.clt}
		Suppose Assumption \ref{assum.main} holds, with $k \ge \lambda_d+2$ in part \eqref{assum.drift} therein.  Then: 
		
		\begin{enumerate}[(i)]
			
			\item \label{main.thm.part.0} We have $\eta^n_0 \stackrel{d}{\to} \eta_0^\ast$ in $\cH^{-k}$, where $\eta_0^\ast$ is an $\cH^{-k}$-valued centered Gaussian random variable characterized by
			\begin{align*} 
				\E\big[
				\langle \eta_0^\ast,f_1\rangle_{\cH^{-k},\cH^k}
				\langle \eta_0^\ast,f_2\rangle_{\cH^{-k},\cH^k}
				\big]
				&=
				\langle \mu_0,f_1 f_2\rangle
				-
				\langle \mu_0,f_1\rangle \langle \mu_0,f_2\rangle,
				\quad
				f_1,f_2\in C_c^\infty(\R^d).
			\end{align*}

			\item \label{main.thm.part.1} There exists a solution pair $(\eta,W)$ of the fluctuation SPDE with initial condition $\eta_0$ in the sense of Definition \ref{def:SPDE.limit}, where $\eta_0\stackrel{d}{=}\eta_0^\ast$. Moreover, the fluctuation SPDE enjoys pathwise uniqueness: If $(\eta,W)$ and $(\widetilde{\eta},W)$ are two solutions on the same filtered probability space with the same initial condition $\eta_0$ and the same driving process $W$, then $\eta=\widetilde{\eta}$ $\PP$-a.s. in $C([0,T];\cH^{-k})$. Consequently, the first component of the solution pair is unique in law once the joint law of the initial condition and the driving process is fixed, i.e.,  if $(\eta^1,W^1)$ and $(\eta^2,W^2)$ are two solutions, possibly defined on different probability spaces, with initial conditions $\eta_0^1$ and $\eta_0^2$, and if $\Law(\eta_0^1,W^1)=\Law(\eta_0^2,W^2)$, then $\Law(\eta^1)=\Law(\eta^2)$.
			
			\item \label{main.thm.part.2} We have $\eta^n \stackrel{d}{\to} \eta$ in $C([0,T];\cH^{-k})$, where $\eta$ is the unique-in-law first component, in the sense of part \eqref{main.thm.part.1}, of any solution pair $(\eta,W)$ with initial condition $\eta_0$ such that $ \Law(\eta_0,W)=\Law(\eta_0^\ast)\otimes\Law(W^\ast)$,
				and where $W^\ast$ is a centered Gaussian process in $C([0,T];\cH^{-k-1})$ satisfying the covariance formula \eqref{def.SPDE.cov}.
		\end{enumerate}
	\end{theorem}

	\begin{remark} \label{rmk:unique.in.law}
		The uniqueness-in-law statement in Theorem \ref{thm:fluctuation.clt}\eqref{main.thm.part.1} follows from existence and pathwise uniqueness by a natural extension of the Yamada-Watanabe Theorem, see, e.g., \cite[Proposition 5.3.20 and Corollary 5.3.23]{karatzas2014brownian}. The same theorem also implies that $\eta$ is a strong solution of \eqref{eq: SPDE.integral} in the probabilistic sense: There exists a function
		\begin{align*}
			h:\cH^{-k}\times C([0,T];\cH^{-k-1}) \to C([0,T];\cH^{-k})
		\end{align*}
		which is measurable from $\B(\cH^{-k}\times C([0,T];\cH^{-k-1}))$ to $\B(C([0,T];\cH^{-k}))$, and, for every $t\in[0,T]$, measurable from $\B(\cH^{-k})\otimes \wh{\B}_t$ to $\sigma(\pi_s:s\in[0,t])$, such that
		\begin{align*}
			\eta = h(\eta_0,W),
		\end{align*}
		$\PP$-a.s. Here $\wh{\B}_t$ denotes the augmentation of $\sigma(\rho_s:s\in[0,t])$ by the null sets of the law of $W$, where $\rho_s(\omega)=\omega(s)$ for $\omega\in C([0,T];\cH^{-k-1})$ and $\pi_s(\omega)=\omega(s)$ for $\omega\in C([0,T];\cH^{-k})$.
	\end{remark}
	
	\subsection{Examples of interaction matrices} \label{sec:examples}
We now give several classes of interaction matrices that satisfy Assumption \ref{assum.main}\eqref{assum:mat}, and hence to which our Central Limit Theorem \ref{thm:fluctuation.clt} applies.

\begin{example}\label{ex:graphs}
	
		\begin{enumerate}[(a)]
			\item \label{rem.b}
			Assumption \ref{assum.main}\eqref{assum:mat} holds when
			$\xi_{ij} = \mathbf{1}_{\{i\sim j\}}/m_n$
			is the scaled adjacency matrix of a simple undirected $m_n$-regular graph and
			$\sqrt{n}/m_n \to 0$. 
			\item  \label{rem.c}
			Assumption \ref{assum.main}\eqref{assum:mat} also holds, in the \emph{quenched} sense, for the
			Erd\H{o}s--R\'enyi graph $G(n,p_n)$ when $\liminf_{n\to\infty} np_n/\sqrt{n}=\infty$, where
			$\xi_{ij}=\mathbf{1}_{\{i\sim j\}}/\text{deg}(i)$ and $\text{deg}(i)$ denotes the degree of
			vertex $i$. By the multiplicative Chernoff bound and the Borel-Cantelli lemma,
			$\max_i|\text{deg}(i)/np_n-1|\to0$ almost surely. On this event, $\min_i\text{deg}(i)\sim np_n$,
			where $a_n\sim b_n$ means $a_n/b_n\to1$, so the maximum entry satisfies
			$\max_{i,j}\xi_{ij}=1/\min_i\text{deg}(i)\sim(np_n)^{-1}=o(n^{-1/2})$, while the column sums
			converge to $1$ uniformly and \eqref{eq:column.sum.assump} holds. The matrix $\xi$ is also
			nonnegative with zero diagonal and satisfies \eqref{eq:row.sum.assump}. Consequently, by
			Remark \ref{remark.sufficient.matrix}, $\xi$ satisfies
			Assumption \ref{assum.main}\eqref{assum:mat} for all sufficiently large $n$, and
			Theorem \ref{thm:fluctuation.clt} holds in the quenched sense, almost surely with respect to
			the randomness of the graph.
		\end{enumerate}
	
\end{example}

\begin{example} \label{ex:LS}
	For the spatial interaction models studied in \cite{lucon2014mean,lucon2016transition} and in the absence of disorder variables,
	Theorem \ref{thm:fluctuation.clt} recovers their  Central Limit Theorem in the subcritical regime,
	and the supercritical regime shows that the $n^{-1/2}$ threshold in \eqref{eq:sufficient.matrix.max}
	is sharp. Let $\mathcal{X}:=\R/\Z$ be the
	circle and let $\Lambda_N:=\{-N,\dots,N\}$, with $-N$ and $N$ identified,
	so that $n:=|\Lambda_N|=2N$. Each particle $\theta^i$, $i\in\Lambda_N$, takes values in
	$\R^d$ and is located at the site $x_i:=\tfrac{i}{2N}\in\mathcal{X}$. The spatial weight is
	the singular, polynomially decaying kernel $\Psi(x,y):=d(x,y)^{-\alpha}\bm{1}_{\{x\neq y\}}$,
	where $d(\cdot,\cdot)$ is the distance on $\mathcal{X}$ and $\alpha\in[0,1)$. In the absence of disorder variables, the model considered in
	\cite{lucon2014mean,lucon2016transition} is
	\begin{align}
		\dd\theta_t^i = c(\theta_t^i)\,\dd t + \frac1n\sum_{j\in\Lambda_N}\Psi(x_i,x_j)\,
		\Gamma(\theta_t^i,\theta_t^j)\,\dd t + \dd B_t^i, \quad i\in\Lambda_N.
		\label{eq:LS.spatial.model}
	\end{align}
	While it does not directly fit into our setting, we are able to recover the Central Limit
	Theorem for \eqref{eq:LS.spatial.model} in the subcritical regime $\alpha<1/2$ by considering
	the system
	\begin{align}
		\dd X_t^i = c(X_t^i)\,\dd t + A_\alpha\sum_{j\in\Lambda_N}\xi_{ij}\,\Gamma(X_t^i,X_t^j)\,
		\dd t + \dd B_t^i, \quad i\in\Lambda_N,
		\label{eq:LS.fixed.coefficient.system}
	\end{align}
	where
	\begin{align*}
		\xi_{ij} := \frac{\Psi(x_i,x_j)}{\sum_{k\in\Lambda_N}\Psi(x_i,x_k)}, \quad
		A_\alpha := \int_{\mathcal{X}}\Psi(0,y)\,\dd y = \frac{2^\alpha}{1-\alpha}.
	\end{align*}
	The matrix
	$\xi$ is nonnegative and symmetric, has zero diagonal, and satisfies \eqref{eq:row.sum.assump}.
	Moreover, for sufficiently large $n$, it is straightforward to see from
	\cite[Lemma 3.4]{lucon2016transition} that $\max_{i,j\in\Lambda_N}\xi_{ij}\asymp n^{\alpha-1}$,
	so \eqref{eq:sufficient.matrix.max} holds as long as $\alpha<1/2$. Hence, by Remark \ref{remark.sufficient.matrix}, Assumption \ref{assum.main}\eqref{assum:mat} is satisfied.
	Thus, provided Assumption \ref{assum.main}\eqref{assum.drift}, \eqref{assum.init.transport} on the
	drifts and initial conditions hold, our Central Limit Theorem \ref{thm:fluctuation.clt} applies to
	$\mu_t^{n,X}:=n^{-1}\sum_{i\in\Lambda_N}\delta_{X_t^i}$.
	
	To transfer this Central Limit Theorem to
	$\mu_t^{n,\theta}:=n^{-1}\sum_{i\in\Lambda_N}\delta_{\theta_t^i}$, we couple
	\eqref{eq:LS.spatial.model} and \eqref{eq:LS.fixed.coefficient.system} with the same Brownian
	motions $B^i$ and initial conditions. The Lipschitz and boundedness assumptions on $c$ and
	$\Gamma$, together with \eqref{eq:row.sum.assump}, then give by Gronwall's inequality
	\begin{align}
		\max_{i\in\Lambda_N}\sup_{0\le t\le T}|\theta_t^i-X_t^i| \lesssim \max_{i\in\Lambda_N}
		\Big|\frac1n\sum_{j\in\Lambda_N}\Psi(x_i,x_j)-A_\alpha\Big| \asymp n^{\alpha-1},
		\label{eq:LS.Gronwall.bound}
	\end{align}
	where the last estimate is \cite[Lemma 3.4]{lucon2016transition}. Since $k\ge\lambda_d+2$, so that $k>d/2+1$, the Sobolev Embedding Theorem, see, e.g.,
	\cite[Theorem 4.12, Case A with $\Omega=\R^d$, $n=d$, $p=2$, $j=1$, and
	$m=k-1$]{adams2003sobolev}, yields $\cH^{k}\hookrightarrow C_b^1(\R^d)$. In particular, this implies
	\begin{align*}
		\sup_{0\le t\le T}\sqrt{n}\,\|\mu_t^{n,\theta}-\mu_t^{n,X}\|_{\cH^{-k}} \lesssim
		\sqrt{n}\,\max_{i\in\Lambda_N}\sup_{0\le t\le T}|\theta_t^i-X_t^i| \lesssim n^{\alpha-1/2}\to0.
	\end{align*}
	The Central Limit Theorem for $\mu^{n,X}$ therefore transfers to $\mu^{n,\theta}$, recovering
	the  Central Limit Theorem \cite[Theorem 2.7]{lucon2016transition} in the subcritical regime $\alpha < 1/2$.

	In fact, our assumptions on $c$ and $\Gamma$ are  weaker. We
	assume only $c,\Gamma\in C_b^{k+1}$ (Assumption \ref{assum.main}\eqref{assum.drift}) and allow
	them to depend on time, whereas \cite{lucon2016transition} requires bounded
	derivatives up to order $3d+12$, together with an $L^1$-integrability
	condition requiring $\Gamma$ and its derivatives up to order $2d+5$ to be integrable in one
	state variable, uniformly in the other. The assumptions on the initial law are not directly
	comparable. We require the
	quadratic transport inequality \eqref{eq:init.transport}, whereas \cite{lucon2016transition}
	requires the initial distribution to be absolutely continuous with density in $L^p(\R^d)$ for
	some $p>d$ and to have finite higher-order moments. Neither condition implies the other. The
	transport inequality \eqref{eq:init.transport} accommodates some singular laws such as Dirac masses, while their condition allows tails heavier than Gaussian.
	
	For the supercritical regime $\alpha>1/2$, \cite[Theorem 2.8]{lucon2016transition} shows that
	$n^{1-\alpha}(\mu_t^{n,\theta}-\mu_t)$ converges to a nontrivial deterministic limit, where
	$\mu_t$ is the mean field limit of \eqref{eq:LS.spatial.model} and, by \eqref{eq:LS.Gronwall.bound},
	also of \eqref{eq:LS.fixed.coefficient.system}.
	Therefore, 
	$\sqrt{n}\,(\mu_t^{n,\theta}-\mu_t)$ diverges.
	Since
	$\max_{i,j\in\Lambda_N}\xi_{ij}\asymp n^{\alpha-1}$, the condition \eqref{eq:sufficient.matrix.max}
	holds when $\alpha<1/2$ and fails when $\alpha>1/2$. 
	Thus, the $n^{-1/2}$
	threshold in \eqref{eq:sufficient.matrix.max} is sharp for this class of spatial interaction models.

\end{example}

	\section{Additional notations and preliminaries}\label{sec:prelim}
	\subsection{Additional notations for proofs} \subsubsection{Measures}\label{subsubsec:measures}
	For each $t\in[0,T]$, let $P_t\in\P((\R^d)^n)$ denote the law of $(X_t^1,\dots,X_t^n)$, where $(X^1,\dots,X^n)$ is a weak solution of \eqref{eq:X}. For any $m\in\N$ and any distinct indices $i_1,\dots,i_m\in[n]$, we write $P_t^{i_1\cdots i_m}\in\P((\R^d)^m)$ for the law of $(X_t^{i_1},\dots,X_t^{i_m})$. In particular, $P_t^i$, $P_t^{ij}$, and $P_t^{ij\ell}$ correspond to the laws of $X_t^i$, $(X_t^i,X_t^j)$, and $(X_t^i,X_t^j,X_t^\ell)$, respectively. Finally, for $\mu\in\P(\R^d)$ and $\ell\in\N$, we write $\mu^{\otimes \ell}\in\P((\R^d)^\ell)$ for the $\ell$-fold product measure.

	\subsubsection{Weighted empirical measures}To streamline our proofs, we adopt the following notations similar to those in  \cite[Section 1.3]{coppini2023central}.  
	Recall the definition of $\hxi$ given in \eqref{hatxi.def}. For $t\in[0,T]$, let $\wh{\eta}^n_t$ be the (rescaled) weighted empirical (signed) measure on $(\R^d)^2$, defined by
	\begin{align}\label{eq:nuhat}
		\wh{\eta}^n_t
		:=
		\frac1{n^{3/2}}\sum_{i,j=1}^{n}\hxi_{ij}\,
		\delta_{(X^{i}_{t},X^{j}_{t})}.
	\end{align}
	\subsubsection{Fourier transforms} \label{sec:Fourier}
	We use $\SS(\R^d)$ and $\SS'(\R^d)$ to denote the Schwartz space and the space of tempered distributions, respectively. For $\phi\in\SS(\R^d)$, we define the Fourier and inverse Fourier transforms by
	\begin{align} \label{eq:FT.function}
		\F[\phi](u)
		&=\int_{\R^d} e^{-2\pi i\langle u,x\rangle}\phi(x)\,\dd x,\quad
		\F^{-1}[\phi](x)
		=\int_{\R^d} e^{2\pi i\langle u,x\rangle}\phi(u)\,\dd u.
	\end{align}
	For any $\nu\in\SS'(\R^d)$, the Fourier and inverse Fourier transforms are defined by duality as
	\begin{align*}
		\langle \F[\nu],\phi\rangle
		&:=\langle \nu,\F[\phi]\rangle,\quad
		\langle \F^{-1}[\nu],\phi\rangle
		:=\langle \nu,\F^{-1}[\phi]\rangle,\quad \phi\in\SS(\R^d).
	\end{align*}
	We will also use the Fourier characterization of the $\cH^{-k}$-norm (see, for example, \cite[7.62]{adams2003sobolev}): For any finite (signed) measure $\nu$ and any integer $k\ge \lambda_d + 2$,
	\begin{align}\label{eq:Hminus.k.Fourier}
		\|\nu\|_{\cH^{-k}}^2
		\asymp
		\int_{\R^d}(1+|u|^2)^{-k}\,\big|\F[\nu](u)\big|^2\,\dd u.
	\end{align}
	Here $a\asymp b$ means that there exist constants $0 < c \le C < \infty$, depending only on $d$ and $k$, such that $cb\le a\le Cb$.
	\subsubsection{Bessel potentials, Bessel kernels, and Bessel functions}\label{subsubsec:bessel}
	Finally, we introduce the Bessel potentials. For any $s\in\R$, define $J^{s}:\SS'(\R^d)\to\SS'(\R^d)$ by
	\begin{align}\label{eq:Js.definition}
		J^{s}\nu
		:= \F^{-1}\big[(1+|u|^{2})^{s/2}\F[\nu]\big].
	\end{align}
	Under the Fourier convention \eqref{eq:FT.function}, we have $	\F[\Delta \phi](u)=-4\pi^2|u|^2\F[\phi](u) $ for $\phi \in \SS(\R^d)$.
	Hence, if $k = 2 \ell$ is an even integer, then
	\begin{align*} 
		J^k = (I - \Delta/(4\pi^2))^{\ell}
	\end{align*}
	on $\SS(\R^d) $, and therefore also on $\SS'(\R^d)$ by duality.

	For $k\in\N$, by \eqref{eq:Js.definition}, Plancherel Theorem, and the  equivalence between Sobolev and Bessel-potential norms, see \cite[Section~7.62]{adams2003sobolev}, the dual pairing between $\cH^{-k}$ and $\cH^k$ satisfies
	\begin{align}
		\langle u,v\rangle_{\cH^{-k},\cH^{k}}
		=
		\langle J^{-k}u,J^{k}v\rangle_{L^{2}},
		\quad u\in\cH^{-k},\ v\in\cH^{k}.
		\label{eq:Hk.duality.bessel}
	\end{align}

For the weighted spaces, we prove in Appendix \ref{sec:pf.hminus-kw} that, whenever $u\in\SS'(\R^d)$ satisfies $J^{-k}u\in L^2(w(x)\,\dd x)$, one has $u\in\cH_w^{-k}$ and
		\begin{align}\label{eq:Hminus.kw.characterization}
			\|u\|_{\cH_w^{-k}}
			\lesssim
			\|J^{-k}u\|_{L^2(w(x)\,\dd x)}.
	\end{align} 
	
For $s>0$, let $G_s$ denote the Bessel kernel defined by
		\begin{align}  \label{eq:Gs.definition}
			G_s:=\F^{-1}\big[(1+|u|^{2})^{-s/2}\big],
		\end{align}
		so that, for every $f\in\SS(\R^d)$, one has $J^{-s}f=G_s*f$. Moreover, by \cite[equations (2.7), (2.10)]{aronszajn1961theory}, $G_s$ admits the radial representation
	\begin{align} \label{eq:bessel:radial.form}
		G_s(x)=c_s\,|x|^{\alpha}K_{\alpha}(|x|), \quad \alpha=\frac{s-d}{2},
	\end{align}
	where $K_\alpha$ is the modified Bessel function of the second kind, defined in \cite[equation (3.1)]{aronszajn1961theory}, and $c_s>0$ is a constant. Moreover,
	using \cite[equation (3.7)]{aronszajn1961theory} with $\nu=-\alpha$ and $m=1$, and the symmetry
	\begin{align} \label{eq:K.symmetry}
		K_{-\nu}=K_{\nu}
	\end{align}
	from \cite[equation (3.3)]{aronszajn1961theory}, we obtain the derivative identity 
	\begin{align}\label{eq:bessel:derivative.identity}
		\frac{\dd}{\dd r}\big(r^{\alpha}K_{\alpha}(r)\big)=-\,r^{\alpha}K_{\alpha-1}(r).
	\end{align}
	
	Finally, by \cite[equation (3.4)]{aronszajn1961theory}, as $r\downarrow 0$, we have the small-argument asymptotics
	\begin{align}
		K_{\nu}(r)
		&\sim 2^{\nu-1}\Gamma(\nu)\,r^{-\nu},
		\quad \nu>0, \label{eq:bessel:Knu.small.r}\\
		K_{0}(r)
		&\sim \log(1/r). \label{eq:bessel:K0.small.r}
	\end{align}
	Here $\sim$ denotes asymptotic equivalence, that is, the ratio of the two sides tends to $1$. Moreover, by \cite[equation (3.5)]{aronszajn1961theory}, as $r\to\infty$, we have the large-argument asymptotic
	\begin{align}
		K_\nu(r)\sim \sqrt{\frac{\pi}{2r}}\,e^{-r}, \quad \nu \in \R. \label{eq:bessel:Knu.large.r}
	\end{align}

	\subsection{Preliminaries}
	The proof of the following proposition is deferred to Appendix  \ref{sec:proof-lem-preliminaries}. 
	\begin{proposition}
		\label{lem:preliminaries}
		Suppose Assumption \ref{assum.main} holds with $k \ge 0$. Then:
		\begin{enumerate}[(i)]
			\item\label{lem:preliminaries:i} \emph{(Well-posedness)} For every $n\in\N$, the system \eqref{eq:X} admits a unique strong solution on $[0,T]$. Moreover, the McKean-Vlasov equation \eqref{eq:MKV} admits a unique strong solution $(Y_t)_{t\in[0,T]}$. In particular, the measure flow $(\mu_t)_{t\in[0,T]}$ is uniquely determined.

				\item\label{lem:preliminaries:iv} \emph{(Uniform moment bounds)}
				For every $p\in[1,\infty)$, we have
				\begin{align} \label{eq:uniform.moment.bound.X.Y}
					\sup_{n\in\N}\max_{i\in[n]} \,
					\E\Big[\sup_{t\in[0,T]} |X_t^i|^p\Big]<\infty,
					\quad \text{and} \quad
					\E\Big[\sup_{t\in[0,T]} |Y_t|^p\Big]<\infty.
				\end{align}
				In particular, taking $p=2\lambda_d$,
				\begin{align} \label{eq:w.momet.statement}
					\sup_{n\in\N}\max_{i\in[n]} \,
					\E\Big[\sup_{t\in[0,T]}w(X_t^i)\Big]<\infty,
					\quad \text{and} \quad
					\E\Big[\sup_{t\in[0,T]}w(Y_t)\Big]<\infty.
				\end{align}

			\item\label{lem:preliminaries:ii} \emph{(Convergence of empirical measures)} As $n\to\infty$,
			\begin{align*}
				\sup_{t\in[0,T]}\E\bigg[\W_2^2\Big(\frac{1}{n}\sum_{i=1}^n\delta_{X_t^i},\mu_t\Big)\bigg]\to 0.
			\end{align*} 
			\item\label{lem:preliminaries:iii} \emph{(Poincar\'e inequality)} There exists $C< \infty$, independent of $t\in[0,T]$ and $n \in \N$, such that for any $g\in W^{1,2}(P_t)$,
			\begin{align*}
				\Var_{P_t}(g)
				:=
				\int_{(\R^d)^n} g^2\,\dd P_t
				-
				\Big(\int_{(\R^d)^n} g\,\dd P_t\Big)^2
				\le
				C\int_{(\R^d)^n} |\nabla g|^2\,\dd P_t,
			\end{align*}
			where $W^{1,2}(P_t)$ denotes the set of $g\in L^2(P_t)$ whose weak gradient $\nabla g$ belongs to $L^2(P_t)$.

			\item\label{lem:preliminaries:v} \emph{(Transport inequality for $\mu_t$)} There exists $\gamma_T<\infty$ such that for all $t\in[0,T]$ and all $\nu\in\P(\R^d)$,
			\begin{align}\label{eq:mu_t_transport}
				\W_2^2(\nu,\mu_t)\le \gamma_T\,H(\nu\,|\,\mu_t).
			\end{align}
		\end{enumerate}
	\end{proposition}

		The next preliminary result concerns quantitative propagation of chaos for the particle system \eqref{eq:X}. It controls the convergence toward i.i.d. copies of the McKean-Vlasov equation \eqref{eq:MKV} in relative entropy, in a maximum sense for the $3$-particle marginals and in an average sense for the $2$-particle marginals. The proof of this lemma follows  from the results in \cite{lacker2024quantitative} and is  deferred to Appendix \ref{sec.poc.pf}.

		\begin{lemma}[Entropy bounds]\label{lem:max.ent.POC}
			Suppose Assumption \ref{assum.main} holds with $k \ge 0$. Then, for each $t\in[0,T]$:
			\begin{enumerate}[(i)]
				\item\label{lem:max.ent.POC:i} \emph{(Maximum three-particle entropy.)}
				\begin{align*} 
					\max_{\substack{i,j,k\in[n]\\ i\neq j,\ i\neq k,\ j\neq k}}
					H(P_t^{ijk}\,|\, \mu_t^{\otimes 3})
					\lesssim
					\max_{i, j \in [n]}\xi_{ij}^2.
				\end{align*}
				\item\label{lem:max.ent.POC:ii} \emph{(Average two-particle entropy.)}
				\begin{align*}
					\frac{1}{n(n-1)}\sum_{\substack{i, j \in [n] \\ i \neq j}} H(P_t^{ij}\,|\, \mu_t^{\otimes 2})
					\lesssim
					\frac{1}{n}\sum_{i=1}^n\Big(\sum_{j=1}^n(\xi_{ij}^2+\xi_{ji}^2)\Big)^2.
				\end{align*}
			\end{enumerate}
		\end{lemma}
		
		\begin{remark}\label{rem:data.processing}
			We will  use the data processing inequality for relative entropy: For any probability measures $\nu,\nu'$ on a common measurable space and any measurable map $f$:
			\begin{align*}
				H(\nu\circ f^{-1}\,|\, \nu'\circ f^{-1}) \le H(\nu\,|\, \nu').
			\end{align*}
			Applied with coordinate projections, this yields for $t\in[0,T]$,
			\begin{align}
				H(P_t^{ij}\,|\, \mu_t^{\otimes 2}) &\le H(P_t^{ijk}\,|\, \mu_t^{\otimes 3}),
				\quad i\neq j,\ i\neq k,\ j\neq k,
				\label{eq:DPI.3to2}\\
				H(P_t^{i}\,|\, \mu_t) &\le H(P_t^{ij}\,|\, \mu_t^{\otimes 2}),
				\quad i\neq j.
				\label{eq:DPI.2to1}
			\end{align}
			Averaging \eqref{eq:DPI.2to1} over $j\neq i$ and then over $i$ yields, for $n\ge 2$,
			\begin{align}
				\frac{1}{n}\sum_{i=1}^n H(P_t^{i}\,|\, \mu_t)
				\le
				\frac{1}{n(n-1)}\sum_{\substack{i, j \in [n] \\ i \neq j}} H(P_t^{ij}\,|\, \mu_t^{\otimes 2}) \lesssim
				\frac{1}{n}\sum_{i=1}^n\Big(\sum_{j=1}^n(\xi_{ij}^2+\xi_{ji}^2)\Big)^2.
				\label{eq:DPI.avg.1to0}
			\end{align}
			Similarly, combining \eqref{eq:DPI.3to2} with Lemma~\ref{lem:max.ent.POC}(\ref{lem:max.ent.POC:i}) yields, when $n\ge 3$,
			\begin{align}
				\max_{\substack{i, j \in [n] \\ i \neq j}} H(P_t^{ij}\,|\, \mu_t^{\otimes 2})
				\lesssim
				\max_{i, j \in [n]}\xi_{ij}^2.
				\label{eq:DPI.maxpair}
			\end{align}
		\end{remark}
		
		The final preliminary result we need is the weak semimartingale decomposition of $\eta^n$.
		\begin{lemma}\label{lem:eta.dym}
			Let $f\in C_b^2(\R^d)$. For $t \in [0,T]$,
			\begin{align}\label{eq:eta.semimartingale}
				\begin{split}
					\langle \eta^n_t, f \rangle
					&= \langle \eta^n_0, f \rangle
					+ \int_0^t \big\langle \eta^n_s, \LL_{s,\mu^n_s} f\big\rangle \,\dd s 
					+ \frac{n}{n-1} \int_0^t \big\langle \wh{\eta}^n_s(\dd x, \dd y), b(s,x,y)\cdot \nabla f(x)\big\rangle \,\dd s \\
					&\quad
					+\frac{\sqrt{n}}{n-1}\int_0^t \big\langle \mu_s^n(\dd x)\mu_s^n(\dd y), b(s,x,y)\cdot \nabla f(x)\big\rangle \,\dd s
					+ \frac{\sigma}{\sqrt{n}} \sum_{i=1}^{n} \int_0^t \nabla f(X^i_s)\cdot \dd B^i_s .
				\end{split}
			\end{align}
		\end{lemma}

		\begin{proof}
			Fix $i\in[n]$. It\^{o}'s formula yields
			\begin{align*}
				\dd \langle \mu^n_t,f\rangle
				&=
				\Big\langle \mu^n_t,
				\frac{\sigma^2}{2}\Delta f
				+ b_0(t,\cdot)\cdot \nabla f
				\Big\rangle \, \dd t \nonumber\\
				&\quad
				+ \frac1n\sum_{i,j=1}^n \xi_{ij}\, b(t,X^i_t,X^j_t)\cdot \nabla f(X^i_t)\,\dd t
				+ \frac{\sigma}{n}\sum_{i=1}^n \nabla f(X^i_t)\cdot \dd B^i_t .
			\end{align*}
			Using the definition of $\wh{\eta}^n_t$ in \eqref{eq:nuhat} together with $ 	\xi_{ij} = (\hxi_{ij}+1)/(n-1)$,
			we rewrite
			\begin{align*}
				\frac1n\sum_{i,j=1}^n \xi_{ij}\, b(t,X^i_t,X^j_t)\cdot \nabla f(X^i_t)
				&=
				\frac{n}{n-1}
				\Big\langle \mu^n_t(\dd x),
				\big\langle \mu^n_t(\dd y), b(t,x,y)\big\rangle \cdot \nabla f(x)
				\Big\rangle \nonumber\\
				&\quad
				+ \frac{\sqrt{n}}{n-1}
				\big\langle \wh{\eta}^n_t,
				b(t,x,y)\cdot \nabla f(x)
				\big\rangle .
			\end{align*}
			Therefore,
			\begin{align}
				\begin{split}
				\dd \langle \mu^n_t,f\rangle
				&=
				\Big\langle \mu^n_t,
				\frac{\sigma^2}{2}\Delta f
				+ b_0(t,\cdot)\cdot \nabla f
				\Big\rangle \, \dd t
				+
				\frac{n}{n-1}
				\Big\langle \mu^n_t(\dd x),
				\big\langle \mu^n_t(\dd y), b(t,x,y)\big\rangle \cdot \nabla f(x)
				\Big\rangle \, \dd t \\
				&\quad
				+ \frac{\sqrt{n}}{n-1}
				\big\langle \wh{\eta}^n_t,
				b(t,x,y)\cdot \nabla f(x)
				\big\rangle \, \dd t
				+ \frac{\sigma}{n}\sum_{i=1}^n \nabla f(X^i_t)\cdot \dd B^i_t .
				\end{split}
				\label{eq:mun.dym}
			\end{align}
			Next, let $(Y_t)_{t\in[0,T]}$ solve \eqref{eq:MKV} (recall Proposition \ref{lem:preliminaries}\eqref{lem:preliminaries:i}) and let $\mu_t$ denote the law of $Y_t$. Applying It\^{o}'s formula to $f(Y_t)$ and taking expectations yields
			\begin{align}\label{eq:mu.dym}
				\dd \langle \mu_t,f\rangle
				&=
				\Big\langle \mu_t,
				\frac{\sigma^2}{2}\Delta f
				+ b_0(t,\cdot)\cdot \nabla f
				\Big\rangle \, \dd t 
				+
				\Big\langle \mu_t(\dd x),
				\big\langle \mu_t(\dd y), b(t,x,y)\big\rangle \cdot \nabla f(x)
				\Big\rangle \, \dd t .
			\end{align}
			Subtracting \eqref{eq:mu.dym} from \eqref{eq:mun.dym}, multiplying by $\sqrt{n}$, and using $\eta_t^n=\sqrt{n}(\mu_t^n-\mu_t)$ gives
			\begin{align*}
				\dd\langle \eta_t^n, f\rangle
				=&
				\Big\langle \eta^n_t,
				\frac{\sigma^2}{2}\Delta f + b_0(t,\cdot)\cdot \nabla f
				\Big\rangle \,  \dd t +  \sqrt{n} \Big\langle \mu^n_t(\dd x),
				\big\langle \mu^n_t(\dd y), b(t,x,y)\big\rangle\cdot \nabla f(x)
				\Big\rangle \, \dd t\nonumber\\
				&
				-\sqrt{n}
				\Big\langle \mu_t(\dd x),
				\big\langle \mu_t(\dd y), b(t,x,y)\big\rangle\cdot \nabla f(x)
				\Big\rangle
				\,  \dd t \nonumber\\
				&
				+\frac{\sqrt{n}}{n-1}
				\Big\langle \mu^n_t(\dd x),
				\big\langle \mu^n_t(\dd y), b(t,x,y)\big\rangle\cdot \nabla f(x)
				\Big\rangle \,  \dd t \\
				&
				+\frac{n}{n-1}\big\langle \wh{\eta}^n_t, b(t,x,y)\cdot \nabla f(x)\big\rangle \,  \dd t
				+\frac{\sigma}{\sqrt{n}}\sum_{i=1}^n \nabla f(X^i_t)\cdot \dd B^i_t .
			\end{align*}
			Using
			\begin{align*}
				&\sqrt{n}\bigg[
				\Big\langle \mu^n_t(\dd x),
				\big\langle \mu^n_t(\dd y), b(t,x,y)\big\rangle\cdot \nabla f(x)
				\Big\rangle
				-\Big\langle \mu_t(\dd x),
				\big\langle \mu_t(\dd y), b(t,x,y)\big\rangle\cdot \nabla f(x)
				\Big\rangle
				\bigg] \\
				=\,&
				\bigg\langle \eta_t^n, \int_{\R^d} b(t,\cdot,y)\cdot\nabla f(\cdot)\,\mu_t^n(\dd y)\bigg\rangle 
				+\bigg\langle \eta_t^n,  \int_{\R^d} b(t,y,\cdot)\cdot\nabla f(y)\,\mu_t(\dd y)\bigg\rangle,
			\end{align*}
			we obtain
			\begin{align*}
				\dd\langle \eta_t^n, f\rangle
				&=
				\big\langle \eta_t^n,\LL_{t,\mu_t^n}f\big\rangle \,\dd t
				+\frac{\sqrt{n}}{n-1}
				\Big\langle \mu^n_t(\dd x),
				\big\langle \mu^n_t(\dd y), b(t,x,y)\big\rangle\cdot \nabla f(x)
				\Big\rangle \,\dd t \\
				&\quad
				+\frac{n}{n-1}\big\langle \wh{\eta}^n_t, b(t,x,y)\cdot \nabla f(x)\big\rangle \,\dd t
				+\frac{\sigma}{\sqrt{n}}\sum_{i=1}^n \nabla f(X^i_t)\cdot \dd B^i_t .
			\end{align*}
			Integrating in time yields the claimed decomposition.
		\end{proof}

		\section{Tightness}\label{sec:tightness}
		Our proof of the main Theorem \ref{thm:fluctuation.clt} begins with establishing the tightness of $(\eta^n)$.  As a preparation, we prove the following weighted second moment bound. (Recall the definition of $\hxi$ given in \eqref{hatxi.def}.)
		\begin{lemma}\label{lem:cancellation}
			Suppose Assumption \ref{assum.main} holds with $k \ge 0$. Then, for each $t \in [0,T]$, and each bounded measurable $G : \R^d \times \R^d\to \R^d$,
			\begin{align}
				\sum_{i=1}^n \E\Big[\Big|\sum_{j=1}^n \hxi_{ij}\,G(X_t^i,X_t^j)\Big|^2\Big]
				&\lesssim \Vert G\Vert_\infty^2\Big( \sum_{i,j=1}^n \hxi_{ij}^2
				+ \sum_{i=1}^n\Big(\sum_{j=1}^n|\hxi_{ij}|\Big)^2\max_{i,j\in[n]}\xi_{ij}\Big), \label{lem:cancellation.row}
			\end{align}
			and
			\begin{align}
				\begin{split}
					\sum_{i=1}^n \E\Big[\Big|\sum_{j=1}^n \hxi_{ji}\,G(X_t^j,X_t^i)\Big|^2\Big]
					\lesssim \Vert G\Vert_\infty^2&\Big(  \sum_{i,j=1}^n \hxi_{ij}^2 + \sum_{i=1}^n\Big(\sum_{j=1}^n|\hxi_{ji}|\Big)^2\max_{i,j\in[n]}\xi_{ij} \\
					&
					+ \sum_{i=1}^n\Big(\sum_{j=1}^n\hxi_{ji}\Big)^2\Big).
				\end{split}
				\label{lem:cancellation.col}
			\end{align}
			Also, for each $g_t \in C_b^1(\R^d\times\R^d)$ satisfying $ 		\E[g_t(Y_t^1,Y_t^2)] = 0$,
			where $Y^1$ and $Y^2$ are independent copies of the  McKean-Vlasov solution from Proposition \ref{lem:preliminaries}\eqref{lem:preliminaries:i}, we have
			\begin{align} \label{lem:cancellation.unweighted}
				\E \bigg[\Big( \sum_{i,j=1}^{n} g_t(X_t^i,X_t^j)\Big)^2 \bigg]
				\lesssim
				n^2 \|g_t\|_\infty^2 + n^3 \|\nabla g_t\|_\infty^2.
			\end{align}
		\end{lemma}
		
		\begin{proof}
			We first prove \eqref{lem:cancellation.row}. Let $Y^1,\dots,Y^n$ be i.i.d. copies of the solution $Y$ to the McKean-Vlasov equation \eqref{eq:MKV} (recall Proposition \ref{lem:preliminaries}\eqref{lem:preliminaries:i}). For each $i \in [n]$, we have
			\begin{align} \label{pf.variance.decomp}
				\begin{split}
					\E\bigg[\Big|\sum_{j=1}^n \hxi_{ij}\,G(Y_t^i,Y_t^j)\,\Big|^2\, \bigg | \, Y_t^i\bigg]
					=&\, \Var\Big(\sum_{j\neq i}\hxi_{ij} G(Y_t^i,Y_t^j)\,\Big|\,Y_t^i\Big) \\
					&+\bigg|\E\Big[\sum_{j=1}^n \hxi_{ij}\, G(Y_t^i,Y_t^j)\,\Big|\,Y_t^i\Big]\bigg|^2.
				\end{split}
			\end{align}
			For the conditional variance term, note that the random vectors $\{ G(Y_t^i,Y_t^j)\}_{j \neq i}$ are conditionally i.i.d. given $Y_t^i$. Therefore
			\begin{align} \label{pf.cancellation.variance}
				\Var\Big(\sum_{j\neq i}\hxi_{ij} G(Y_t^i,Y_t^j)\,\Big|\,Y_t^i\Big)
				&=
				\sum_{j\neq i}\hxi_{ij}^2\,\Var\big( G(Y_t^i,Y_t^j)\,\big|\,Y_t^i\big)
				\le
				\Vert  G\Vert_\infty^2\sum_{j\neq i}\hxi_{ij}^2.
			\end{align}
			For the conditional mean term, from Assumption \ref{assum.main}\eqref{assum:mat} and the definition of $\hxi$ given in \eqref{hatxi.def}, we deduce that $\hxi_{ii} = -1$ and $\sum_{j \neq i} \hxi_{ij} = 0$, which implies
			\begin{align} \label{pf.cancellation.moment}
				\begin{split}
					\bigg|\E\Big[\sum_{j=1}^n \hxi_{ij} G(Y_t^i,Y_t^j)\,\Big|\,Y_t^i\Big]\bigg|
					&=\Big|- G(Y_t^i,Y_t^i)+\sum_{j\ne i}\hxi_{ij}\,\E[ G(Y_t^i,Y_t^j)\,|\, Y_t^i]\Big| \\
					&=  |G(Y^i_t, Y^i_t) | 
					\le \| G\|_\infty.
				\end{split}
			\end{align}
			Putting \eqref{pf.cancellation.variance} and \eqref{pf.cancellation.moment} back into \eqref{pf.variance.decomp}, taking expectations, and summing over $i$ yields
			\begin{align} \label{Ybound}
				\sum_{i=1}^n \E\Big[\Big|\sum_{j=1}^n \hxi_{ij}\, G(Y_t^i,Y_t^j)\Big|^2\Big]
				&\le  \| G\|_\infty^2\Big(n+\sum_{i=1}^n \sum_{j\neq i}\hxi_{ij}^2\Big)
				= \| G\|_\infty^2\sum_{i,j=1}^n \hxi_{ij}^2.
			\end{align}
			
			Next,  expanding the squares and applying the triangle inequality, we get for each $i \in [n]$,
			\begin{align}
				\begin{split}
				&\bigg|\E\Big[\Big|\sum_{j=1}^n \hxi_{ij}\, G(X_t^i,X_t^j)\Big|^2\Big]
				-\E\Big[\Big|\sum_{j=1}^n \hxi_{ij}\, G(Y_t^i,Y_t^j)\Big|^2\Big]\bigg|
				\\
				&=
				\bigg|\sum_{j,k=1}^n \hxi_{ij}\hxi_{ik}
				\Big(\E\big[ G(X^i_t,X^j_t)\cdot G(X^i_t,X^k_t)\big]
				-\E\big[ G(Y^i_t,Y^j_t)\cdot G(Y^i_t,Y^k_t)\big]\Big)\bigg|
				\\
				&\le
				\sum_{j,k=1}^n |\hxi_{ij}|\,|\hxi_{ik}|\,
				\Big|\E\big[ G(X^i_t,X^j_t)\cdot G(X^i_t,X^k_t)\big]
				-\E\big[ G(Y^i_t,Y^j_t)\cdot G(Y^i_t,Y^k_t)\big]\Big|.
				\end{split}
				\label{eq:hatxi.square.triangle}
			\end{align}
			
			Recalling the notations in Section~\ref{subsubsec:measures}, if the indices $i,j,k$ are pairwise distinct, then $\Law(Y_t^i,Y_t^j,Y_t^k)=\mu_t^{\otimes 3}$. By Pinsker's inequality,
			\begin{align*}
				\Big|\E\big[ G(X^i_t,X^j_t) \cdot  G(X^i_t,X^k_t)\big]
				-\E\big[ G(Y^i_t,Y^j_t) \cdot  G(Y^i_t,Y^k_t)\big]\Big|
				&\lesssim
				\|G\|_\infty^2\,\sqrt{H(P_t^{ijk}\,|\, \mu_t^{\otimes 3})}.
			\end{align*}
			If some indices coincide, we apply the same argument to the relevant $2$-particle or $1$-particle marginal. By the data processing inequalities in Remark~\ref{rem:data.processing}, namely \eqref{eq:DPI.3to2} and \eqref{eq:DPI.2to1}, the corresponding lower-order entropy is bounded by a $3$-particle entropy. Therefore, by Lemma~\ref{lem:max.ent.POC}(\ref{lem:max.ent.POC:i}),
			\begin{align}
				\Big|\E\big[ G(X^i_t,X^j_t) \cdot  G(X^i_t,X^k_t)\big]
				-\E\big[ G(Y^i_t,Y^j_t) \cdot  G(Y^i_t,Y^k_t)\big]\Big|
				&\lesssim
				\|G\|_\infty^2\,\max_{i,j\in[n]}\xi_{ij}.
				\label{eq:hatxi.term.bound}
			\end{align}
			
			Plugging \eqref{eq:hatxi.term.bound} into \eqref{eq:hatxi.square.triangle} gives
			\begin{align*}
				\bigg|\E\Big[\Big|\sum_{j=1}^n \hxi_{ij}\, G(X_t^i,X_t^j)\Big|^2\Big]
				-\E\Big[\Big|\sum_{j=1}^n \hxi_{ij}\, G(Y_t^i,Y_t^j)\Big|^2\Big]\bigg|
				&\lesssim
				\|G\|_\infty^2\,\max_{p,q\in[n]}\xi_{pq}\,
				\Big(\sum_{j=1}^n|\hxi_{ij}|\Big)^2.
			\end{align*}
			Summing over $i$ and combining with \eqref{Ybound} proves \eqref{lem:cancellation.row}.
			
			The proof of \eqref{lem:cancellation.col} is similar. One replaces $\sum_{j=1}^n \hxi_{ij}G(X_t^i,X_t^j)$ by $\sum_{j=1}^n \hxi_{ji}G(X_t^j,X_t^i)$ throughout. The conditional variance term is handled in the same way: For each $ i \in [n]$,
			\begin{align}
				\Var\Big(\sum_{j\neq i}\hxi_{ji}\,G(Y_t^j,Y_t^i)\,\Big|\,Y_t^i\Big)
				&=
				\sum_{j\neq i}\hxi_{ji}^2\,\Var\big(G(Y_t^j,Y_t^i)\,\big|\,Y_t^i\big)
				\le
				\|G\|_\infty^2\sum_{j\neq i}\hxi_{ji}^2.
				\label{eq:cancellation.col.var}
			\end{align}
			The comparison between $(X^i_t)_{i \in [n]}$ and $(Y^i_t)_{i \in [n]}$ is also unchanged after swapping the indices $i$ and $j$: For each $ i \in [n]$,
			\begin{align}
				\bigg|\E\Big[\Big|\sum_{j=1}^n \hxi_{ji}\,G(X_t^j,X_t^i)\Big|^2\Big]
				-
				\E\Big[\Big|\sum_{j=1}^n \hxi_{ji}\,G(Y_t^j,Y_t^i)\Big|^2\Big]\bigg|
				&\lesssim
				\|G\|_\infty^2\,\max_{p,q\in[n]}\xi_{pq}\,
				\Big(\sum_{j=1}^n|\hxi_{ji}|\Big)^2.
				\label{eq:cancellation.col.comp}
			\end{align}
			The only place where the argument differs is the conditional mean term. Specifically, for each $i$, by the independence of $(Y_t^i)_{i\in[n]}$,
			\begin{align*}
				\E\bigg[\sum_{j=1}^n \hxi_{ji}G(Y_t^j,Y_t^i)\,\Big|\,Y_t^i\bigg]
				&=
				-\,G(Y_t^i,Y_t^i)
				+
				\Big(\sum_{j\ne i}\hxi_{ji}\Big)\int_{\R^d} G(y,Y_t^i)\,\mu_t(\dd y),
			\end{align*}
			where we used $\hxi_{ii}=-1$. Therefore,
			\begin{align*}
				\sum_{i=1}^n
				\bigg|\E\bigg[\sum_{j=1}^n \hxi_{ji}G(Y_t^j,Y_t^i)\,\bigg|\,Y_t^i\bigg]\bigg|^2
				&\lesssim
				\|G\|_\infty^2\bigg(
				n+\sum_{i=1}^n\Big(\sum_{j=1}^n\hxi_{ji}\Big)^2
				\bigg).
			\end{align*}
			Combining this with \eqref{eq:cancellation.col.var} yields
			\begin{align*}
				\sum_{i=1}^n \E\Big[\Big|\sum_{j=1}^n \hxi_{ji}\,G(Y_t^j,Y_t^i)\Big|^2\Big]
				&\lesssim
				\|G\|_\infty^2\bigg(
				\sum_{i,j=1}^n \hxi_{ij}^2
				+
				\sum_{i=1}^n\Big(\sum_{j=1}^n\hxi_{ji}\Big)^2
				\bigg),
			\end{align*}
			since $\hxi_{ii}=-1$ for every $i$.
			Together with \eqref{eq:cancellation.col.comp}, this proves \eqref{lem:cancellation.col}.
			
			\medskip
			
			In the proof of \eqref{lem:cancellation.unweighted},  we use $g:=g_t$ for simplicity of notation. Let
			\begin{align*}
				H(x_1,\dots,x_n):=\sum_{i,j=1}^n g(x_i,x_j).
			\end{align*}
			Then,
			\begin{align*}
				\E \bigg[\Big( \sum_{i,j=1}^{n} g(X_t^i,X_t^j)\Big)^2  \bigg]
				=
				\Var\big(H(X_t^1,\dots,X_t^n)\big)
				+
				\big|\E[H(X_t^1,\dots,X_t^n)]\big|^2.
			\end{align*}
			
			We first estimate the variance term. By the Poincaré inequality in Proposition
			\ref{lem:preliminaries}\eqref{lem:preliminaries:iii},
			\begin{align*}
				\Var\big(H(X_t^1,\dots,X_t^n)\big)
				&\lesssim
				\sum_{k=1}^n \E\big[|\nabla_{x_k}H(X_t^1,\dots,X_t^n)|^2\big].
			\end{align*}
			For each $k\in[n]$, $ \nabla_{x_k}H(x_1,\dots,x_n)
			=
			\sum_{i=1}^n \nabla_1 g(x_k,x_i)
			+
			\sum_{j=1}^n \nabla_2 g(x_j,x_k)$.
			Hence,
			\begin{align*}
				|\nabla_{x_k}H(x_1,\dots,x_n)|^2
				&\lesssim
				\Big(\sum_{i=1}^n |\nabla_1 g(x_k,x_i)|\Big)^2
				+
				\Big(\sum_{j=1}^n |\nabla_2 g(x_j,x_k)|\Big)^2
				\lesssim
				n^2\|\nabla g\|_\infty^2.
			\end{align*}
			Summing over $k$ yields
			\begin{align}
				\Var\big(H(X_t^1,\dots,X_t^n)\big)
				\lesssim
				n^3\|\nabla g\|_\infty^2.
				\label{eq:unweighted.var.bound}
			\end{align}
			
			Let $Y^1,\dots,Y^n$ be i.i.d. copies of the solution $Y$ to the McKean-Vlasov equation \eqref{eq:MKV}. Since $\E[g(Y_t^1,Y_t^2)]=0$, we have
			\begin{align}
				\bigg| \E\Big[\sum_{i,j=1}^n g(Y_t^i,Y_t^j)\Big] \bigg|
				&=
				\Big| \sum_{i\neq j}\E[g(Y_t^i,Y_t^j)]
				+
				\sum_{i=1}^n \E[g(Y_t^i,Y_t^i)]   \Big|
				=
				n\,  \big| \E[g(Y_t^1,Y_t^1)]  \big| \le 	n\|g\|_\infty.
				\label{eq:unweighted.mean.Y}
			\end{align}
			
			Now, we compare the expectations under the $X$-system and the $Y$-system:
			\begin{align}
				\begin{split}
				&\bigg|
				\E\Big[\sum_{i,j=1}^n g(X_t^i,X_t^j)\Big]
				-
				\E\Big[\sum_{i,j=1}^n g(Y_t^i,Y_t^j)\Big]
				\bigg|
				\\
				&\le
				\sum_{i\neq j}
				\big|\E[g(X_t^i,X_t^j)]-\E[g(Y_t^i,Y_t^j)]\big|
				+
				\sum_{i=1}^n
				\big|\E[g(X_t^i,X_t^i)]-\E[g(Y_t^i,Y_t^i)]\big|.
				\end{split}
				\label{eq:unweighted.mean.compare.split}
			\end{align}
			By the definition of $\W_2$ and Jensen's inequality,
			\begin{align*}
				\big|\E[g(X_t^i,X_t^j)]-\E[g(Y_t^i,Y_t^j)]\big|
				\le
				\|\nabla g\|_\infty \W_2(P_t^{ij},\mu_t^{\otimes 2}),
				\quad i\neq j,
			\end{align*}
			and similarly
			\begin{align*}
				\big|\E[g(X_t^i,X_t^i)]-\E[g(Y_t^i,Y_t^i)]\big|
				&\lesssim
				\|\nabla g\|_\infty \W_2(P_t^i,\mu_t).
			\end{align*}
			Using the Cauchy-Schwarz inequality, the quadratic transport inequality for $\mu_t$ from Proposition \ref{lem:preliminaries}\eqref{lem:preliminaries:v}, its tensorized version for $\mu_t^{\otimes 2}$ (see, e.g., \cite[Proposition 1.9]{gozlan2010transport}), and the entropy bound from Lemma \ref{lem:max.ent.POC}\eqref{lem:max.ent.POC:ii}, we obtain
			\begin{align*}
				\sum_{i\neq j}\W_2(P_t^{ij},\mu_t^{\otimes 2})
				&\le
				\sqrt{n(n-1)}
				\Big(\sum_{i\neq j}\W_2(P_t^{ij},\mu_t^{\otimes 2})^2\Big)^{1/2}
				\notag\\
				&\lesssim
				\sqrt{n(n-1)}
				\Big(\sum_{i\neq j} H(P_t^{ij}\,|\,\mu_t^{\otimes 2})\Big)^{1/2}
				\notag\\
				&\lesssim
				\sqrt{n(n-1)}
				\Big(
				n\sum_{i=1}^n \Big(\sum_{j=1}^n(\xi_{ij}^2+\xi_{ji}^2)\Big)^2
				\Big)^{1/2}.
			\end{align*}
			Together with \eqref{eq:xi_row_sq} in Assumption \ref{assum.main}\eqref{assum:mat}, this implies
			\begin{align*}
				\sum_{i\neq j}\W_2(P_t^{ij},\mu_t^{\otimes 2})
				\lesssim
				n^{3/2}.
			\end{align*}
			Similarly, by \eqref{eq:DPI.avg.1to0} of Remark \ref{rem:data.processing} and \eqref{eq:xi_row_sq},
			\begin{align*}
				\sum_{i=1}^n \W_2(P_t^i,\mu_t)
				&\le
				\sqrt{n} \Big(\sum_{i=1}^n \W_2(P_t^i,\mu_t)^2\Big)^{1/2}
				\lesssim
				\sqrt{n} \Big(\sum_{i=1}^n H(P_t^i\,|\,\mu_t)\Big)^{1/2}
				\lesssim
				\sqrt{n}.
			\end{align*} 
			Plugging these bounds into \eqref{eq:unweighted.mean.compare.split}, we obtain
			\begin{align}
				\bigg|
				\E\Big[\sum_{i,j=1}^n g(X_t^i,X_t^j)\Big]
				-
				\E\Big[\sum_{i,j=1}^n g(Y_t^i,Y_t^j)\Big]
				\bigg|
				\lesssim
				n^{3/2}\|\nabla g\|_\infty.
				\label{eq:unweighted.mean.compare.final}
			\end{align}
			Combining \eqref{eq:unweighted.mean.Y} and \eqref{eq:unweighted.mean.compare.final},
			\begin{align}
				\E[H(X_t^1,\dots,X_t^n)]^2
				&\lesssim
				n^2\|g\|_\infty^2 + n^3\|\nabla g\|_\infty^2.
				\label{eq:unweighted.mean.bound}
			\end{align}
			Finally, combining \eqref{eq:unweighted.var.bound} and \eqref{eq:unweighted.mean.bound}
			gives \eqref{lem:cancellation.unweighted}.
		\end{proof}

		Recall the definitions of $(\eta^n)$ from \eqref{eq:def.fluc.process} and $(\wh{\eta}^n)$ from \eqref{eq:nuhat}. We next establish second moment bounds for the  pairings $ \langle \eta_t^n,\varphi\rangle$ and $\langle  \wh{\eta}_t^n,g\rangle$, which are used in the tightness argument.

		\begin{lemma}\label{lem:tightness.prelim}
			Suppose Assumption \ref{assum.main} holds with $k \ge 0$, and fix $t\in[0,T]$. Then:
			\begin{enumerate}[(i)]
				\item\label{tightness.prelim.eta}
				For every $\varphi \in C^1(\R^d)$ such that $ M(\varphi):= \big\|(1+|\cdot|)^{-1}\nabla\varphi\big\|_\infty<\infty$,
				we have
				\begin{align} \label{eq:eta.second.moment}
					\E \big[\langle \eta_t^n,\varphi\rangle^2\big]
					\lesssim
					M(\varphi)^2
					\Big(
					1
					+
					\sum_{i=1}^n \Big(\sum_{j=1}^n \big(\xi_{ij}^2 + \xi_{ji}^2\big) \Big)^2
					\Big).
				\end{align}
				
				\item\label{tightness.prelim.eta.hat}
				For every $g \in C^1_b(\R^d \times \R^d)$, we have
				\begin{align*}
					\E\big[\langle \wh{\eta}_t^n,g\rangle^2\big]
					\lesssim&\,
					\frac{\|\nabla g\|_\infty^2}{n^3}\bigg(
					\sum_{i,j=1}^n \hxi_{ij}^2
					+
					\max_{i,j\in[n]}\xi_{ij}
					\sum_{i=1}^n
					\bigg[
					\Big(\sum_{j=1}^n|\hxi_{ij}|\Big)^2
					+
					\Big(\sum_{j=1}^n|\hxi_{ji}|\Big)^2
					\bigg]
					\\
					&
					+
					\sum_{i=1}^n\Big(\sum_{j=1}^n\hxi_{ji}\Big)^2
					+
					\max_{i,j\in[n]}\xi_{ij}^2
					\Big(\sum_{i \neq j}|\hxi_{ij}|\Big)^2
					\bigg)
					+
					\frac{\big(|g(0,0)|+\|\nabla g\|_\infty\big)^2}{n}.
				\end{align*}
				\item\label{tightness.prelim.integrand}
				For every $\varphi\in C^2_b(\R^d)$, we have
				\begin{align*}
					\E\Big[
					\Big\langle \eta_t^n,\Big(\LL_{t,\mu_t^n}-\frac{\sigma^2}{2}\Delta\Big)\varphi\Big\rangle
					^2\Big]
					&\lesssim 
					\big(\|\nabla\varphi\|_\infty^2+\|\nabla^2\varphi\|_\infty^2\big)
					\Big(
					1
					+
					\sum_{i=1}^n \Big(\sum_{j=1}^n (\xi_{ij}^2+\xi_{ji}^2)\Big)^2
					\Big).
				\end{align*}
				
			\end{enumerate}
		\end{lemma}
		
		\begin{proof}
			\noindent 	\textit{(i)}. 
			We note that 
			\begin{align*}
				|\varphi(x)-\varphi(0)|
				&\le
				\int_0^1 |\nabla\varphi(rx)\cdot x|\,\dd r
				\le
				M(\varphi)\int_0^1 (1+r|x|)|x|\,\dd r
				\lesssim
				M(\varphi)(1+|x|^2),
			\end{align*}
			so all expectations below are finite by Proposition \ref{lem:preliminaries}\eqref{lem:preliminaries:iv}. Define $F: (\R^d)^n \to \R$ by
			\begin{align*}
				F(x_1,\dots,x_n)
				:=
				\frac1{\sqrt{n}}\sum_{i=1}^n\big(\varphi(x_i)-\langle \mu_t,\varphi\rangle\big).
			\end{align*}
			Then we can write $\langle \eta_t^n,\varphi\rangle = F(X_t^1,\dots,X_t^n)$, and so
			\begin{align*}
				\E \big[\big|\langle \eta_t^n,\varphi\rangle\big|^2\big]
				=
				\Var(F(X_t^1,\dots,X_t^n))
				+
				\E[ F(X_t^1,\dots,X_t^n)]^2.
			\end{align*}
			For the variance term, the Poincar\'e inequality in Proposition \ref{lem:preliminaries}\eqref{lem:preliminaries:iii} implies
			\begin{align*}
				\Var(F(X_t^1,\dots,X_t^n))
				\lesssim
				\E\bigg[\sum_{i=1}^n|\nabla_{x_i}F(X_t^1,\dots,X_t^n)|^2 \bigg] 
				=
				\frac1n\sum_{i=1}^n\E\big[|\nabla\varphi(X_t^i)|^2\big]
				\lesssim
				M(\varphi)^2,
			\end{align*}
			where in the last step we used  $|\nabla\varphi(x)|\le M(\varphi)(1+|x|)$ together with Proposition \ref{lem:preliminaries}\eqref{lem:preliminaries:iv}. 
			
			For the mean term, note that
			\begin{align*}
				|\varphi(x)-\varphi(y)|
				&\le
				|x-y|\int_0^1 |\nabla\varphi(y+r(x-y))|\,\dd r \notag\\
				&\le
				M(\varphi) |x-y|\int_0^1 \bigl(1+|y+r(x-y)|\bigr)\,\dd r \notag\\
				&\le
				M(\varphi) (1+|x|+|y|)|x-y|. 
			\end{align*}
			Let $Y$ be the solution to the McKean-Vlasov equation \eqref{eq:MKV} (recall Proposition \ref{lem:preliminaries}\eqref{lem:preliminaries:i}).  Then, for any coupling $\pi$ between $P_t^i$ and $\mu_t$, we have
			\begin{align*}
				\big|\E[\varphi(X_t^i)]-\E[\varphi(Y_t)] \big|
				&=
				\bigg|\int_{\R^d\times\R^d} \varphi(x)-\varphi(y)\,\pi(\dd x,\dd y) \bigg| \\
				&\le \int_{\R^d\times\R^d} \big|\varphi(x)-\varphi(y)\big|\,\pi(\dd x,\dd y)\\
				&\le
				M(\varphi)\int_{\R^d\times\R^d}(1+|x|+|y|)|x-y|\,\pi(\dd x,\dd y).
			\end{align*}
			Applying the Cauchy-Schwarz inequality, we find
			\begin{align*}
				&\int_{\R^d\times\R^d}(1+|x|+|y|)|x-y|\,\pi(\dd x,\dd y)\\
				&\le
				\bigg(\int_{\R^d\times\R^d}(1+|x|+|y|)^2\,\pi(\dd x,\dd y)\bigg)^{1/2}
				\bigg(\int_{\R^d\times\R^d}|x-y|^2\,\pi(\dd x,\dd y)\bigg)^{1/2}.
			\end{align*}
			By Proposition \ref{lem:preliminaries}\eqref{lem:preliminaries:iv}, the first factor is bounded uniformly, and therefore
			\begin{align*}
				\big|\E[\varphi(X_t^i)]-\E[\varphi(Y_t)]\big|
				&\lesssim
				M(\varphi)\bigg(\int_{\R^d\times\R^d}|x-y|^2\,\pi(\dd x,\dd y)\bigg)^{1/2}.
			\end{align*}
			Since $\pi$ was arbitrary, we obtain
			\begin{align*}
				\big|\E[\varphi(X_t^i)]-\E[\varphi(Y_t)]\big|
				&\lesssim
				M(\varphi)\W_2(P_t^i,\mu_t).
			\end{align*}
			Therefore,
			\begin{align*}
				\big|\E [F(X_t^1,\dots,X_t^n)]\big|
				&\lesssim
				\frac{M(\varphi)}{\sqrt{n}}\sum_{i=1}^n \W_2(P_t^i,\mu_t)
				\le
				\frac{M(\varphi)}{\sqrt{n}}\sum_{i=1}^n\sqrt{\gamma_T\,H(P_t^i\,|\,\mu_t)},
			\end{align*}
			where in the second inequality we used Proposition \ref{lem:preliminaries}(\ref{lem:preliminaries:v}). Applying the Cauchy-Schwarz inequality and the average entropy bound in  \eqref{eq:DPI.avg.1to0} of Remark \ref{rem:data.processing}, we have
			\begin{align*}
				\E [F(X_t^1,\dots,X_t^n)]^2
				&\lesssim
				M(\varphi)^2\sum_{i=1}^n H(P_t^i\,|\,\mu_t) \lesssim
				M(\varphi)^2
				\sum_{i=1}^n \Big(\sum_{j=1}^n \big(\xi_{ij}^2 + \xi_{ji}^2\big)\Big)^2.
			\end{align*}
			Combining the variance and mean bounds yields \eqref{eq:eta.second.moment}.
			
			\medskip
			\noindent
			\textit{(ii)}. Let $Y^1,\dots,Y^n$ be i.i.d. copies of the solution to the McKean-Vlasov equation \eqref{eq:MKV} (recall Proposition \ref{lem:preliminaries}\eqref{lem:preliminaries:i}). We write
			\begin{align*}
				\langle \wh{\eta}_t^n,g\rangle
				&=
				\frac1{n^{3/2}}\sum_{i,j=1}^n\hxi_{ij}\big(g(X_t^i,X_t^j)-\E[g(X_t^i,X_t^j)]\big)
				\notag\\
				&\quad+
				\frac1{n^{3/2}}\sum_{i,j=1}^n\hxi_{ij}\big(\E[g(X_t^i,X_t^j)]-\E[g(Y_t^i,Y_t^j)]\big)
				\notag\\
				&\quad+
				\frac1{n^{3/2}}\sum_{i,j=1}^n\hxi_{ij}\E[g(Y_t^i,Y_t^j)]
				=: \mathrm{I}+\mathrm{II}+\mathrm{III}.
			\end{align*}
			
			For $\mathrm{I}$, let $	H(x_1,\dots,x_n)
			:=
			\sum_{i,j=1}^n\hxi_{ij}g(x_i,x_j) $.
			Then,
			\begin{align*}
				\nabla_{x_i}H(x_1,\dots,x_n)
				&=
				\sum_{j=1}^n\hxi_{ij}\nabla_1 g(x_i,x_j)
				+
				\sum_{j=1}^n\hxi_{ji}\nabla_2 g(x_j,x_i).
			\end{align*}
			Hence, by the Poincar\'e inequality in Proposition \ref{lem:preliminaries}(\ref{lem:preliminaries:iii}),
			\begin{align*}
				\E[\mathrm{I}^2]
				&=
				\frac1{n^3}\Var\big(H(X_t^1,\dots,X_t^n)\big) \lesssim
				\frac1{n^3}\sum_{i=1}^n\E \big[ \big|\nabla_{x_i}H(X_t^1,\dots,X_t^n)\big|^2\big].
			\end{align*}
			Therefore,
			\begin{align*}
				\E[\mathrm{I}^2]
				&\lesssim
				\frac1{n^3}\sum_{i=1}^n\E\bigg[\Big|\sum_{j=1}^n\hxi_{ij}\nabla_1 g(X_t^i,X_t^j)\Big|^2\bigg]
				+
				\frac1{n^3}\sum_{i=1}^n\E\bigg[\Big|\sum_{j=1}^n\hxi_{ji}\nabla_2 g(X_t^j,X_t^i)\Big|^2\bigg].
			\end{align*}
			Applying \eqref{lem:cancellation.row} and   \eqref{lem:cancellation.col}  of Lemma  \ref{lem:cancellation}  to the first and second terms, respectively,  we obtain
			\begin{align*}
				\E[\mathrm{I}^2]
				\lesssim
				\frac{\|\nabla g\|_\infty^2}{n^3}&\bigg(
				\sum_{i,j=1}^n\hxi_{ij}^2
				+
				\sum_{i=1}^n\Big(\sum_{j=1}^n|\hxi_{ij}|\Big)^2\max_{i,j\in[n]}\xi_{ij}
				\notag\\
				&
				+
				\sum_{i=1}^n\Big(\sum_{j=1}^n|\hxi_{ji}|\Big)^2\max_{i,j\in[n]}\xi_{ij}
				+
				\sum_{i=1}^n\Big(\sum_{j=1}^n\hxi_{ji}\Big)^2
				\bigg).
			\end{align*}
			
			For $\mathrm{II}$, write $\mathrm{II}=\mathrm{II}_1+\mathrm{II}_2$, where
			\begin{align*}
				\mathrm{II}_1
				&:=
				\frac1{n^{3/2}}\sum_{i\ne j}\hxi_{ij}\big(\E[g(X_t^i,X_t^j)]-\E[g(Y_t^i,Y_t^j)]\big) \\
				\mathrm{II}_2
				&:=
				\frac1{n^{3/2}}\sum_{i=1}^n\hxi_{ii}\big(\E[g(X_t^i,X_t^i)]-\E[g(Y_t^i,Y_t^i)]\big).
			\end{align*}
			
			For $\mathrm{II}_1$, by Jensen's inequality,
			\begin{align*}
				\big|\E[g(X_t^i,X_t^j)]-\E[g(Y_t^i,Y_t^j)]\big|
				&\le
				\|\nabla g\|_\infty\,\W_2(P_t^{ij},\mu_t^{\otimes 2}), \quad i \neq j.
			\end{align*}
			Since $\mu_t$ satisfies \eqref{eq:mu_t_transport} in Proposition \ref{lem:preliminaries}(\ref{lem:preliminaries:v}), the product measure $\mu_t^{\otimes 2}$  satisfies the quadratic transport inequality with the same constant by tensorization (see, e.g., \cite[Proposition 1.9]{gozlan2010transport}). In conjunction with the maximum entropy bound in  \eqref{eq:DPI.maxpair} of Remark \ref{rem:data.processing},
			\begin{align*}
				\W_2(P_t^{ij},\mu_t^{\otimes 2})
				&\le
				\sqrt{\gamma_T\,H(P_t^{ij}\,|\,\mu_t^{\otimes 2})}
				\lesssim
				\max_{i,j\in[n]}\xi_{ij}, \quad i \neq j.
			\end{align*}
			Hence,
			\begin{align*}
				|\mathrm{II}_1|
				&\lesssim
				\frac{\|\nabla g\|_\infty}{n^{3/2}}
				\Big(\sum_{i \neq j}|\hxi_{ij}|\Big)\max_{i,j\in[n]}\xi_{ij}.
			\end{align*}
			
			For $\mathrm{II}_2$, define $h:\R^d\to\R$ by $h(x)=g(x,x)$. Then $h$ is Lipschitz and $\|\nabla h\|_\infty\lesssim \|\nabla g\|_\infty$. Hence, by Jensen's inequality and \eqref{eq:DPI.avg.1to0},
			\begin{align*}
				\big|\E[h(X_t^i)]-\E[h(Y_t^i)]\big|
				&\lesssim
				\|\nabla g\|_\infty\,\W_2(P_t^i,\mu_t)
				\le
				\|\nabla g\|_\infty\sqrt{\gamma_T\,H(P_t^i\,|\,\mu_t)}.
			\end{align*}
			Therefore, by $\hxi_{ii} = - 1$,  the Cauchy-Schwarz inequality,  and Lemma \ref{lem:max.ent.POC}(\ref{lem:max.ent.POC:ii}),
			\begin{align*}
				|\mathrm{II}_2|^2
				&\lesssim
				\frac{\|\nabla g\|_\infty^2}{n^2}\sum_{i=1}^n H(P_t^i\,|\,\mu_t)\lesssim
				\frac{\|\nabla g\|_\infty^2}{n^2}
				\sum_{i=1}^n\Big(\sum_{j=1}^n(\xi_{ij}^2+\xi_{ji}^2)\Big)^2.
			\end{align*}
			Also, by Assumption \ref{assum.main}\eqref{assum:mat},
			\begin{align*}
				\sum_{j=1}^n(\xi_{ij}^2+\xi_{ji}^2)
				&\le
				\max_{p,q\in[n]}\xi_{pq}\sum_{j=1}^n\xi_{ij}
				+
				\max_{p,q\in[n]}\xi_{pq}\sum_{j=1}^n\xi_{ji}
				\lesssim
				1,
			\end{align*}
			so $|\mathrm{II}_2|^2\lesssim \|\nabla g\|_\infty^2/n$. Combining the bounds for $\mathrm{II}_1$ and $\mathrm{II}_2$, we obtain
			\begin{align*}
				|\mathrm{II}|^2
				&\lesssim
				\frac{\|\nabla g\|_\infty^2}{n^3}
				\Big(\sum_{i \neq j}|\hxi_{ij}|\Big)^2\max_{i,j\in[n]}\xi_{ij}^2
				+
				\frac{\|\nabla g\|_\infty^2}{n}.
			\end{align*}
			
			For $\mathrm{III}$, using Assumption \ref{assum.main}\eqref{assum:mat} and the definition of $\hxi$ given in \eqref{hatxi.def}, we deduce that $\hxi_{ii}=-1$ and $\sum_{j\ne i}\hxi_{ij}=0$, and so
			\begin{align*}
				\mathrm{III}
				&=
				-
				\frac1{n^{3/2}}\sum_{i=1}^n\E[g(Y_t^i,Y_t^i)] =
				-
				\frac1{\sqrt{n}}\E[g(Y_t^1,Y_t^1)].
			\end{align*}
			Since $g$ is Lipschitz,
			\begin{align*}
				\big|\E[g(Y_t^1,Y_t^1)]\big|
				&\le
				|g(0,0)|+\|\nabla g\|_\infty\,\E\big[\big|(Y_t^1,Y_t^1)\big|\big]
				\lesssim
				|g(0,0)|+\|\nabla g\|_\infty.
			\end{align*}
			In particular,
			\begin{align*}
				|\mathrm{III}|^2
				\lesssim
				\frac{\big(|g(0,0)|+\|\nabla g\|_\infty\big)^2}{n}.
			\end{align*}
			
			Combining these bounds with $		\E\big[\big|\langle \wh{\eta}_t^n,g\rangle\big|^2\big]
			\lesssim
			\E[\mathrm{I}^2]+|\mathrm{II}|^2+|\mathrm{III}|^2$
			yields the desired estimate.

			\medskip
			\noindent
			\textit{(iii)}.
			We let
			\begin{align*}
				\psi_t(x):=\Big(\LL_{t,\mu_t}-\frac{\sigma^2}{2}\Delta\Big)\varphi(x)
				=& \,\,
				b_0(t,x)\cdot \nabla\varphi(x)
				+ \int_{\R^d} b(t,x,y)\cdot \nabla \varphi(x)\,\mu_t(\dd y)\\
				&
				+ \int_{\R^d} b(t,y,x)\cdot \nabla \varphi(y)\,\mu_t(\dd y)
			\end{align*}
			and first estimate $ M(\psi_t):=\big\|(1+|\cdot|)^{-1}\nabla \psi_t\big\|_\infty$.
			For each $x\in\R^d$, Assumption \ref{assum.main}\eqref{assum.drift} implies that we can differentiate under the integral sign, and thus
			\begin{align*}
				|\nabla \psi_t(x)|
				&\le
				|(\nabla b_0)(t,x)|\,|\nabla\varphi(x)| + |b_0(t,x)|\,|\nabla^2\varphi(x)| \notag\\
				&\quad
				+
				\int_{\R^d}
				\Big(
				|(\nabla_x b)(t,x,y)|\,|\nabla \varphi(x)|
				+
				|b(t,x,y)|\,|\nabla^2\varphi(x)|
				\Big)\,\mu_t(\dd y) \notag\\
				&\quad
				+
				\int_{\R^d}
				|(\nabla_{x'} b)(t,y,x)|\,|\nabla \varphi(y)|\,\mu_t(\dd y).
			\end{align*}
			Therefore,
			\begin{align*}
				(1+|x|)^{-1}|\nabla \psi_t(x)|
				&\lesssim
				\|\nabla\varphi\|_\infty
				+
				\|\nabla^2\varphi\|_\infty
			\end{align*}
			because $(1+|x|)^{-1}\le 1$ and $\mu_t$ is a probability measure. Taking the supremum over $x \in \R^d$ yields
			\begin{align*}
				M(\psi_t)
				\lesssim
				\|\nabla\varphi\|_\infty+\|\nabla^2\varphi\|_\infty.
			\end{align*}
			Applying part (i) with $\psi_t$ in place of $\varphi$, we obtain
			\begin{align}
				\begin{split}
				\E\Big[\Big|\Big\langle \eta_t^n,\Big(\LL_{t,\mu_t}-\frac{\sigma^2}{2}\Delta\Big)\varphi\Big\rangle\Big|^2\Big]
				&\lesssim
				M(\psi_t)^2
				\Big(
				1
				+
				\sum_{i=1}^n\Big(\sum_{j=1}^n(\xi_{ij}^2+\xi_{ji}^2)\Big)^2
				\Big) \\
				&\lesssim
				\big(\|\nabla\varphi\|_\infty^2+\|\nabla^2\varphi\|_\infty^2\big)
				\Big(
				1
				+
				\sum_{i=1}^n\Big(\sum_{j=1}^n(\xi_{ij}^2+\xi_{ji}^2)\Big)^2
				\Big).
				\end{split}
				\label{eq:part3.firstterm.corrected}
			\end{align}
			
			Moreover, using $\mu_t^n-\mu_t=\frac{1}{\sqrt{n}}\eta_t^n$, we have
			\begin{align*}
				\Big\langle \eta_t^n,\Big(\LL_{t,\mu_t^n}-\frac{\sigma^2}{2}\Delta\Big)\varphi\Big\rangle
				=& \,
				\Big\langle \eta_t^n,\Big(\LL_{t,\mu_t}-\frac{\sigma^2}{2}\Delta\Big)\varphi\Big\rangle \\
				&+
				\frac{1}{\sqrt{n}}\big\langle \eta_t^n(\dd x) \, \eta_t^n(\dd y),\, b(t,x,y)\cdot \nabla\varphi(x)\big\rangle.
			\end{align*}
			Hence, by $ |a+b|^2\lesssim |a|^2+|b|^2 $,
			\begin{align}
				\begin{split}
				&\,\E\Big[
				\Big\langle \eta_t^n,\Big(\LL_{t,\mu_t^n}-\frac{\sigma^2}{2}\Delta\Big)\varphi\Big\rangle
				^2\Big] \\
				&\lesssim
				\E\Big[\Big\langle \eta_t^n,\Big(\LL_{t,\mu_t}-\frac{\sigma^2}{2}\Delta\Big)\varphi\Big\rangle^2\Big]
				+
				\frac{1}{n}	\E\Big[\big\langle \eta_t^n(\dd x)\eta_t^n(\dd y),\, b(t,x,y)\cdot\nabla\varphi(x)\big\rangle^2\Big].
				\end{split}
				\label{eq:part3.split.corrected}
			\end{align}
			
			For the second term, let
			\begin{align*}
				g_t^{\varphi}(x,y)
				:=
				\int_{\R^d} \int_{\R^d} 
				b(t,z,z')\cdot\nabla\varphi(z)\,(\delta_x-\mu_t)(\dd z)\,(\delta_y-\mu_t)(\dd z').
			\end{align*}
			Then we can write 
			\begin{align*}
				\frac{1}{\sqrt{n}}\big\langle \eta_t^n(\dd x)\eta_t^n(\dd y),\, b(t,x,y)\cdot \nabla\varphi(x)\big\rangle
				=
				\frac{1}{n^{3/2}}\sum_{i,j=1}^n g_t^{\varphi}(X_t^i,X_t^j).
			\end{align*}
			Also, expanding $g_t^\varphi$, we have
			\begin{align} \label{eq:gvarphi}
				\begin{split}
					g_t^\varphi(x,y)
					=&
					\, \,b(t,x,y)\cdot\nabla\varphi(x)
					-
					\int_{\R^d} b(t,x,z')\cdot\nabla\varphi(x)\,\mu_t(\dd z') \\
					&
					-
					\int_{\R^d} b(t,z,y)\cdot\nabla\varphi(z)\,\mu_t(\dd z)+
					\int_{\R^d} \int_{\R^d} 
					b(t,z,z')\cdot\nabla\varphi(z)\,\mu_t(\dd z)\mu_t(\dd z').
				\end{split}
			\end{align}
			Moreover, since $b$, $\nabla_x b$, and $\nabla_y b$ are bounded by Assumption \ref{assum.main}\eqref{assum.drift} and $\mu_t$ is a probability measure,
			\begin{align*}
				\|g_t^\varphi\|_\infty+\|\nabla g_t^\varphi\|_\infty
				\lesssim
				\|\nabla\varphi\|_\infty+\|\nabla^2\varphi\|_\infty.
			\end{align*}
			Additionally, we see from \eqref{eq:gvarphi} that if $Y^1$ and $Y^2$ are independent copies of the McKean-Vlasov solution, then $ \E\big[g_t^{\varphi}(Y_t^1,Y_t^2)\big]=0$.
			Therefore, applying \eqref{lem:cancellation.unweighted}  of Lemma \ref{lem:cancellation} with $g_t^\varphi$, we obtain
			\begin{align*}
				\E \bigg[\Big(\sum_{i,j=1}^n g_t^{\varphi}(X_t^i,X_t^j)\Big)^2\bigg]
				&\lesssim
				n^2 \|g_t^{\varphi}\|_\infty^2 + n^3 \|\nabla g_t^{\varphi}\|_\infty^2.
			\end{align*}
			Hence,
			\begin{align}
				\begin{split}
					\E\Big[\Big|\frac{1}{\sqrt{n}}\big\langle \eta_t^n(\dd x)\eta_t^n(\dd y),\, b(t,x,y)\cdot\nabla\varphi(x)\big\rangle\Big|^2\Big]
				&=
				\frac{1}{n^3}\E \bigg[\Big(\sum_{i,j=1}^n g_t^{\varphi}(X_t^i,X_t^j)\Big)^2\bigg] \\
				&\lesssim
				\frac{1}{n}\|g_t^{\varphi}\|_\infty^2+\|\nabla g_t^{\varphi}\|_\infty^2 \lesssim
				\|\nabla\varphi\|_\infty^2+\|\nabla^2\varphi\|_\infty^2.
				\end{split}
				\label{eq:part3.secondterm.corrected}
			\end{align}
			
			Combining \eqref{eq:part3.firstterm.corrected}, \eqref{eq:part3.split.corrected}, and \eqref{eq:part3.secondterm.corrected}, 
			we conclude that
			\begin{align*}
				\E\Big[\Big|
				\Big\langle \eta_t^n,\Big(\LL_{t,\mu_t^n}-\frac{\sigma^2}{2}\Delta\Big)\varphi\Big\rangle
				\Big|^2\Big]
				&\lesssim
				\big(\|\nabla\varphi\|_\infty^2+\|\nabla^2\varphi\|_\infty^2\big)
				\Big(
				1
				+
				\sum_{i=1}^n \Big(\sum_{j=1}^n (\xi_{ij}^2+\xi_{ji}^2)\Big)^2
				\Big),
			\end{align*}
			as claimed.
		\end{proof}
		
		We first prove the tightness of the one-dimensional marginals $(\eta^n_t)_{n \in \N}$ in $\cH^{-k}$.
		
		\begin{proposition}\label{prop:tight.marginal}
			Suppose Assumption \ref{assum.main} holds with $k \ge \lambda_d + 2$. Then for each $t \in [0,T]$, $(\eta^n_t)_{n \in \N}$  is tight in $\cH^{-k}$.
		\end{proposition}
		
		\begin{proof}
			Since $k>\frac d2$, the Sobolev Embedding Theorem, see, e.g., \cite[Theorem 4.12, Case A with $\Omega=\R^d$, $n=d$, $p=2$, $j=0$, and $m=k$]{adams2003sobolev}, yields $\cH^k\hookrightarrow C_b(\R^d)$. Therefore,
				\begin{align*}
					\|\varphi\|_{L^\infty}
					\lesssim
					\|\varphi\|_{\cH^k},
					\quad
					\varphi\in \cH^k.
				\end{align*}
				Fix $t\in[0,T]$. Since $\eta_t^n$ is a finite signed measure, for any $\varphi\in \cH^k$,
				\begin{align*}
					|\langle \eta_t^n,\varphi\rangle|
					=
					\left|\int_{\R^d}\varphi(x)\,\eta_t^n(\dd x)\right|
					\le
					|\eta_t^n|(\R^d)\,\|\varphi\|_{L^\infty}
					\lesssim
					|\eta_t^n|(\R^d)\,\|\varphi\|_{\cH^k}
					<
					\infty.
				\end{align*}
				Thus $\eta_t^n$ defines a bounded linear functional on $\cH^k$, and hence $\eta_t^n\in \cH^{-k}$ for every $n\in\N$.  
			
			For $u \in \R^d$, let $\varphi_u : \R^d \to \mathbb{C}$ be given by $\varphi_u(x) :=e^{-2\pi i  u \cdot x}$. We prove tightness by obtaining a uniform weighted $\cH^1$ bound for the map $u\mapsto \langle \eta_t^n,\varphi_u\rangle$, and then invoking Lemma \ref{lem:compact.embed}.  By Lemma \ref{lem:tightness.prelim}\eqref{tightness.prelim.eta} applied to the real and imaginary parts of $\varphi_u$, together with  \eqref{eq:xi_row_sq} of Assumption \ref{assum.main}\eqref{assum:mat}, we have
			\begin{align}
				\E\big[\big|\langle \eta_t^n,\varphi_u\rangle\big|^2\big]
				\lesssim
				|u|^2.
				\label{eq:tight.marginal.mode}
			\end{align}
			
			Note that  $\partial_{u_j}\varphi_u = -2\pi i\,x_j\varphi_u$, for $j \in [d]$, and define
			\begin{align}
				\psi_{u,j}(x)=x_j\varphi_u(x),
				\quad
				x\in\R^d.
				\label{eq:psi.u.j}
			\end{align}
			We claim that	\begin{align} \label{eq:tight.marginal.derivative.mode}
				\E\big[\big|\langle \eta_t^n,\psi_{u,j}\rangle\big|^2\big] \lesssim 1+|u|^2.
			\end{align}
			Indeed, the real and imaginary parts of $\psi_{u,j}$ are continuously differentiable, and the bound
			\begin{align*}
				|\nabla \psi_{u,j}(x)| \lesssim 1+|u|\,|x|
			\end{align*}
			implies
			\begin{align*}
				M(\Re\, \psi_{u,j}) + M(\Im\,\psi_{u,j}) \lesssim 1+|u|,
			\end{align*}
			where we recall that $M(\varphi):=\big\|(1+|\cdot|)^{-1}\nabla\varphi\big\|_\infty$.
			Therefore, Lemma \ref{lem:tightness.prelim}\eqref{tightness.prelim.eta}, applied to $\Re\,\psi_{u,j}$ and $\Im\, \psi_{u,j}$, yields
			\begin{align*}
				\E		\big[\big|\langle \eta_t^n,\psi_{u,j}\rangle\big|^2\big] = \E\big[\langle \eta_t^n,\Re\,\psi_{u,j}\rangle^2\big]
				+
				\E\big[\langle \eta_t^n,\Im\,\psi_{u,j}\rangle^2\big]
				\lesssim
				(1+|u|)^2
				\lesssim
				1+|u|^2,
			\end{align*}
			which is exactly \eqref{eq:tight.marginal.derivative.mode}.
			
			Since $\eta_t^n$ has finite first moment by Proposition \ref{lem:preliminaries}\eqref{lem:preliminaries:iv}, the map $u\mapsto \langle \eta_t^n,\varphi_u\rangle$ is $C^1$, and
			\begin{align*}
				\partial_{u_j}\langle \eta_t^n,\varphi_u\rangle
				=
				-2\pi i\,\langle \eta_t^n,\psi_{u,j}\rangle.
			\end{align*}
			Therefore, by \eqref{eq:tight.marginal.derivative.mode},
			\begin{align}
				\E\big[\big|\partial_{u_j}\langle \eta_t^n,\varphi_u\rangle\big|^2\big]
				=
				(2\pi)^2\E\big[\big|\langle \eta_t^n,\psi_{u,j}\rangle\big|^2\big]
				\lesssim
				1+|u|^2,
				\quad
				j=1,\dots,d.
				\label{eq:tight.marginal.gradient.mode}
			\end{align}
			Define
			\begin{align*}
				Z_n
				:=
				\int_{\R^d}(1+|u|^2)^{-k+1}
				\big|\langle \eta_t^n,\varphi_u\rangle\big|^2\,\dd u
				+
				\int_{\R^d}(1+|u|^2)^{-k}
				\big|\nabla_u\langle \eta_t^n,\varphi_u\rangle\big|^2\,\dd u.
			\end{align*}
			Then Tonelli's Theorem together with \eqref{eq:tight.marginal.mode}, \eqref{eq:tight.marginal.gradient.mode} and $k \ge \lambda_d + 2$ yield
			\begin{align}
				\sup_{n\in\N}
				\E[Z_n]
				<
				\infty.
				\label{eq:tight.marginal.weighted.H1}
			\end{align}
			
			For $M>0$, let $\K_{M,k}$ be as in Lemma \ref{lem:compact.embed}. By that lemma, $\K_{M,k}$ is compact in $L^2((1+|u|^2)^{-k}\,\dd u)$. By \eqref{eq:Hminus.k.Fourier}, its preimage under the Fourier transform is the compact set
			\begin{align*}
				K_{M,k}
				:=
				\bigg\{
				\xi\in \cH^{-k}:
				&\int_{\R^d}(1+|u|^2)^{-k+1}
				\big|\F[\xi](u)\big|^2\,\dd u \notag\\
				&+
				\int_{\R^d}(1+|u|^2)^{-k}
				\big|\nabla_u\F[\xi](u)\big|^2\,\dd u
				\le M
				\bigg\}.
			\end{align*}
			Since $\eta_t^n$ is a finite  measure, $\F[\eta_t^n](u)=\langle \eta_t^n,\varphi_u\rangle$. Hence, by Markov's inequality and \eqref{eq:tight.marginal.weighted.H1},
			\begin{align*}
				\sup_{n\in\N}\PP\big(\eta_t^n\notin K_{M,k}\big)
				&=
				\sup_{n\in\N}\PP(Z_n>M)
				\le
				\frac{1}{M}\sup_{n\in\N}\E[Z_n]
				\lesssim
				\frac{1}{M}.
			\end{align*}
			Thus, $(\eta^n_t)_{n \in \N}$  is tight in $\cH^{-k}$.
		\end{proof}
		It remains to upgrade the tightness of the time marginals to tightness in  $C([0,T];\cH^{-k})$.

		\begin{proposition}\label{prop:tight}
			Suppose Assumption \ref{assum.main} holds with $k \ge \lambda_d + 2$. Then $(\eta^n)_{n\in\N}$ is tight in $C([0,T];\cH^{-k})$.
		\end{proposition}

		\begin{proof}
			With Proposition \ref{prop:tight.marginal} establishing the tightness of the one-dimensional marginals in $\cH^{-k}$, to conclude tightness in $C([0,T];\cH^{-k})$, it remains to verify Aldous' criterion, see, e.g., \cite[Lemma 23.12, Theorems 23.11, 23.9, 23.8]{kallenberg2021foundations}. Let $\delta>0$ and let $0\le \tau_1\le \tau_2\le T$ be stopping times such that $\tau_2-\tau_1\le \delta$ a.s. Recalling the definition of $\varphi_u(x):=e^{-2\pi i\,u\cdot x}$, and noting that $\eta_{\tau_2}^n-\eta_{\tau_1}^n$ is a finite signed measure, so that $ 	\F[\eta_{\tau_2}^n-\eta_{\tau_1}^n](u)
				=
				\langle \eta_{\tau_2}^n-\eta_{\tau_1}^n,\varphi_u\rangle$,
			the Fourier characterization of Sobolev norms in \eqref{eq:Hminus.k.Fourier} and Tonelli's Theorem yield
			\begin{align}
				\E\big[\|\eta^n_{\tau_2}-\eta^n_{\tau_1}\|_{\cH^{-k}}^2\big]
				&\asymp
				\int_{\R^d}(1+|u|^2)^{-k}\,
				\E\big[\big|\langle \eta^n_{\tau_2}-\eta^n_{\tau_1},\varphi_u\rangle\big|^2\big]\,\dd u.
				\label{eq:tight.fourier.reduction}
			\end{align}
			
			Fix $u\in\R^d$. Applying Lemma \ref{lem:eta.dym}  to the real and imaginary parts of $\varphi_u$, and using $\Delta\varphi_u=-4\pi^2|u|^2\varphi_u$, we obtain 
		\begin{align}
			\begin{split}
			\langle \eta_{\tau_2}^n-\eta_{\tau_1}^n,\varphi_u\rangle
			=&
			-2\pi^2\sigma^2|u|^2\int_{\tau_1}^{\tau_2}\langle \eta_s^n,\varphi_u\rangle\,\dd s
			+\int_{\tau_1}^{\tau_2}
			\Big\langle \eta_s^n,\Big(\LL_{s,\mu_s^n}-\frac{\sigma^2}{2}\Delta\Big)\varphi_u\Big\rangle\,\dd s \\
			&
			+\frac{n}{n-1}\int_{\tau_1}^{\tau_2}
			\big\langle \wh{\eta}_s^n(\dd x,\dd y),\, b(s,x,y)\cdot \nabla\varphi_u(x)\big\rangle\,\dd s \\
			&
			+\frac{\sqrt{n}}{n-1}\int_{\tau_1}^{\tau_2}\big\langle \mu_s^n(\dd x)\mu_s^n(\dd y),\, b(s,x,y)\cdot \nabla\varphi_u(x)\big\rangle\,\dd s \\
			&	+\frac{\sigma}{\sqrt{n}}\sum_{i=1}^n\int_{\tau_1}^{\tau_2}\nabla\varphi_u(X_s^i)\cdot \dd B_s^i.
			\end{split}
			\label{eq:tight.increment.mode}
		\end{align}
		
		By Lemma  \ref{lem:tightness.prelim}\eqref{tightness.prelim.integrand}, Lemma \ref{lem:tightness.prelim}\eqref{tightness.prelim.eta.hat}, Assumption \ref{assum.main}\eqref{assum:mat}, and Lemma \ref{lem:assum_i_implies_mat}, uniformly in $s\in[0,T]$,
		\begin{align}
			\begin{split}
			&\E\Big[\Big|
			\Big\langle \eta_s^n,\Big(\LL_{s,\mu_s^n}-\frac{\sigma^2}{2}\Delta\Big)\varphi_u\Big\rangle
			\Big|^2\Big]
			+
			\E\Big[\Big|
			\big\langle \wh{\eta}_s^n(\dd x,\dd y),\, b(s,x,y)\cdot \nabla\varphi_u(x)\big\rangle
			\Big|^2\Big]\\
			&\lesssim
			\|\nabla\varphi_u\|_\infty^2+\|\nabla^2\varphi_u\|_\infty^2
			\lesssim
			|u|^2 + |u|^4,
			\end{split}
			\label{eq:tight.mode.integrand.bound}
		\end{align}
		and, since $b$ is bounded while $|\nabla\varphi_u|\lesssim |u|$,
		\begin{align}
			\Big|\frac{\sqrt{n}}{n-1}\big\langle \mu_s^n(\dd x)\mu_s^n(\dd y),\, b(s,x,y)\cdot \nabla\varphi_u(x)\big\rangle\Big|^2
			\lesssim\frac{ |u|^2}{n}.
			\label{eq:tight.mode.remainder.bound}
		\end{align}
		Consequently, by the Cauchy-Schwarz inequality and the bound $\tau_2-\tau_1\le \delta$,
		\begin{align}
			\begin{split}
			&		\E\bigg[\bigg|\int_{\tau_1}^{\tau_2}
			\Big\langle \eta_s^n,\Big(\LL_{s,\mu_s^n}-\frac{\sigma^2}{2}\Delta\Big)\varphi_u\Big\rangle\,\dd s\bigg|^2\bigg] \\
			&		+
			\E\bigg[\bigg|\int_{\tau_1}^{\tau_2}
			\big\langle \wh{\eta}_s^n(\dd x,\dd y),\, b(s,x,y)\cdot \nabla\varphi_u(x)\big\rangle\,\dd s\bigg|^2\bigg] \lesssim
			\delta (|u|^2 + |u|^4).
			\end{split}
			\label{eq:tight.increment.transport}
		\end{align}
		Similarly, by \eqref{eq:tight.mode.remainder.bound},
		\begin{align*}
			\E\bigg[\bigg|\frac{\sqrt{n}}{n-1}\int_{\tau_1}^{\tau_2}
			\big\langle \mu_s^n(\dd x)\mu_s^n(\dd y),\, b(s,x,y)\cdot \nabla\varphi_u(x)\big\rangle\,\dd s\bigg|^2\bigg]
			&\lesssim
			\frac{\delta|u|^2}{n}.
		\end{align*}
		For the stochastic integral in \eqref{eq:tight.increment.mode}, Itô's isometry gives
		\begin{align}
			\E\bigg[\bigg|
			\frac{\sigma}{\sqrt{n}}\sum_{i=1}^n\int_{\tau_1}^{\tau_2}\nabla\varphi_u(X_s^i)\cdot \dd B_s^i
			\bigg|^2\bigg]
			&=
			\frac{\sigma^2}{n}\sum_{i=1}^n
			\E\bigg[\int_{\tau_1}^{\tau_2}|\nabla\varphi_u(X_s^i)|^2\,\dd s\bigg]
			\lesssim
			\delta |u|^2.
			\label{eq:tight.increment.mart}
		\end{align}
		
		Applying the variation-of-constants formula to the equation for $t\mapsto \langle \eta_t^n,\varphi_u\rangle$ given by Lemma \ref{lem:eta.dym}, we obtain, for every $t\in[0,T]$,
		\begin{align}
			\begin{split}
			\langle \eta_t^n,\varphi_u\rangle
			=&\,
			e^{-2\pi^2\sigma^2|u|^2 t}\langle \eta_0^n,\varphi_u\rangle
			+\int_0^t e^{-2\pi^2\sigma^2|u|^2(t-s)}
			\Big\langle \eta_s^n,\Big(\LL_{s,\mu_s^n}-\frac{\sigma^2}{2}\Delta\Big)\varphi_u\Big\rangle\,\dd s \\
			&
			+\frac{n}{n-1}\int_0^t e^{-2\pi^2\sigma^2|u|^2(t-s)}
			\big\langle \wh{\eta}_s^n(\dd x,\dd y),\, b(s,x,y)\cdot \nabla\varphi_u(x)\big\rangle\,\dd s \\
			&
			+\frac{\sqrt{n}}{n-1}\int_0^t e^{-2\pi^2\sigma^2|u|^2(t-s)}\big\langle \mu_s^n(\dd x)\mu_s^n(\dd y),\, b(s,x,y)\cdot \nabla\varphi_u(x)\big\rangle\,\dd s \\
			&
			+\frac{\sigma}{\sqrt{n}}\sum_{i=1}^n\int_0^t e^{-2\pi^2\sigma^2|u|^2(t-s)}\nabla\varphi_u(X_s^i)\cdot \dd B_s^i.
			\end{split}
			\label{eq:tight.mode.voc}
		\end{align}
		Since $X_0^1,\dots,X_0^n$ are i.i.d. with law $\mu_0$ and $|\varphi_u|=1$,
		\begin{align}
			\E\big[\big|\langle \eta_0^n,\varphi_u\rangle\big|^2\big]
			&=
			\E\Big[\Big|\frac1{\sqrt{n}}\sum_{i=1}^n
			\big(\varphi_u(X_0^i)-\langle \mu_0,\varphi_u\rangle\big)\Big|^2\Big]
			=
			\Var(\varphi_u(X^i_0))
			\le 1.
			\label{eq:tight.initial.mode}
		\end{align}
		Also,
		\begin{align}
			|u|^2\int_0^t e^{-2\pi^2\sigma^2|u|^2(t-s)}\,\dd s
			\lesssim 1.
			\label{eq:kernel.bound.1}
		\end{align}
		Similarly,
		\begin{align}
			|u|^2\int_0^t e^{-4\pi^2\sigma^2|u|^2(t-s)}\,\dd s
			&=
			\frac{1-e^{-4\pi^2\sigma^2|u|^2t}}{4\pi^2\sigma^2}
			\lesssim 1.
			\label{eq:kernel.bound.2}
		\end{align}
		Using \eqref{eq:tight.mode.voc} and the elementary inequality $ |a_1+\cdots+a_5|^2 \le 5\sum_{j=1}^5 |a_j|^2 $, we obtain
		\begin{align*}
			\E\big[\big|\langle \eta_t^n,\varphi_u\rangle\big|^2\big]
			\lesssim& \,
			\E\big[\big|\langle \eta_0^n,\varphi_u\rangle\big|^2\big]
			+
			\E\bigg[\bigg|\int_0^t e^{-2\pi^2\sigma^2|u|^2(t-s)}
			\Big\langle \eta_s^n,\Big(\LL_{s,\mu_s^n}-\frac{\sigma^2}{2}\Delta\Big)\varphi_u\Big\rangle\,\dd s\bigg|^2\bigg] \\
			&
			+
			\E\bigg[\bigg|\int_0^t e^{-2\pi^2\sigma^2|u|^2(t-s)}
			\big\langle \wh{\eta}_s^n(\dd x,\dd y),\, b(s,x,y)\cdot \nabla\varphi_u(x)\big\rangle\,\dd s\bigg|^2\bigg] \\
			&
			+
			\E\bigg[\bigg|\frac{\sqrt{n}}{n-1}\int_0^t e^{-2\pi^2\sigma^2|u|^2(t-s)}
			\big\langle \mu_s^n(\dd x)\mu_s^n(\dd y),\, b(s,x,y)\cdot \nabla\varphi_u(x)\big\rangle\,\dd s\bigg|^2\bigg] \\
			&
			+
			\E\bigg[\bigg|
			\frac{\sigma}{\sqrt{n}}\sum_{i=1}^n\int_0^t e^{-2\pi^2\sigma^2|u|^2(t-s)}
			\nabla\varphi_u(X_s^i)\cdot \dd B_s^i
			\bigg|^2\bigg].
		\end{align*}
		By \eqref{eq:tight.initial.mode}, the first term is at most $1$. For the second term, the Cauchy-Schwarz inequality and Tonelli's Theorem yield the bound
		\begin{align*}
			&
			\bigg(\int_0^t e^{-2\pi^2\sigma^2|u|^2(t-s)}\,\dd s\bigg)
			\int_0^t e^{-2\pi^2\sigma^2|u|^2(t-s)}
			\E\Big[\Big|
			\Big\langle \eta_s^n,\Big(\LL_{s,\mu_s^n}-\frac{\sigma^2}{2}\Delta\Big)\varphi_u\Big\rangle
			\Big|^2\Big] \,\dd s \lesssim 1,
		\end{align*}
		where we used \eqref{eq:tight.mode.integrand.bound} and \eqref{eq:kernel.bound.1}. The third and fourth terms are estimated in exactly the same way, using \eqref{eq:tight.mode.integrand.bound} and \eqref{eq:tight.mode.remainder.bound}, respectively, together  with \eqref{eq:kernel.bound.1}.
		For the fifth term, Itô's isometry gives
		\begin{align*}
			&	\,\E\bigg[\bigg|
			\frac{\sigma}{\sqrt{n}}\sum_{i=1}^n\int_0^t e^{-2\pi^2\sigma^2|u|^2(t-s)}
			\nabla\varphi_u(X_s^i)\cdot \dd B_s^i
			\bigg|^2\bigg] \\
			&=
			\frac{\sigma^2}{n}\sum_{i=1}^n
			\E\bigg[\int_0^t e^{-4\pi^2\sigma^2|u|^2(t-s)}
			|\nabla\varphi_u(X_s^i)|^2\,\dd s\bigg] \lesssim
			|u|^2\int_0^t e^{-4\pi^2\sigma^2|u|^2(t-s)}\,\dd s
			\lesssim 1,
		\end{align*}
		by $\|\nabla\varphi_u\|_\infty\lesssim |u|$ and \eqref{eq:kernel.bound.2}. Therefore,
		\begin{align*}
			\sup_{t\in[0,T]}\E\big[\big|\langle \eta_t^n,\varphi_u\rangle\big|^2\big]\lesssim 1.
		\end{align*}
		
		Consequently,
		\begin{align}
			\E\bigg[\bigg|
			2\pi^2\sigma^2|u|^2\int_{\tau_1}^{\tau_2}\langle \eta_s^n,\varphi_u\rangle\,\dd s
			\bigg|^2\bigg]
			&\lesssim
			|u|^4\,\delta\int_0^T
			\E\big[\big|\langle \eta_s^n,\varphi_u\rangle\big|^2\big]\,\dd s
			\lesssim
			\delta |u|^4.
			\label{eq:tight.increment.diffusion}
		\end{align}
		Together with \eqref{eq:tight.increment.mode}, \eqref{eq:tight.increment.transport}, \eqref{eq:tight.increment.mart}, and \eqref{eq:tight.increment.diffusion}, we conclude that
		\begin{align}
			\E\big[\big|\langle \eta_{\tau_2}^n-\eta_{\tau_1}^n,\varphi_u\rangle\big|^2\big]
			\lesssim
			\delta (|u|^2 + |u|^4).
			\label{eq:tight.increment.mode.bound}
		\end{align}
		Substituting \eqref{eq:tight.increment.mode.bound} into \eqref{eq:tight.fourier.reduction}, we obtain
		\begin{align*}	
			\E\big[\|\eta_{\tau_2}^n-\eta_{\tau_1}^n\|_{\cH^{-k}}^2\big]
			&\lesssim
			\delta\int_{\R^d}(1+|u|^2)^{-k} (|u|^2 + |u|^4)\,\dd u.
		\end{align*}
		The integral is finite since $k\ge \lambda_d+2$. Hence,
		\begin{align*}
			\lim_{\delta\downarrow 0}\,
			\limsup_{n\to\infty}\,
			\sup_{0\le \tau_1\le \tau_2\le T:\,\tau_2-\tau_1<\delta}
			\E\big[\|\eta_{\tau_2}^n-\eta_{\tau_1}^n\|_{\cH^{-k}}^2\big]
			&=
			0.
		\end{align*}
		This verifies Aldous' criterion.
	\end{proof}

	We also need tightness of the martingale parts as $\cH^{-k-1}$-valued processes. For each $n\in\N$, define
	\begin{align}\label{eq:Wn.Hvalued}
		W_t^n
		=
		-\frac{\sigma}{\sqrt{n}}\sum_{i=1}^n\sum_{\ell=1}^d
		\int_0^t \partial_\ell \delta_{X_s^i}\,\dd B_s^{i,\ell},
		\quad t\in[0,T].
	\end{align}
	Here, for $x\in\R^d$ and $\ell=1,\dots,d$, $\partial_\ell\delta_x$ denotes the distributional derivative of the Dirac mass at $x$, viewed as an element of $\cH^{-k-1}$, characterized by
	\begin{align}\label{eq:derivative.delta.pairing}
		\langle -\partial_\ell\delta_x,f\rangle_{\cH^{-k-1},\cH^{k+1}}
		=
		\partial_\ell f(x),
		\quad
		f\in \cH^{k+1}.
	\end{align}
	Since $k\ge \lambda_d$, equivalently $k>d/2$, the Sobolev Embedding Theorem, see, e.g., \cite[Theorem 4.12, Case A with $\Omega=\R^d$, $n=d$, $p=2$, $j=1$, and $m=k$]{adams2003sobolev}, yields $\cH^{k+1}\hookrightarrow C_b^1(\R^d)$. Therefore,
	\begin{align*}
		\sup_{x\in\R^d} |\nabla f(x)|
		\lesssim
		\|f\|_{\cH^{k+1}},
		\quad
		f\in \cH^{k+1}.
	\end{align*}
	Consequently, for each $\ell=1,\dots,d$,
	\begin{align}\label{eq:derivative.delta.bound}
		\sup_{x\in\R^d}\|\partial_\ell\delta_x\|_{\cH^{-k-1}}
		=
		\sup_{x\in\R^d}\sup_{\|f\|_{\cH^{k+1}}\le 1}
		\big|\langle \partial_\ell\delta_x,f\rangle_{\cH^{-k-1},\cH^{k+1}}\big|
		\lesssim 1.
	\end{align}
	Thus, the stochastic integral above is well defined as an $\cH^{-k-1}$-valued Itô integral.

	\begin{proposition} \label{prop:W.tight}
		Suppose Assumption \ref{assum.main} holds with  $k\ge \lambda_d+2$. Then, $(W^n)_{n\in\N}$ is tight in $C([0,T];\cH^{-(k+1)})$.
	\end{proposition}
	
	\begin{proof}
		The proof follows the same approach as the proofs of Propositions \ref{prop:tight.marginal} and \ref{prop:tight}. We first prove tightness of the one-dimensional marginals. Fix $t\in[0,T]$. We have
		\begin{align*}
			\langle W_t^n,\varphi_u\rangle
			&=
			\frac{\sigma}{\sqrt{n}}\sum_{i=1}^n\int_0^t \nabla\varphi_u(X_s^i)\cdot \dd B_s^i = 
			-\frac{2\pi i\,\sigma}{\sqrt{n}}\sum_{i=1}^n\int_0^t \varphi_u(X_s^i)\,u\cdot \dd B_s^i.
		\end{align*}
		Therefore, by the Itô isometry,
		\begin{align}
			\E\big[\big|\langle W_t^n,\varphi_u\rangle\big|^2\big]
			=
			4\pi^2\sigma^2 t\,|u|^2
			\lesssim
			|u|^2.
			\label{eq:tight.marginal.mode.W}
		\end{align}
		Recalling the definition of $\psi_{u,j}$ in \eqref{eq:psi.u.j}, we now show that
		\begin{align}
			\E\big[\big|\langle W_t^n,\psi_{u,j}\rangle\big|^2\big]
			\lesssim
			1+|u|^2.
			\label{eq:tight.marginal.derivative.mode.W}
		\end{align}
		To prove this, write
		\begin{align*}
			\langle W_t^n,\psi_{u,j}\rangle
			=
			\frac{\sigma}{\sqrt{n}}\sum_{i=1}^n
			\int_0^t \nabla\psi_{u,j}(X_s^i)\cdot \dd B_s^i.
		\end{align*}
		Since
		\begin{align*}
			\nabla\psi_{u,j}(x)
			=
			\varphi_u(x) e_j
			-
			2\pi i\,x_j \,\varphi_u(x) u,
		\end{align*}
where $e_j$ denotes the $j$-th standard unit vector in $\R^d$, it follows that
		\begin{align*}
			\big|\nabla\psi_{u,j}(x)\big|
			\lesssim
			1+|u|\,|x|.
		\end{align*}
		Hence, by the Itô isometry and the uniform moment bound in Proposition \ref{lem:preliminaries}\eqref{lem:preliminaries:iv},
		\begin{align*}
			\E\big[\big|\langle W_t^n,\psi_{u,j}\rangle\big|^2\big]
			&\lesssim
			\frac1n\sum_{i=1}^n
			\E\bigg[\int_0^t \big(1+|u|^2|X_s^i|^2\big)\,\dd s\bigg] 
			\lesssim
			1+|u|^2.
		\end{align*}
		
		Note that the map $u\mapsto \langle W_t^n,\varphi_u\rangle$ is $C^1$, and
		\begin{align*}
			\partial_{u_j}\langle W_t^n,\varphi_u\rangle
			=
			-2\pi i\,\langle W_t^n,\psi_{u,j}\rangle.
		\end{align*}
		Therefore, by \eqref{eq:tight.marginal.derivative.mode.W},
		\begin{align}
			\E\big[\big|\partial_{u_j}\langle W_t^n,\varphi_u\rangle\big|^2\big]
			=
			(2\pi)^2\E\big[\big|\langle W_t^n,\psi_{u,j}\rangle\big|^2\big]
			\lesssim
			1+|u|^2,
			\quad
			j \in [d].
			\label{eq:tight.marginal.gradient.mode.W}
		\end{align}
		Define
		\begin{align*}
			Z_n
			=
			\int_{\R^d}(1+|u|^2)^{-k}
			\big|\langle W_t^n,\varphi_u\rangle\big|^2\,\dd u
			+
			\int_{\R^d}(1+|u|^2)^{-k-1}
			\big|\nabla_u\langle W_t^n,\varphi_u\rangle\big|^2\,\dd u.
		\end{align*}
		Then, Tonelli's Theorem together with \eqref{eq:tight.marginal.mode.W} and \eqref{eq:tight.marginal.gradient.mode.W} yields
		\begin{align}
			\sup_{n\in\N}\E[Z_n]<\infty.
			\label{eq:tight.marginal.weighted.H1.W}
		\end{align}
		
		For $M>0$, let $\K_{M,k+1}$ be as in Lemma \ref{lem:compact.embed}. By that lemma, $\K_{M,k+1}$ is compact in $L^2((1+|u|^2)^{-k-1}\,\dd u)$. By \eqref{eq:Hminus.k.Fourier}, its preimage under the Fourier transform is the compact
		\begin{align*}
			K_{M,k+1}
			:=
			\bigg\{
			\xi\in \cH^{-k-1}:
			& \int_{\R^d}(1+|u|^2)^{-k}
			\big|\F[\xi](u)\big|^2\,\dd u \\
			& +  
			\int_{\R^d}(1+|u|^2)^{-k-1}
			\big|\nabla_u\F[\xi](u)\big|^2\,\dd u
			\le M
			\bigg\}.
		\end{align*}
			Since $W_t^n$ is defined as an $\cH^{-(k+1)}$-valued stochastic integral of derivatives of Dirac masses, its Fourier transform satisfies $\F[W_t^n](u)=\langle W_t^n,\varphi_u\rangle$.
	Hence, by Markov's inequality and \eqref{eq:tight.marginal.weighted.H1.W},
		\begin{align*}
			\sup_{n\in\N}\PP\big(W_t^n\notin K_{M,k+1}\big)
			&=
			\sup_{n\in\N}\PP(Z_n>M)
			\le
			\frac{1}{M}\sup_{n\in\N}\E[Z_n]
			\lesssim
			\frac{1}{M}.
		\end{align*}
		Thus, $(W_t^n)_{n\in\N}$ is tight in $\cH^{-k-1}$ for each fixed $t\in[0,T]$.
		
		With tightness of the one-dimensional marginals now established, to conclude tightness in $C([0,T];\cH^{-k-1})$, it remains to verify Aldous' criterion. Let $\delta>0$, and let $0\le\tau_1\le\tau_2\le T$ be stopping times such that $\tau_2-\tau_1\le\delta$ almost surely. The Hilbert-space version of the  Itô isometry (see, e.g., \cite[Theorem 2.3]{gawarecki2010stochastic}) and \eqref{eq:derivative.delta.bound} give
		\begin{align*}
			\E\big[\|W_{\tau_2}^n-W_{\tau_1}^n\|_{\cH^{-k-1}}^2\big]
			&=
			\frac{\sigma^2}{n}\sum_{i=1}^n\sum_{\ell=1}^d
			\E\bigg[\int_{\tau_1}^{\tau_2}\|\partial_\ell\delta_{X_s^i}\|_{\cH^{-k-1}}^2\,\dd s\bigg] \lesssim
			\delta.
		\end{align*}
		Hence,
		\begin{align*}
			\lim_{\delta\downarrow 0}
			\limsup_{n\to\infty}
			\sup_{0\le\tau_1\le\tau_2\le T:\,\tau_2-\tau_1<\delta}
			\E\big[\|W_{\tau_2}^n-W_{\tau_1}^n\|_{\cH^{-k-1}}^2\big]
			=
			0.
		\end{align*}
		This verifies Aldous' criterion and completes the proof.
	\end{proof}

\section{Convergence} \label{sec:conv}
In this section, we show that any limit point of $(\eta^n)_{n \in \N}$ is a solution of the fluctuation SPDE in the sense of Definition \ref{def:SPDE.limit}.
\begin{proposition}\label{prop:convergence}
	Suppose Assumption \ref{assum.main} holds with $k \ge \lambda_d + 2$. Let $(\eta^n)_{n\in\N}$ be a subsequence, still denoted by $(\eta^n)_{n\in\N}$, such that $\eta^n \stackrel{d}{\to} \eta$ in $C([0,T];\cH^{-k})$
	for some $\eta\in C([0,T];\cH^{-k})$. Set $\eta_0:=\eta(0)$. Then, there exists a process $W\in C([0,T];\cH^{-k-1})$, independent of $\eta_0$, such that $(\eta,W)$ is a solution of the fluctuation SPDE with initial condition $\zeta_0:=\eta_0$ in the sense of Definition \ref{def:SPDE.limit}.
\end{proposition}

To prove Proposition \ref{prop:convergence}, we pass to the limit in the semimartingale decomposition of $\langle\eta^n_t,f\rangle$. The martingale part gives the Gaussian noise, and the finite-variation part gives the drift, which requires passing to the limit in the pairing $\langle\eta^n_s,\LL_{s,\mu_s}f\rangle$. This is delicate because the coefficient $b$ in $\LL_{s,\mu_s}$ produces  a bounded term that need not decay at infinity, so $\LL_{s,\mu_s}f$ lies in $\cH^k_w$ rather than $\cH^k$, whereas $\eta^n$ converges only in the unweighted space $\cH^{-k}$. We pass to the limit by a cutoff argument, using the unweighted convergence on the compact part and a uniform weighted tail bound, which we establish in   Proposition \ref{prop.conv.int.eta}. Its proof uses weighted square-integrability bounds for the Bessel kernel, which we prove in Lemma \ref{lem:bessel:weighted-L2} below.

Recall the definition of the weight $w$ in \eqref{eq:w.def} and the notation from Section \ref{subsubsec:bessel}.

	\begin{lemma}\label{lem:bessel:weighted-L2}
		\begin{enumerate}[(i)]
			\item\label{lem:bessel:weighted-L2:G}
			For all $\N \ni k\ge \lambda_d$, 
			\begin{align}\label{eq:Gk:weighted.L2}
				\int_{\R^d} |G_k(x)|^2\,w(x)\,\dd x<\infty.
			\end{align}
			\item\label{lem:bessel:weighted-L2:grad}
			If $\N \ni k\ge \lambda_d+1$, then
			\begin{align}\label{eq:gradGk:weighted.L2}
				\int_{\R^d} |\nabla G_k(x)|^2\,w(x)\,\dd x<\infty.
			\end{align}
		\end{enumerate}
	\end{lemma}
	
	\begin{proof}
		We prove \eqref{lem:bessel:weighted-L2:G} and \eqref{lem:bessel:weighted-L2:grad} simultaneously. Fix $0<r_0<R_0$ and decompose
		\begin{align*}
			\R^d = B_{r_0}\cup(B_{R_0}\setminus B_{r_0})\cup(\R^d\setminus B_{R_0}).
		\end{align*}
		On  $B_{R_0}\setminus B_{r_0}$, $w$ is bounded, and both $G_k$ and $\nabla G_k$ are smooth. Hence the integrands in \eqref{eq:Gk:weighted.L2} and \eqref{eq:gradGk:weighted.L2} are integrable there.
		
		On $B_{r_0}$,  $w$ is bounded above by a  constant, so it is enough to consider the corresponding unweighted integrals.
		For \eqref{eq:Gk:weighted.L2}, the radial form \eqref{eq:bessel:radial.form}, the symmetry \eqref{eq:K.symmetry}, and the small-argument asymptotics \eqref{eq:bessel:Knu.small.r} and \eqref{eq:bessel:K0.small.r} imply that, as $|x| \to 0$,
		\begin{align}\label{eq:Gk:small.r}
			|G_k(x)|
			\asymp
			\begin{cases}
				|x|^{k-d}, & k<d,\\
				\log(1/|x|), & k=d,\\
				1, & k>d.
			\end{cases}
		\end{align}
		If $k<d$, then polar coordinates give
		\begin{align*}
			\int_{B_{r_0}} |G_k(x)|^2\,\dd x
			\asymp
			\int_0^{r_0} r^{d-1}r^{2(k-d)}\,\dd r
			=
			\int_0^{r_0} r^{2k-d-1}\,\dd r,
		\end{align*}
		which is finite provided $2k>d$, that is, $k\ge \lambda_d$. If $k=d$, then \eqref{eq:Gk:small.r} yields
		\begin{align*}
			\int_{B_{r_0}} |G_k(x)|^2\,\dd x
			\asymp
			\int_0^{r_0} r^{d-1}\log(1/r)^2\,\dd r
			=
			\int_{\log(1/r_0)}^\infty u^2 e^{-du}\,\dd u
			<
			\infty,
		\end{align*}
		where we used the change of variables $u=\log(1/r)$. If $k>d$, then $G_k$ is bounded near the origin by \eqref{eq:Gk:small.r}. This gives the local integrability needed for \eqref{eq:Gk:weighted.L2}.
		
		For \eqref{eq:gradGk:weighted.L2}, the radial form \eqref{eq:bessel:radial.form} and the derivative identity \eqref{eq:bessel:derivative.identity} give, for $x\neq 0$,
		\begin{align}\label{eq:gradGk:representation}
			|\nabla G_k(x)|
			=
			c_k\,|x|^{(k-d)/2}K_{(k-d-2)/2}(|x|).
		\end{align}
		Using again the small-argument asymptotics \eqref{eq:bessel:Knu.small.r} and \eqref{eq:bessel:K0.small.r}, as $|x|\downarrow 0$,
		\begin{align}\label{eq:gradGk:small.r}
			|\nabla G_k(x)|
			\asymp
			\begin{cases}
				|x|^{k-d-1}, & k<d+2,\\
				|x|\log(1/|x|), & k=d+2,\\
				|x|, & k>d+2.
			\end{cases}
		\end{align}
		If $k<d+2$, then polar coordinates yield
		\begin{align*}
			\int_{B_{r_0}} |\nabla G_k(x)|^2\,\dd x
			\asymp
			\int_0^{r_0} r^{d-1}r^{2(k-d-1)}\,\dd r
			=
			\int_0^{r_0} r^{2k-d-3}\,\dd r,
		\end{align*}
		which is finite provided $2k > d+2$, that is, $k \ge \lambda_d + 1$.
		If $k\ge d+2$, then $\nabla G_k$ is bounded near the origin by \eqref{eq:gradGk:small.r}. This gives the local integrability needed for \eqref{eq:gradGk:weighted.L2}.
		
		Finally, on $\R^d\setminus B_{R_0}$, the large-argument asymptotics \eqref{eq:bessel:Knu.large.r}, together with \eqref{eq:bessel:radial.form} and \eqref{eq:gradGk:representation}, shows that both $G_k$ and $\nabla G_k$ decay exponentially as $|x|\to\infty$. Since $w$ has polynomial growth, both weighted tail integrals are finite. Combining the three regions proves \eqref{eq:Gk:weighted.L2} and \eqref{eq:gradGk:weighted.L2}.
	\end{proof}
	The next proposition makes use of Lemma \ref{lem:bessel:weighted-L2} to prove a basic convergence statement for pairings with deterministic time-dependent test functions. 
	
	\begin{proposition}\label{prop.conv.int.eta}
		Suppose Assumption \ref{assum.main} holds with $k\ge \lambda_d+2$. Let $(\eta^n)_{n\in\N}$ be a subsequence, still denoted by $(\eta^n)_{n\in\N}$, such that $\eta^n \stackrel{d}{\to} \eta$ in $C([0,T];\cH^{-k})$. Then:
		\begin{enumerate}[(i)]
			\item \label{weighted.square.int} $\eta\in L^2([0,T];\cH_w^{-(k-1)})$ a.s.
			\item \label{determin.conv.pair}
			For every deterministic  $g \in L^\infty([0,T]; C_b^{k}(\R^d))$,
			\begin{align} \label{eq: conv.int.eta}
				\int_0^T \langle \eta_t^n, g_t\rangle_{\cH^{-k}_w,\cH^{k}_w}\,\dd t
				\stackrel{d}{\to}
				\int_0^T \langle \eta_t, g_t\rangle_{\cH^{-k}_w,\cH^{k}_w}\,\dd t,
			\end{align}
			with the integral on the right-hand side being finite a.s.
		\end{enumerate}
	\end{proposition}
	
	\begin{proof}
		Since $\lambda_d>d/2$, we have
		\begin{align*}
			\int_{\R^d} w(x)^{-1}\,\dd x<\infty.
		\end{align*}
		Hence, $g \in L^\infty([0,T];\cH^k_w)$. Once \eqref{weighted.square.int} is established, the embedding $\cH_w^{-(k-1)}\hookrightarrow \cH_w^{-k}$ implies that the integral on the right-hand side of \eqref{eq: conv.int.eta} is finite.
		
	We need the following additional notation. For any $R>0$, let $\phi_R\in C_c^\infty(\R^d)$ be a smooth cutoff function such that $\phi_R\equiv 1$ on $B_R$, $\phi_R\equiv 0$ on $\R^d\setminus B_{R+1}$, and, for every $m\in\N$,
		\begin{align*}
			\sup_{R>0}\|\phi_R\|_{C_b^m}<\infty.
			\nonumber
		\end{align*}
		
\noindent\textit{Step 1.}
		We write $g_t=g_{t,R}+h_{t,R}$ with $g_{t,R}=g_t\,\phi_R$ and $h_{t,R}=g_t(1-\phi_R)$. Since $g_{t,R}\in \cH^{k}$ for Lebesgue almost every $t \in [0,T]$ and $\esssup_{t\in[0,T]}\|g_{t,R}\|_{\cH^{k}}<\infty$ for each fixed $R$, the functional
		\begin{align*}
			\alpha \mapsto 
			\int_0^T \langle \alpha_t, g_{t,R}\rangle_{\cH^{-k},\cH^{k}}\,\dd t
		\end{align*}
		is well-defined and continuous on $C([0,T];\cH^{-k})$. Therefore, by the Continuous Mapping Theorem,
		\begin{align}\label{eq:compat.part.conv}
			\int_0^T\langle \eta_t^n,g_{t,R}\rangle_{\cH^{-k},\cH^{k}}\,\dd t
			\stackrel{d}{\to}
			\int_0^T\langle \eta_t,g_{t,R}\rangle_{\cH^{-k},\cH^{k}}\,\dd t.
		\end{align}
		
		\medskip\noindent\textit{Step 2.}
		Recalling the Bessel kernel from \eqref{eq:Gs.definition}, define
		\begin{align}\label{eq:F_x.def}
			F_x(z)=G_{k-1}(x-z).
		\end{align}
		We next show, for any $t  \in  [0, T]$, that
		\begin{align}\label{varianceJKfinal}
			\int_{\R^d}
			n\,\Var \bigg(\frac{1}{n}\sum_{i=1}^n F_x(X_t^i)\bigg) w(x)\,\dd x
			\lesssim 1.
		\end{align}
		Let $(\rho_\varepsilon)_{\varepsilon>0}$ be an even mollifier on $\R^d$ with $\rho_\varepsilon\ge0$, $\int_{\R^d}\rho_\varepsilon(x)\,\dd x=1$, and $\operatorname{supp}(\rho_\varepsilon)\subset B_\varepsilon$.
		Let $F_x^\varepsilon:=F_x*\rho_\varepsilon$. Then $F_x^\varepsilon\in C_b^\infty(\R^d)$, so the Poincar\'e inequality in Proposition \ref{lem:preliminaries}\eqref{lem:preliminaries:iii} applies to the function
		$(x_1,\dots,x_n)\mapsto \frac{1}{n}\sum_{i=1}^n F_x^\varepsilon(x_i)$. Moreover, by
		Lemma \ref{lem:bessel:weighted-L2} and Lemma \ref{lem:weighted-L2-mollification},
		\begin{align} \label{eq:int.F.eps.minus.F}
			G_{k-1}*\rho_\varepsilon\to G_{k-1},
			\quad
			\nabla(G_{k-1}*\rho_\varepsilon)\to\nabla G_{k-1}
			\quad\text{in }L^2(w(x)\,\dd x).
		\end{align}
		Since $\rho_\varepsilon$ is even, we have
		$F_x^\varepsilon(y)=(G_{k-1}*\rho_\varepsilon)(x-y)$. Hence, using the change of variable $z = x-y$, the inequality $w(z+y)\lesssim w(z)w(y)$, and Proposition
		\ref{lem:preliminaries}\eqref{lem:preliminaries:iv}, we have for each $i\in[n]$,
		\begin{align}
			\begin{split}
			&\int_{\R^d}\int_{\R^d}
			|F_x^\varepsilon(y)-F_x(y)|^2w(x)\,\dd x\,P_t^i(\dd y)
			\\
			&
			\lesssim
			\bigg(\int_{\R^d}w(y)\,P_t^i(\dd y)\bigg)
			\int_{\R^d}|(G_{k-1}*\rho_\varepsilon)(z)-G_{k-1}(z)|^2w(z)\,\dd z
			\to0,
			\end{split}
			\label{eq:int.grad.F.eps.minus.grad.F}
		\end{align}
		as $\varepsilon \downarrow 0$.
		Similarly, using $\nabla_yF_x^\varepsilon(y)=-\nabla(G_{k-1}*\rho_\varepsilon)(x-y)$ and $\nabla_yF_x(y)=-\nabla G_{k-1}(x-y)$, we obtain as $\varepsilon \downarrow 0$,
		\begin{align*}
			\int_{\R^d}\int_{\R^d}
			|\nabla F_x^\varepsilon(y)-\nabla F_x(y)|^2w(x)\,\dd x\,P_t^i(\dd y)
			\to0.
		\end{align*}
		Consequently, by Jensen's inequality and \eqref{eq:int.F.eps.minus.F}, \eqref{eq:int.grad.F.eps.minus.grad.F}, we have as $\varepsilon \downarrow 0$,
		\begin{align} \label{eq:exp.F.eps.minus.F}
			\int_{\R^d}
			\E\bigg[
			\bigg(
			\frac{1}{n}\sum_{i=1}^n \big(F_x^\varepsilon-F_x\big)(X_t^i)
			\bigg)^2\bigg]w(x)\,\dd x
			\to 0
		\end{align}
		and
		\begin{align} \label{eq:exp.grad.F.eps.minus.grad.F}
			\int_{\R^d}
			\E \bigg[
			\bigg|
			\frac{1}{n}\sum_{i=1}^n \big(\nabla F_x^\varepsilon-\nabla F_x\big)(X_t^i)
			\bigg|^2\bigg]w(x)\,\dd x
			\to 0.
		\end{align}
		Therefore, applying the Poincar\'e inequality to the function
		$(x_1,\dots,x_n)\mapsto \frac{1}{n}\sum_{i=1}^n F_x^\varepsilon(x_i)$, integrating against $w(x)\,\dd x$, and letting
		$\varepsilon\downarrow0$ by using \eqref{eq:exp.F.eps.minus.F}, \eqref{eq:exp.grad.F.eps.minus.grad.F}, we obtain
		\begin{align*}
			\int_{\R^d}
			n\,\Var\bigg(\frac{1}{n}\sum_{i=1}^n F_x(X_t^i)\bigg)w(x)\,\dd x
			&\lesssim
			\frac{1}{n}\sum_{i=1}^n
			\int_{\R^d}\int_{\R^d}|\nabla F_x(y)|^2w(x)\,\dd x\,P_t^i(\dd y).
		\end{align*}
		Applying Tonelli’s theorem and the change of variables $z=x-y$ to the right-hand side,
		\begin{align}\label{eq:variance:tonelli}
			\int_{\R^d}
			n\,\Var\bigg(\frac{1}{n}\sum_{i=1}^n F_x(X_t^i)\bigg) w(x)\,\dd x
			&\lesssim
			\frac{1}{n}\sum_{i=1}^n\int_{\R^d}\int_{\R^d}|\nabla G_{k-1}(z)|^2\,w(z+y)\,\dd z\,P_t^i(\dd y).
		\end{align}
		Using the inequality $w(z+y)\lesssim  w(z)w(y)$, Proposition \ref{lem:preliminaries}\eqref{lem:preliminaries:iv}, and Lemma \ref{lem:bessel:weighted-L2}\eqref{lem:bessel:weighted-L2:grad}, it follows that the right-hand side of \eqref{eq:variance:tonelli} is $\lesssim 1$, which proves \eqref{varianceJKfinal}.
		
		\medskip\noindent\textit{Step 3.}
		Recall the definition of $F_x$ in \eqref{eq:F_x.def}, fix $t \in [0,T]$, and for $i \in [n]$ define
		\begin{align*}
			I_i
			&=
			\int_{\R^d}\left(\E[F_x(X_t^i)]-\E[F_x(Y_t)]\right)^2 w(x)\,\dd x =
			\int_{\R^d}\left|\int_{\R^d}G_{k-1}(x-z)\,(P_t^i-\mu_t)(\dd z)\right|^2 w(x)\,\dd x.
		\end{align*}
		We next show that
		\begin{align}\label{transportJKfinal}
			\sum_{i=1}^n I_i \lesssim 1.
		\end{align}
		By the definition of the total variation measure and Minkowski's integral inequality,
		\begin{align}
			\begin{split}
			\sqrt{I_i}
			&=
			\bigg(
			\int_{\R^d}
			\left(
			\int_{\R^d} G_{k-1}(x-z)\,(P_t^i-\mu_t)(\dd z)
			\right)^2
			w(x)\,\dd x
			\bigg)^{1/2} \\
			&\le
			\bigg(
			\int_{\R^d}
			\left(
			\int_{\R^d} |G_{k-1}(x-z)|\,|P_t^i-\mu_t|(\dd z)
			\right)^2
			w(x)\,\dd x
			\bigg)^{1/2} \\
			&\le
			\int_{\R^d}
			\left(
			\int_{\R^d}|G_{k-1}(x-z)|^2\,w(x)\,\dd x
			\right)^{1/2}
			\,|P_t^i-\mu_t|(\dd z).
			\end{split}
			\label{eq:transport:minkowski}
		\end{align}
		Using the change of variables $u=x-z$ and the inequality $w(u+z)\lesssim w(u)w(z)$,
		\begin{align}\label{eq:transport:inner.bound}
			\int_{\R^d}|G_{k-1}(x-z)|^2\,w(x)\,\dd x
			\lesssim
			w(z)\int_{\R^d}|G_{k-1}(u)|^2\,w(u)\,\dd u.
		\end{align}
		The last integral is finite by Lemma \ref{lem:bessel:weighted-L2}\eqref{lem:bessel:weighted-L2:G}. Substituting \eqref{eq:transport:inner.bound} into \eqref{eq:transport:minkowski} yields
		\begin{align}\label{eq:transport:reduce}
			\sqrt{I_i}
			\lesssim
			\left(\int_{\R^d}|G_{k-1}(u)|^2\,w(u)\,\dd u\right)^{1/2}
			\int_{\R^d}\sqrt{w(z)}\,|P_t^i-\mu_t|(\dd z).
		\end{align}
		The weighted Pinsker inequality in Lemma \ref{lem:weightedPinsker}, together with the uniform moment bound in Proposition \ref{lem:preliminaries}\eqref{lem:preliminaries:iv}, gives
		\begin{align}\label{eq:transport:pinsker}
			\left(\int_{\R^d}\sqrt{w(z)}\,|P_t^i-\mu_t|(\dd z)\right)^2
			\lesssim
			H(P_t^i\,|\,\mu_t).
		\end{align}
		Combining \eqref{eq:transport:reduce} and \eqref{eq:transport:pinsker} yields
		\begin{align*}
			I_i
			\lesssim
			\left(\int_{\R^d}|G_{k-1}(u)|^2\,w(u)\,\dd u\right)
			H(P_t^i\,|\,\mu_t).
		\end{align*}
		Summing over $i \in [n]$, and using \eqref{eq:DPI.avg.1to0} of Remark \ref{rem:data.processing} and \eqref{eq:xi_row_sq} of Assumption \ref{assum.main}\eqref{assum:mat}, gives \eqref{transportJKfinal}.
		
		\medskip
		\noindent\textit{Step 4.}
		Recall the definition of the Bessel operator $J$ in \eqref{eq:Js.definition}. We next prove that
		\begin{align}\label{weighted.bound.Jkminusone}
			\E \bigg[\int_0^T \int_{\R^d}\big(J^{-(k-1)}\eta_t^n(x)\big)^2\,w(x)\,\dd x\,\dd t
			\bigg]\lesssim 1.
		\end{align}
		To see this, note that for each $x \in \R^d$,
		\begin{align*}
			J^{-(k-1)}\eta_t^n(x)
			=
			(G_{k-1} * \eta_t^n)(x)
			=
			\sqrt{n}\bigg(\frac1n\sum_{i=1}^n F_x(X_t^i) - (\mu_t * G_{k-1})(x)\bigg).
		\end{align*}
		Therefore,
		\begin{align}\label{eq:Jk:pointwise.decomp}
			\E \big[(J^{-(k-1)}\eta_t^n(x))^2\big]
			&=
			n\,\Var\bigg(\frac{1}{n}\sum_{i=1}^n F_x(X_t^i)\bigg)
			+
			n\bigg(\frac{1}{n}\sum_{i=1}^n\big(\E[F_x(X_t^i)]-\E[F_x(Y_t)]\big)\bigg)^2.
		\end{align}
		Therefore, integrating \eqref{eq:Jk:pointwise.decomp} against $w(x)\,\dd x$ and using Jensen’s inequality on the second term on the right-hand side, we get
		\begin{align}
			\begin{split}
			\int_{\R^d}\E\big[(J^{-(k-1)}\eta_t^n(x))^2\big]\,w(x)\,\dd x
			&\le
			\int_{\R^d}
			n\,\Var\bigg(\frac{1}{n}\sum_{i=1}^n F_x(X_t^i)\bigg) w(x)\,\dd x \\
			&\quad
			+\sum_{i=1}^{n} \int_{\R^d} \left(\E[F_x(X_t^i)]-\E[F_x(Y_t)]\right)^2 w(x)\,\dd x.
			\end{split}
			\label{eq:decomJk}
		\end{align}
		Combining \eqref{eq:decomJk} with Steps 2 and 3, and then integrating over $t\in[0,T]$, we obtain \eqref{weighted.bound.Jkminusone}. Equivalently, by \eqref{eq:Hminus.kw.characterization},
		\begin{align*}
			\sup_{n\in\N}
				\E\bigg[\int_0^T
				\|\eta_t^n\|_{\cH_w^{-(k-1)}}^2\,\dd t
				\bigg]<\infty.
		\end{align*}
		Since $\cH_w^{-(k-1)}\hookrightarrow \cH_w^{-k}$, it follows that
		\begin{align}\label{eq:uniform.norm.bound.eta.n}
			\sup_{n\in\N}
				\E\bigg[\int_0^T
				\|\eta_t^n\|_{\cH_w^{-k}}^2\,\dd t
				\bigg]<\infty.
		\end{align}
		
		\medskip
		\noindent\textit{Step 5.}
		Next, we prove the tail estimate
		\begin{align}\label{L1.conv.R}
			\lim_{R\to\infty}\sup_{n\in\N}\E\bigg[\bigg|\int_0^T \langle \eta_t^n,h_{t,R}\rangle_{\cH_w^{-k},\cH_w^k}\,\dd t\bigg|\bigg]=0.
		\end{align}
		Fix $R>0$. We first show that the pairing in \eqref{L1.conv.R} is well-defined. Since $h_{t,R}=g_t(1-\phi_R)$, the Leibniz rule, the uniform $C_b^k$-bounds on $g$, and the construction of $\phi_R$ imply that, for each $0\le |\boldsymbol{\alpha}|\le k$,
		\begin{align*} 
			|D^{\boldsymbol{\alpha}}h_{t,R}(x)|\lesssim  1,
			\quad	\text{for} \,\, x\in\R^d, \quad \text{and} \quad 	D^{\boldsymbol{\alpha}}h_{t,R}(x)=0, \quad \text{for}
			\,\, x\in B_R.
		\end{align*}
		Therefore, for a.e. $t\in[0,T]$,
		\begin{align} \label{eq:htR.norm.finite}
			\sum_{0\le |\boldsymbol{\alpha}|\le k}\int_{\R^d}\frac{|D^{\boldsymbol{\alpha}}h_{t,R}(x)|^2}{w(x)}\,\dd x
			\lesssim
			\int_{\R^d\setminus B_R} w(x)^{-1}\,\dd x
			<\infty.
		\end{align}
		Thus, $h_{t,R}$ has finite $\cH^k_w$-norm. On the other hand, by \eqref{eq:uniform.norm.bound.eta.n}, for each $n\in\N$, we have $\eta_t^n\in \cH_w^{-k}$ for a.e. $t\in[0,T]$, a.s. Hence, the pairing $\langle \eta_t^n,h_{t,R}\rangle_{\cH_w^{-k},\cH_w^k}$ is well-defined for a.e. $t\in[0,T]$, a.s.
		
		By the Cauchy-Schwarz inequality,
		\begin{align*}
			\E\bigg[\bigg|\int_0^T \langle \eta_t^n,h_{t,R}\rangle_{\cH_w^{-k},\cH_w^k}\,\dd t\bigg|\bigg]
			&\le
			\bigg(\E\bigg[\int_0^T \|\eta_t^n\|_{\cH_w^{-k}}^2\,\dd t\bigg]\bigg)^{1/2}
			\bigg(\int_0^T \|h_{t,R}\|_{\cH_w^k}^2\,\dd t\bigg)^{1/2}.
		\end{align*}
		The first factor on the right-hand side is bounded uniformly in $n$ by \eqref{eq:uniform.norm.bound.eta.n}. For the second factor, using \eqref{eq:htR.norm.finite}, we obtain
		\begin{align*}
			\int_0^T \|h_{t,R}\|_{\cH_w^k}^2\,\dd t
			\lesssim
			\int_{\R^d\setminus B_R} w(x)^{-1}\,\dd x \xrightarrow[R\to\infty]{} 0.
		\end{align*}
		This yields \eqref{L1.conv.R}.
		
		\medskip\noindent\textit{Step 6.}
		We next prove the corresponding tail estimate for the limit:
		\begin{align}\label{eq:eta.tail.L1}
			\lim_{R\to\infty}\E\bigg[\bigg|\int_0^T\langle \eta_t,h_{t,R}\rangle_{\cH^{-k}_w,\cH^{k}_w}\,\dd t\bigg|\bigg]=0.
		\end{align}
		To this end, we first establish \eqref{weighted.square.int}. Define
		\begin{align*}
			F(\alpha)
			=
			\int_0^T\|\alpha_t\|_{\cH_w^{-(k-1)}}^2\,\dd t,
			\text{ if }  \alpha\in L^2([0,T];\cH_w^{-(k-1)}), \quad F(\alpha) = 
			\infty,
			\, \, \text{otherwise.}
		\end{align*}
		Note that $F$ is lower semicontinuous on $C([0,T];\cH^{-k})$. Indeed, suppose that $\alpha^m\to\alpha$ in $C([0,T];\cH^{-k})$ and $\liminf_{m\to\infty}F(\alpha^m)<\infty$. Choose a subsequence $(\alpha^{m_\ell})_{\ell \in \N}$ such that $F(\alpha^{m_\ell})\to\liminf_{m\to\infty}F(\alpha^m)$, so that $(\alpha^{m_\ell})_{\ell \in \N}$ is bounded in $L^2([0,T];\cH_w^{-(k-1)})$. Hence, a further subsequence converges weakly in $L^2([0,T];\cH_w^{-(k-1)})$ to some $\beta\in L^2([0,T];\cH_w^{-(k-1)})$. By the continuous embedding  $\cH_w^{-(k-1)}\hookrightarrow \cH^{-k}$, the same subsequence converges weakly to $\beta$ in $L^2([0,T];\cH^{-k})$, while $\alpha^m\to\alpha$ strongly in $L^2([0,T];\cH^{-k})$. Therefore,  $\beta=\alpha$ in $L^2([0,T];\cH^{-k})$, and the weak lower semicontinuity of the $L^2([0,T];\cH_w^{-(k-1)})$-norm gives
		\begin{align*}
			F(\alpha)
			=
			\|\alpha\|_{L^2([0,T];\cH_w^{-(k-1)})}^2
			\le
			\liminf_{\ell \to\infty}
			\|\alpha^{m_\ell}\|_{L^2([0,T];\cH_w^{-(k-1)})}^2
			=
			\liminf_{m\to\infty}F(\alpha^m).
		\end{align*}
		Hence, by the Portmanteau Theorem and \eqref{weighted.bound.Jkminusone}, \eqref{eq:Hminus.kw.characterization},
		\begin{align*}
				\E[F(\eta)]
				\le
				\liminf_{n\to\infty}\E[F(\eta^n)]
				<\infty.
		\end{align*}
		This proves \eqref{weighted.square.int}. In particular, since $\cH_w^{-(k-1)}\hookrightarrow \cH_w^{-k}$,
		\begin{align}\label{eq:eta.weighted.finite}
				\E\bigg[\int_0^T\|\eta_t\|_{\cH_w^{-k}}^2\,\dd t\bigg]<\infty.
		\end{align}
		
		Now fix $R>0$. By Step 5, $h_{t,R}\in \cH_w^k$ for a.e. $t\in[0,T]$. Moreover, by \eqref{eq:htR.norm.finite},
		\begin{align*}
			\int_0^T\|h_{t,R}\|_{\cH_w^k}^2\,\dd t
			\lesssim
			\int_{\R^d\setminus B_R}w(x)^{-1}\,\dd x\, \xrightarrow[R\to\infty]{} 0.
		\end{align*}
		Therefore, by the Cauchy-Schwarz inequality and \eqref{eq:eta.weighted.finite},
		\begin{align*}
			\E\bigg[\bigg|\int_0^T\langle \eta_t,h_{t,R}\rangle_{\cH^{-k}_w,\cH^{k}_w}\,\dd t\bigg|\bigg]
			&\le
			\bigg(\E\bigg[\int_0^T\|\eta_t\|_{\cH_w^{-k}}^2\,\dd t\bigg]\bigg)^{1/2}
			\bigg(\int_0^T\|h_{t,R}\|_{\cH_w^k}^2\,\dd t\bigg)^{1/2} \xrightarrow[R\to\infty]{} 0.
		\end{align*}
		This  yields \eqref{eq:eta.tail.L1}.
		
		\medskip\noindent\textit{Step 7.}
		Fix $R>0$ and write $g_t=g_{t,R}+h_{t,R}$. Since $g_{t,R}\in \cH^k\cap \cH_w^k$ for a.e. $t \in [0,T]$, the weighted and unweighted pairings agree on $g_{t,R}$ for those $t$. Moreover, by \eqref{eq:compat.part.conv} of Step 1,
		\begin{align}\label{eq:eta.n.g.conv}
			\int_0^T\langle \eta_t^n,g_{t,R}\rangle_{\cH^{-k},\cH^k}\,\dd t
			\stackrel{d}{\to}
			\int_0^T\langle \eta_t,g_{t,R}\rangle_{\cH^{-k},\cH^k}\,\dd t.
		\end{align}
		Let $\varphi:\R\to\R$ be bounded and $1$-Lipschitz. Then:
		\begin{align*}
			&\bigg|\E\bigg[\varphi\bigg(\int_0^T\langle \eta_t^n,g_t\rangle_{\cH_w^{-k},\cH_w^k}\,\dd t\bigg)\bigg]
			-\E\bigg[\varphi\bigg(\int_0^T\langle \eta_t,g_t\rangle_{\cH_w^{-k},\cH_w^k}\,\dd t\bigg)\bigg]\bigg| \\
			&\le
			\bigg|\E\bigg[\varphi\bigg(\int_0^T\langle \eta_t^n,g_{t,R}\rangle_{\cH_w^{-k},\cH_w^k}\,\dd t\bigg)\bigg]
			-\E\bigg[\varphi\bigg(\int_0^T\langle \eta_t,g_{t,R}\rangle_{\cH_w^{-k},\cH_w^k}\,\dd t\bigg)\bigg]\bigg| \\
			&\quad +
			\E\bigg[\bigg|\int_0^T\langle \eta_t^n,h_{t,R}\rangle_{\cH_w^{-k},\cH_w^k}\,\dd t\bigg|\bigg]
			+
			\E\bigg[\bigg|\int_0^T\langle \eta_t,h_{t,R}\rangle_{\cH_w^{-k},\cH_w^k}\,\dd t\bigg|\bigg].
		\end{align*}
		Since the weighted and unweighted pairings agree on $g_{t,R}$, the first term on the right-hand side tends to $0$ by \eqref{eq:eta.n.g.conv}. Therefore,
		\begin{align*}
			&\limsup_{n\to\infty}
			\bigg|\E\bigg[\varphi\bigg(\int_0^T\langle \eta_t^n,g_t\rangle_{\cH_w^{-k},\cH_w^k}\,\dd t\bigg)\bigg]
			-\E\bigg[\varphi\bigg(\int_0^T\langle \eta_t,g_t\rangle_{\cH_w^{-k},\cH_w^k}\,\dd t\bigg)\bigg]\bigg| \\
			&\le
			\sup_{n\in\N}\E\bigg[\bigg|\int_0^T\langle \eta_t^n,h_{t,R}\rangle_{\cH_w^{-k},\cH_w^k}\,\dd t\bigg|\bigg]
			+
			\E\bigg[\bigg|\int_0^T\langle \eta_t,h_{t,R}\rangle_{\cH_w^{-k},\cH_w^k}\,\dd t\bigg|\bigg].
		\end{align*}
		Letting $R\to\infty$ and using \eqref{L1.conv.R} of Step 5 together with \eqref{eq:eta.tail.L1} of Step 6, we obtain
		\begin{align*}
			\lim_{n\to\infty}
			\E\bigg[\varphi\bigg(\int_0^T\langle \eta_t^n,g_t\rangle_{\cH_w^{-k},\cH_w^k}\,\dd t\bigg)\bigg]
			=
			\E\bigg[\varphi\bigg(\int_0^T\langle \eta_t,g_t\rangle_{\cH_w^{-k},\cH_w^k}\,\dd t\bigg)\bigg].
		\end{align*}
		Since this holds for every bounded $1$-Lipschitz $\varphi$, we conclude \eqref{eq: conv.int.eta}.
	\end{proof}

	\begin{proof}[\bf{Proof of Proposition \ref{prop:convergence}}]

		\medskip\noindent\textit{Step 1.}
		Proposition \ref{prop:W.tight} shows that the sequence $(W^n)_{n\in\N}$ is tight in $C([0,T];\cH^{-k-1})$. Since $\eta^n	\stackrel{d}{\to} \eta$, the sequence $(\eta^n,W^n)_{n\in\N}$ is  tight in $ 	\cC := C([0,T];\cH^{-k})\times C([0,T];\cH^{-k-1})$.
		After passing to a further subsequence, we may assume that there exists a process $W\in C([0,T];\cH^{-k-1})$ with
		\begin{align}\label{eq:joint.conv.eta.W} 
			(\eta^n,W^n)
			\stackrel{d}{\to}
			(\eta,W)
			\quad\text{in } \cC.
		\end{align}
		We realize $(\eta,W)$ on the canonical space $\cC$,
		equipped with its Borel $\sigma$-algebra $\F$, the law $\PP$ of $(\eta,W)$, and the usual augmentation of the canonical filtration $\FF=(\F_t)_{t\in[0,T]}$. On this stochastic basis, both coordinate processes are continuous and $\FF$-adapted. Item \eqref{SPDE.def.integrability} of Definition \ref{def:SPDE.limit} follows from Proposition \ref{prop.conv.int.eta}\eqref{weighted.square.int}. It remains to verify items \eqref{SPDE.def.dym} and  \eqref{SPDE.def.weak} of Definition \ref{def:SPDE.limit}. 
		
		\medskip\noindent\textit{Step 2.}
		Fix $f\in C_c^\infty(\R^d)$. By Lemma \ref{lem:eta.dym}, the process $(\langle \eta_t^n,f\rangle)_{t\in[0,T]}$ admits the semimartingale representation \eqref{eq:eta.semimartingale}. We first show that the last two drift terms in \eqref{eq:eta.semimartingale} converge to $0$ in distribution. More specifically, for every $t\in[0,T]$,
		\begin{align}\label{eq:second.drift.conv}
			\int_0^t \big\langle \wh{\eta}^n_s(\dd x,\dd y),\, b(s,x,y)\cdot \nabla f(x)\big\rangle\,\dd s \stackrel{d}{\to} 0,
		\end{align}
		and
		\begin{align}\label{eq:third.drift.conv}
			\frac{\sqrt{n}}{n-1}\int_0^t \big\langle \mu_s^n(\dd x)\mu_s^n(\dd y),\, b(s,x,y)\cdot \nabla f(x)\big\rangle\,\dd s \stackrel{d}{\to} 0.
		\end{align}
		Since $f\in C_c^\infty(\R^d)$ and $b$ has bounded first derivatives, the functions $(x,y)\mapsto b(s,x,y)\cdot \nabla f(x)$, $s\in[0,T]$, form a bounded subset of $C_b^1(\R^d\times\R^d)$. Therefore, Lemma \ref{lem:tightness.prelim}\eqref{tightness.prelim.eta.hat}, together with Assumption \ref{assum.main}\eqref{assum:mat} and Lemma \ref{lem:assum_i_implies_mat}, implies that 
		\begin{align*}
			\sup_{s\in[0,T]}
			\E\left[
			\left|
			\big\langle
			\wh{\eta}^n_s(\dd x,\dd y),
			b(s,x,y)\cdot\nabla f(x)
			\big\rangle
			\right|^2
			\right]
			\longrightarrow 0.
		\end{align*}
		The Cauchy-Schwarz inequality then gives
		\begin{align*}
			\E\bigg[\bigg|\int_0^t \big\langle \wh{\eta}^n_s(\dd x,\dd y),\, b(s,x,y)\cdot \nabla f(x)\big\rangle\,\dd s\bigg|\bigg]\to 0,
		\end{align*}
		which implies \eqref{eq:second.drift.conv}. On the other hand, since $b$ is bounded while $f\in C_c^\infty(\R^d)$, we have
		\begin{align*}
			\bigg|\frac{\sqrt{n}}{n-1}\int_0^t \Big\langle \mu_s^n(\dd x)\mu_s^n(\dd y),\, b(s,x,y)\cdot \nabla f(x)\Big\rangle\,\dd s\bigg|
			\lesssim
			\frac{\sqrt{n}}{n-1},
		\end{align*}
		which tends to $0$ as $n\to\infty$, proving \eqref{eq:third.drift.conv}.
		
		\medskip\noindent\textit{Step 3.}
		For the first drift term in \eqref{eq:eta.semimartingale}, write
		\begin{align} \label{eq:first.drift.decom}
			\int_0^t\big\langle \eta_s^n,\,\LL_{s,\mu_s^n}f\big\rangle\,\dd s
			=
			\int_0^t\big\langle \eta_s^n,\,\LL_{s,\mu_s}f\big\rangle\,\dd s
			+
			\int_0^t \big\langle \eta_s^n,\LL_{s,\mu_s^n}f-\LL_{s,\mu_s}f\big\rangle\,\dd s.
		\end{align}
		We show in this step that the first integral on the right-hand side converges in distribution:
		\begin{align} \label{eq:eta.n.L.conv}
			\int_0^t\big\langle \eta_s^n,\,\LL_{s,\mu_s}f\big\rangle\,\dd s
			\stackrel{d}{\to}
			\int_0^t\big\langle \eta_s,\,\LL_{s,\mu_s}f\big\rangle_{\cH^{-k}_w,\cH^{k}_w}\,\dd s.
		\end{align}
		Recalling \eqref{def.generator}, and using that $f\in C_c^\infty(\R^d)$ while $b_0$ and $b$ satisfy Assumption \ref{assum.main}\eqref{assum.drift}, we see that the function  $ 	(s,x)\mapsto \mathbf{1}_{[0,t]}(s)\,\LL_{s,\mu_s}f(x)$
		belongs to $L^\infty([0,T];C_b^{k}(\R^d))$. Applying Proposition \ref{prop.conv.int.eta}\eqref{determin.conv.pair} to this function, we obtain \eqref{eq:eta.n.L.conv}.

		\medskip\noindent\textit{Step 4.}
		We show that the second integral on the right-hand side of \eqref{eq:first.drift.decom} converges in distribution to $0$.
		By \eqref{def.generator},
		\begin{align*}
			(\LL_{s,\mu_s^n}f-\LL_{s,\mu_s}f)(x)
			=
			\int_{\R^d} b(s,x,x')\cdot \nabla f(x)\,(\mu_s^n-\mu_s)(\dd x').
		\end{align*}
		Since $f\in C_c^\infty(\R^d)$, there exists a compact  $K\subset \R^d$ such that  for every $s\in[0,T]$,  $x'\in\R^d$, and  $0\le |\boldsymbol{\alpha}|\le k$, the function $	x\mapsto D_x^{\boldsymbol{\alpha}}(b(s,x,x')\cdot \nabla f(x)) $, and hence $\LL_{s,\mu_s^n}f-\LL_{s,\mu_s}f$, is supported in $K$. For $s\in[0,T]$, $x\in K$, and $0\le |\boldsymbol{\alpha}|\le k$, define
		\begin{align*}
			\psi_{s,x}^{\boldsymbol{\alpha}}(x')
			=
			D_x^{\boldsymbol{\alpha}}(b(s,x,x')\cdot \nabla f(x)).
		\end{align*}
		Then,
		\begin{align}\label{eq:step4:derivative}
			D^{\boldsymbol{\alpha}}(\LL_{s,\mu_s^n}f-\LL_{s,\mu_s}f)(x)
			=
			\int_{\R^d} \psi_{s,x}^{\boldsymbol{\alpha}}(x')\,(\mu_s^n-\mu_s)(\dd x')
			=
			\frac{1}{\sqrt{n}}\langle \eta_s^n,\psi_{s,x}^{\boldsymbol{\alpha}}\rangle.
		\end{align}
		By Assumption \ref{assum.main}\eqref{assum.drift}, for each $0\le |\boldsymbol{\alpha}|\le k$, the family $(\psi_{s,x}^{\boldsymbol{\alpha}})_{s\in[0,t],\,x\in K}$ satisfies
		\begin{align*}
			\sup_{s\in[0,t]}\sup_{x\in K} M(\psi_{s,x}^{\boldsymbol{\alpha}}) \lesssim 1,
		\end{align*}
		where $ M(\varphi):= \big\|(1+|\cdot|)^{-1}\nabla\varphi\big\|_\infty$. 
		Therefore, Lemma \ref{lem:tightness.prelim}\eqref{tightness.prelim.eta} yields
		\begin{align*}
			\E\Big[\big|\langle \eta_s^n,\psi_{s,x}^{\boldsymbol{\alpha}}\rangle\big|^2\Big]
			\lesssim
			1 + 
			\sum_{i=1}^n\Big(\sum_{j=1}^n(\xi_{ij}^2+\xi_{ji}^2)\Big)^2,
		\end{align*}
		uniformly in $s\in[0,t]$ and $x\in K$. In conjunction with \eqref{eq:step4:derivative}, we obtain
		\begin{align*}
			\E\big[\big\|\LL_{s,\mu_s^n}f-\LL_{s,\mu_s}f\big\|_{\cH^k}^2\big]
			&=
			\sum_{0\le |\boldsymbol{\alpha}|\le k}\int_K \E\big[\big|D^{\boldsymbol{\alpha}}(\LL_{s,\mu_s^n}f-\LL_{s,\mu_s}f)(x)\big|^2\big]\,\dd x \\
			&\lesssim
			\frac{1}{n}
			+
			\frac{1}{n}\sum_{i=1}^n\Big(\sum_{j=1}^n(\xi_{ij}^2+\xi_{ji}^2)\Big)^2,
		\end{align*}
		uniformly in $s\in[0,t]$. Hence, by \eqref{eq:xi_row_sq} of Assumption \ref{assum.main}\eqref{assum:mat},
		\begin{align*}
			\E\bigg[\int_0^t \big\|\LL_{s,\mu_s^n}f-\LL_{s,\mu_s}f\big\|_{\cH^k}^2\,\dd s\bigg] \to 0.
		\end{align*}
		On the other hand, by \eqref{eq:uniform.norm.bound.eta.n} and the continuous embedding
		$\cH_w^{-k}\hookrightarrow \cH^{-k}$,
		\begin{align*}
			\E\bigg[\int_0^t \|\eta_s^n\|_{\cH^{-k}}^2\,\dd s\bigg] \lesssim 1.
		\end{align*}
		Therefore, by the Cauchy-Schwarz inequality,
		\begin{align} \label{eq:drift.diff.conv}
			\begin{split}
				&\E\bigg[\bigg|\int_0^t \langle \eta_s^n,\LL_{s,\mu_s^n}f-\LL_{s,\mu_s}f\rangle\,\dd s\bigg|\bigg] \\
				&\le
				\bigg(\E\bigg[\int_0^t\|\eta_s^n\|_{\cH^{-k}}^2\,\dd s\bigg]\bigg)^{1/2}
				\bigg(\E\bigg[\int_0^t\big\|\LL_{s,\mu_s^n}f-\LL_{s,\mu_s}f\big\|_{\cH^k}^2\,\dd s\bigg]\bigg)^{1/2}
				\to 0.
			\end{split}
		\end{align}

		\medskip\noindent\textit{Step 5.} We verify item \eqref{SPDE.def.weak} of Definition \ref{def:SPDE.limit} in this step. Fix $f\in C_c^\infty(\R^d)$ and $t\in[0,T]$. For $R>0$, let
		\begin{align*}
			g_s=\mathbf{1}_{[0,t]}(s)\,\LL_{s,\mu_s}f,\quad
			g_{R,s}=g_s\phi_R,\quad
			h_{R,s}=g_s(1-\phi_R),\quad s\in[0,T].
		\end{align*}
		For any family $\psi=(\psi_s)_{s\in[0,T]}$ such that the following expressions are well-defined, let
		\begin{align*}
			\J_t^n(\psi)
			&:=
			\langle \eta_t^n,f\rangle
			-\langle \eta_0^n,f\rangle
			-\int_0^t \langle \eta_s^n,\psi_s\rangle\,\dd s
			-\langle W_t^n,f\rangle,\\
			\J_t(\psi)
			&:=
			\langle \eta_t,f\rangle_{\cH^{-k},\cH^k}
			-\langle \eta_0,f\rangle_{\cH^{-k},\cH^k}
			-\int_0^t \langle \eta_s,\psi_s\rangle_{\cH^{-k}_w,\cH^k_w}\,\dd s
			-\langle W_t,f\rangle_{\cH^{-k-1},\cH^{k+1}}.
		\end{align*}
		Also,  let
		\begin{align*}
			\Lambda_t^n:=\J_t^n(g),\qquad
			\Lambda_t:=\J_t(g),\qquad
			\Gamma_{R,t}^n:=\J_t^n(g_R),\qquad
			\Gamma_{R,t}:=\J_t(g_R).
		\end{align*}
		
		For each $R>0$, Assumption \ref{assum.main}\eqref{assum.drift} implies that $g_{R,s}\in \cH^k$ for every $s\in[0,T]$ and $\sup_{s\in[0,T]}\|g_{R,s}\|_{\cH^k}<\infty$. Hence, by Lemma \ref{lem:sob}\eqref{weighted-unweighted-pair}, $\langle \eta_s,g_{R,s}\rangle_{\cH_w^{-k},\cH_w^k}
		=
		\langle \eta_s,g_{R,s}\rangle_{\cH^{-k},\cH^k}$ for a.e. $s \in [0,T]$. Therefore,
		\begin{align*}
			\Gamma_{R,t}
			=
			\langle \eta_t,f\rangle_{\cH^{-k},\cH^k}
			-\langle \eta_0,f\rangle_{\cH^{-k},\cH^k}
			-\int_0^t \langle \eta_s,g_{R,s}\rangle_{\cH^{-k},\cH^k}\,\dd s
			-\langle W_t,f\rangle_{\cH^{-k-1},\cH^{k+1}}.
		\end{align*}
		Moreover, the map
		\begin{align*}
			\cC\ni(\alpha,\beta)\mapsto
			\langle \alpha_t,f\rangle_{\cH^{-k},\cH^k}
			-\langle \alpha_0,f\rangle_{\cH^{-k},\cH^k}
			-\int_0^t \langle \alpha_s,g_{R,s}\rangle_{\cH^{-k},\cH^k}\,\dd s
			-\langle \beta_t,f\rangle_{\cH^{-k-1},\cH^{k+1}}
		\end{align*}
		is continuous on $	\cC$. Therefore, by \eqref{eq:joint.conv.eta.W}, we have $\Gamma_{R,t}^n\stackrel{d}{\to}\Gamma_{R,t}$.
		
		Now, using $g_s=g_{R,s}+h_{R,s}$, we obtain
		\begin{align*}
			\Lambda_t^n
			=
			\Gamma_{R,t}^n
			-\int_0^t\langle \eta_s^n,h_{R,s}\rangle\,\dd s,
			\qquad
			\Lambda_t
			=
			\Gamma_{R,t}
			-\int_0^t\langle \eta_s,h_{R,s}\rangle_{\cH^{-k}_w,\cH^k_w}\,\dd s.
		\end{align*}
		Let $\varphi:\R\to\R$ be bounded and $1$-Lipschitz. Then,
		\begin{align*}
			\big|\E[\varphi(\Lambda_t^n)]-\E[\varphi(\Lambda_t)]\big|
			 \le& \,
			\big|\E[\varphi(\Gamma_{R,t}^n)]-\E[\varphi(\Gamma_{R,t})]\big|
			+
			\E\bigg[\bigg|\int_0^t\langle \eta_s^n,h_{R,s}\rangle\,\dd s\bigg|\bigg]\\
			&+
			\E\bigg[\bigg|\int_0^t\langle \eta_s,h_{R,s}\rangle_{\cH^{-k}_w,\cH^k_w}\,\dd s\bigg|\bigg].
		\end{align*}
		For fixed $R>0$, the first term on the right-hand side converges to $0$ thanks to $\Gamma_{R,t}^n\stackrel{d}{\to}\Gamma_{R,t}$. Moreover, Assumption \ref{assum.main}\eqref{assum.drift} implies $g=(g_s)_{s\in[0,T]}\in L^\infty([0,T];C_b^{k}(\R^d))$. Hence the tail estimates \eqref{L1.conv.R} and \eqref{eq:eta.tail.L1}, proved in Steps 5 and 6 of Proposition \ref{prop.conv.int.eta}, apply to this choice of $g$. Letting $n\to\infty$ and then  $R\to\infty$, we obtain $\Lambda_t^n\stackrel{d}{\to}\Lambda_t$.
		
		Finally, by Lemma \ref{lem:eta.dym}, \eqref{eq:Wn.Hvalued}--\eqref{eq:derivative.delta.pairing}, \eqref{eq:second.drift.conv}--\eqref{eq:third.drift.conv} of Step 2, and \eqref{eq:drift.diff.conv} of Step 4, we already know that $\Lambda_t^n\stackrel{d}{\to}0$. Therefore, it follows that $\Lambda_t=0$ a.s. Since $t$ was arbitrary, applying the above argument to all rational $t\in[0,T]$ and using continuity of $t\mapsto \eta_t$, $t\mapsto W_t$, and $t\mapsto \int_0^t \langle \eta_s,\LL_{s,\mu_s}f\rangle_{\cH^{-k}_w,\cH^k_w}\,\dd s$, we conclude that
		\begin{align*}
			\langle \eta_t,f\rangle_{\cH^{-k},\cH^k}
			=
			\langle \eta_0,f\rangle_{\cH^{-k},\cH^k}
			+
			\int_0^t \big\langle \eta_s,\LL_{s,\mu_s}f\big\rangle_{\cH^{-k}_w,\cH^k_w}\,\dd s
			+
			\langle W_t,f\rangle_{\cH^{-k-1},\cH^{k+1}}
		\end{align*}
		for all $t\in[0,T]$, $\PP$-a.s. This is exactly item \eqref{SPDE.def.weak} of Definition \ref{def:SPDE.limit}, with initial condition $\zeta_0:=\eta_0$.
		
\noindent\medskip\noindent\textit{Step 6.}
It remains to verify item \eqref{SPDE.def.dym} of Definition \ref{def:SPDE.limit} and identify the joint law of $W$ and $\eta_0$. Fix $\ell,m\in\N$, $g_1,\dots,g_\ell\in C_c^\infty(\R^d)$, and $f_1,\dots,f_m\in C_c^\infty(\R^d)$. Define 
\begin{align*}
	&Y^{n,g}
	=
	\big(
	\langle \eta_0^n,g_1\rangle,\dots,
	\langle \eta_0^n,g_\ell\rangle
	\big),  \quad  M^{n,f}
	=
	\big(
	\langle W^n,f_1\rangle,\dots,
	\langle W^n,f_m\rangle
	\big), \\
	& Y^g
	=
	\big(
	\langle \eta_0,g_1\rangle,\dots,
	\langle \eta_0,g_\ell\rangle
	\big),
	\quad \text{and} \quad
	W^f
	=
	\big(
	\langle W,f_1\rangle,\dots,
	\langle W,f_m\rangle
	\big).
\end{align*}
Since  $(\eta^n,W^n)\stackrel{d}{\to}(\eta,W)$ by \eqref{eq:joint.conv.eta.W} of Step 1, the Continuous Mapping Theorem gives
\begin{align}  \label{eq:Yg.Wf.conv}
	(Y^{n,g},M^{n,f})
	\stackrel{d}{\to}
	(Y^g,W^f)
	\quad
	\text{in } \R^\ell\times C([0,T];\R^m).
\end{align}
For $i, j \in [m]$, the quadratic covariation of $\langle W^n,f_i\rangle$ and $\langle W^n,f_j\rangle$ is
\begin{align}\label{eq:An_ij.def}
	A_{ij}^n(t)
	:=
	\big\langle \langle W^n,f_i\rangle,\langle W^n,f_j\rangle\big\rangle_t
	=
	\sigma^2\int_0^t
	\big\langle \mu_s^n,\nabla f_i\cdot\nabla f_j\big\rangle\,\dd s,
\end{align}
and we write $A^n(t):=(A_{ij}^n(t))_{i,j\in[m]}$. Let $C_{ij}(t):=\sigma^2\int_0^{t}\big\langle \mu_r,\nabla f_i\cdot\nabla f_j\big\rangle\,\dd r$ and $C(t):=(C_{ij}(t))_{i,j\in[m]}$. Since $\nabla f_i\cdot\nabla f_j$ is bounded and Lipschitz, and $\sup_{s\in[0,T]}\E[\W_2^2(\mu_s^n,\mu_s)]\to0$ by Proposition \ref{lem:preliminaries}\eqref{lem:preliminaries:ii}, we obtain
\begin{align}\label{eq:brack.L1.conv}
	\E\Big[
	\sup_{t\in[0,T]}
	\big|A_{ij}^n(t)-C_{ij}(t)\big|
	\Big]
	\to 0.
\end{align}
We first identify the joint law of $Y^g$ and $W^f$. Let $\FF^n=(\F_t^n)_{t\in[0,T]}$ denote the augmented filtration generated by the initial positions $X^1_0, \dots, X^n_0$ and the Brownian motions $B^1, \dots, B^n$. Then $M^{n,f}$ is a continuous $\FF^n$-martingale with $M_0^{n,f}=0$, and $\eta_0^n$, hence also $Y^{n,g}$, is $\F_0^n$-measurable.
	
	Fix $q\in\N$, times $t_1,\dots,t_q\in[0,T]$, and vectors $\lambda_1,\dots,\lambda_q\in\R^m$. Let $H(s):=\sum_{r=1}^q\lambda_r\,\mathbf 1_{\{s\le t_r\}}$ and $N_t^n:=\int_0^t H(s)\cdot\dd M_s^{n,f}$. Then,
	\begin{align*}
		N_T^n
		=
		\sum_{r=1}^q\lambda_r\cdot\int_0^T\mathbf 1_{\{s\le t_r\}}\,\dd M_s^{n,f}
		=
		\sum_{r=1}^q\lambda_r\cdot M_{t_r}^{n,f}.
	\end{align*}
	Also, using $\mathbf 1_{\{s\le t_r\}}\mathbf 1_{\{s\le t_{r'}\}}=\mathbf 1_{\{s\le t_r\wedge t_{r'}\}}$ and $A_{ij}^n(0)=0$,
	\begin{align}\label{eq:Nn.bracket}
		v_n := 	\langle N^n\rangle_T
		= \int_0^T H(s)^\top\,\dd A^n(s)\,H(s) = 
		\sum_{r,r'=1}^q\lambda_r^\top A^n(t_r\wedge t_{r'})\lambda_{r'}.
	\end{align}
	Together with \eqref{eq:brack.L1.conv}, we have
	\begin{align} \label{eq:L1.conv.v_n.v}
		\E[|v_n - v|] \to 0, \quad \text{where} \quad v:=\sum_{r,r'=1}^q\lambda_r^\top C(t_r\wedge t_{r'})\lambda_{r'}.
	\end{align} 
	
	Moreover, letting $H_a(s)$ denote the $a$-th component of $H(s)$,  inserting the expression for $\dd A_{ij}^n(s)$ from \eqref{eq:An_ij.def} into \eqref{eq:Nn.bracket} gives
	\begin{align*}
		\langle N^n\rangle_T
		=
		\sigma^2\int_0^T
		\Big\langle
		\mu_s^n,\,
		\Big|\sum_{a=1}^m H_a(s)\,\nabla f_a\Big|^2
		\Big\rangle\,\dd s
		\le
		\sigma^2 T
		\sup_{s\in[0,T]}
		\Big\|\sum_{a=1}^m H_a(s)\,\nabla f_a\Big\|_{L^\infty}^2 \lesssim 1. 
	\end{align*}
	Therefore, $L_t^n:=\exp\!\big(iN_t^n+\tfrac12\langle N^n\rangle_t\big)$ is a  complex-valued continuous true $\FF^n$-martingale and $\E[L_T^n\mid\F_0^n]= L^n_0 = 1$. Since $e^{iN_T^n}=L_T^n\,e^{-v_n/2}$ and $v$ is deterministic,
	\begin{align}\label{eq:cond.diff}
		\E\big[e^{iN_T^n}\mid\F_0^n\big]-e^{-v/2}
		=
		\E\big[L_T^n\big(e^{-v_n/2}-e^{-v/2}\big)\,\big|\,\F_0^n\big].
	\end{align}
	As $x\mapsto e^{-x/2}$ is $\tfrac12$-Lipschitz on $[0,\infty)$ and $|L_T^n|$ is bounded, we have from \eqref{eq:cond.diff} and \eqref{eq:L1.conv.v_n.v},
	\begin{align}\label{eq:exp.o(1)}
		\E\big[\big|\E[e^{iN_T^n}\mid\F_0^n]-e^{-v/2}\big|\big]\lesssim\E[|v_n-v|]\to0. 
	\end{align} Let $\theta\in\R^\ell$. Since $Y^{n,g}$ is $\F_0^n$-measurable, \eqref{eq:exp.o(1)} implies
	\begin{align*}
		\E\big[e^{i\theta\cdot Y^{n,g}}e^{iN_T^n}\big]
		=
		\E\big[e^{i\theta\cdot Y^{n,g}}\,\E[e^{iN_T^n}\mid\F_0^n]\big]
		=
		e^{-v/2}\,\E\big[e^{i\theta\cdot Y^{n,g}}\big]+o(1).
	\end{align*}
Letting $n\to\infty$ and using \eqref{eq:Yg.Wf.conv}, we obtain
	\begin{align} \label{eq:Y.W.Y.M}
		\E\bigg[e^{i\theta\cdot Y^g}\exp\bigg(i\sum_{r=1}^q\lambda_r\cdot W_{t_r}^f\bigg)\bigg]
		=
		e^{-v/2}\,\E\big[e^{i\theta\cdot Y^g}\big].
\end{align}

\noindent\smallskip\noindent\textit{Step 6(a): Independence.}
	It follows from \eqref{eq:Y.W.Y.M} that $W^f$ is independent of $Y^g$. Since $\ell,m$ and the test functions were arbitrary,  $W$ is independent of $\eta_0$.

\noindent\smallskip\noindent\textit{Step 6(b): Covariance.}
	We next show that the covariance formula \eqref{def.SPDE.cov} holds for $W$. Note that setting $\theta=0$ in \eqref{eq:Y.W.Y.M} shows that $W^f$ has covariance
	\begin{align} \label{eq: cov.iden.W}
		\E\big[\langle W_t,f_i\rangle\langle W_s,f_j\rangle\big]=C_{ij}(s\wedge t), \quad i,j \in [m].
	\end{align}
	Since $m$ and $f_1,\dots,f_m\in C_c^\infty(\R^d)$ were arbitrary, \eqref{eq: cov.iden.W} gives the covariance formula for $W$ against all test functions in $C_c^\infty(\R^d)$. It remains to extend the covariance formula to $\cH^{k+1}$. By \eqref{eq:joint.conv.eta.W} of Step 1, $W\in C([0,T];\cH^{-k-1})$ a.s. Let $f_1,f_2\in\cH^{k+1}$, and choose $\{f_1^p\}_{p\in\N},\{f_2^p\}_{p\in\N}\subset C_c^\infty(\R^d)$ such that $f_i^p\to f_i$ in $\cH^{k+1}$ for $i=1,2$. Therefore, for each $t\in[0,T]$ and $i=1,2$,
	\begin{align} \label{eq:pairing.approx.as}
		\big\langle W_t,f_i^p\big\rangle_{\cH^{-k-1},\cH^{k+1}}
		\to
		\langle W_t,f_i\rangle_{\cH^{-k-1},\cH^{k+1}}
		\quad \text{a.s.}
	\end{align}
	Moreover, applying \eqref{eq: cov.iden.W} to the test function $f_i^p-f_i^q$ gives
	\begin{align} \label{eq:pairing.approx.L2}
		\E\Big[
		\big|
		\big\langle W_t,f_i^p-f_i^q\big\rangle_{\cH^{-k-1},\cH^{k+1}}
		\big|^2
		\Big]
		=
		\sigma^2\int_0^t
		\big\langle \mu_u,
		\big|\nabla\big(f_i^p-f_i^q\big)\big|^2
		\big\rangle\,\dd u.
	\end{align}
	Since $k>d/2$, the Sobolev Embedding Theorem, see, e.g., \cite[Theorem 4.12, Case A, with $\Omega=\R^d$, $n=d$, $p=2$, $j=1$, and $m=k+1$]{adams2003sobolev}, gives $\cH^{k+1}\hookrightarrow C_b^1(\R^d)$. Therefore, since $\mu_u$ is a probability measure, the right-hand side of \eqref{eq:pairing.approx.L2} is bounded by
	\begin{align*}
		\sigma^2T\|\nabla(f_i^p-f_i^q)\|_{L^\infty}^2
		\lesssim
		\sigma^2T\|f_i^p-f_i^q\|_{\cH^{k+1}}^2,
	\end{align*}
	which tends to $0$ as $p,q\to\infty$. Thus, $\{\langle W_t, f_i^p\rangle_{\cH^{-k-1},\cH^{k+1}}\}_{p\in\N}$ is a Cauchy sequence in $L^2(\PP)$. In conjunction with \eqref{eq:pairing.approx.as}, we have, for each $t\in[0,T]$ and $i=1,2$,
	\begin{align} \label{eq:pairing.approx.L2.limit}
		\big\langle W_t,f_i^p\big\rangle_{\cH^{-k-1},\cH^{k+1}}
		\to
		\langle W_t,f_i\rangle_{\cH^{-k-1},\cH^{k+1}}
		\quad \text{in } L^2(\PP).
	\end{align}
	Applying \eqref{eq: cov.iden.W} to $f_1^p,f_2^p$, and then passing to the limit as $p\to\infty$, we obtain \eqref{def.SPDE.cov}. Here the convergence of the left-hand side follows from \eqref{eq:pairing.approx.L2.limit} and the Cauchy-Schwarz inequality, while the convergence of the right-hand side follows from the Sobolev embedding $\cH^{k+1}\hookrightarrow C_b^1(\R^d)$.

\smallskip\noindent\textit{Step 6(c): Gaussianity.} It remains to show that $W$ has centered Gaussian finite-dimensional distributions when tested against functions in $\cH^{k+1}$. From \eqref{eq:Y.W.Y.M} at $\theta=0$, we already know this when the test functions belong to $C_c^\infty(\R^d)$. Fix $N\in\N$, $t_1,\dots,t_N\in[0,T]$, and $h_1,\dots,h_N\in\cH^{k+1}$. Choose $h_i^p\in C_c^\infty(\R^d)$ such that $h_i^p\to h_i$ in $\cH^{k+1}$ for each $i\in[N]$. By the same argument as in \eqref{eq:pairing.approx.L2.limit},
	\begin{align*}
		\big(
		\langle W_{t_1},h_1^p\rangle,\dots,
		\langle W_{t_N},h_N^p\rangle
		\big)
		\to
		\big(
		\langle W_{t_1},h_1\rangle,\dots,
		\langle W_{t_N},h_N\rangle
		\big) \quad \text{in } L^2(\PP),
	\end{align*}
	and hence in distribution. Since centered Gaussian laws are preserved under weak limits, the limiting vector is centered Gaussian. 
	\end{proof}

	\section{Uniqueness}\label{sec:unique}

In this section, we prove the pathwise uniqueness statement in Theorem \ref{thm:fluctuation.clt} by an energy inequality argument. Uniqueness for fluctuation SPDEs of this type is also proved in \cite[Theorem~5.19]{meleard1996asymptotic} by a duality argument, in a different setting with time-homogeneous and measure-dependent  coefficients. On the overlap with our setting, the argument requires the coefficients to lie in $C_b^{2\lambda_d+2}$, more regularity than the $b_0,b\in C_b^{k+1}$ of Assumption \ref{assum.main}\eqref{assum.drift}, which in the smallest admissible case $k=\lambda_d+2$ is $C_b^{\lambda_d+3}$, matching $C_b^{2\lambda_d+2}$ when $d=1$ and requiring less regularity when $d\ge2$.

To do so, we first collect some properties of the operator $\LL_{t,\mu_t}$ defined in \eqref{def.generator}. For simplicity of notation, we write
	\begin{align} \label{eq: operator.rewrite}
		\LL_{t,\mu_t}f
		&=
		\frac{\sigma^2}{2}\Delta f+\beta_t\cdot \nabla f+\K_t f,
	\end{align}
	where the function $\beta_t$ and the operator $\K_t$ are defined by
	\begin{align*}
		\beta_t(x)
		:=
		b_0(t,x)+\int_{\R^d} b(t,x,y)\,\mu_t(\dd y), \qquad 	\K_t f(x)
		:=
		\int_{\R^d} b(t,y,x)\cdot \nabla f(y)\,\mu_t(\dd y).
	\end{align*}

		\begin{proposition}\label{prop:L.bound}
		Suppose Assumption \ref{assum.main} holds with $k\ge \lambda_d$. Then, for each $t \in [0,T]$ and $m \in \{k-1, k+1\}$, the operator $\LL_{t,\mu_t}$, defined by \eqref{def.generator} on $C_c^\infty(\R^d)$, extends to a bounded linear operator from $\cH_w^{m+2}$ to $\cH_w^m$. Moreover, the operator norms are bounded uniformly in $t\in[0,T]$.
	\end{proposition}
	\begin{proof}
		Fix $m\in\{k-1,k+1\}$. Since $C_c^\infty(\R^d)$ is dense in $\cH_w^{m+2}$ (see Lemma \ref{lem:sob}\eqref{lem:sob:vii}), it suffices to prove the uniform estimate
		\begin{align*}
			\|\LL_{t,\mu_t}f\|_{\cH_w^m}
			\lesssim
			\|f\|_{\cH_w^{m+2}},
			\quad \text{for all } f\in C_c^\infty(\R^d)  \text{ and }t\in[0,T].
		\end{align*}
		
		Fix $f \in C^\infty_c(\R^d)$. For the Laplacian term on the right-hand side of \eqref{eq: operator.rewrite}, the definition of the $\cH_w^m$ norm gives $\|\Delta f\|_{\cH_w^m}
		\lesssim 
		\|f\|_{\cH_w^{m+2}}$.
		
		For the second term on the right-hand side of  \eqref{eq: operator.rewrite}, Assumption \ref{assum.main}\eqref{assum.drift} implies
		\begin{align*}
			\sup_{t\in[0,T]}\max_{0 \le |\bm{\alpha}|\le m}\sup_{x\in\R^d}
			|D^{\bm{\alpha}}\beta_t(x)|
			<
			\infty.
		\end{align*}
		Hence, by Leibniz rule,
		\begin{align*}
			\sup_{t\in[0,T]}
			\|\beta_t\cdot \nabla f\|_{\cH_w^m}
			\lesssim
			\|f\|_{\cH_w^{m+2}}.
		\end{align*}
		
		For the third term on the right-hand side of \eqref{eq: operator.rewrite}, first note that the weight $w$ in \eqref{eq:w.def} satisfies $|D^{\bm{\alpha}}w^{-1/2}|\lesssim w^{-1/2}$ for every $0\le |\bm{\alpha}|\le m+1$. Hence, for every $0\le |\bm{\alpha}|\le m+1$, 
		Leibniz's rule gives
		\begin{align}
			\sup_{i\in[d]}\big\|D^{\bm{\alpha}}(w^{-1/2}\partial_i f)\big\|_{L^2}
			&\lesssim
			\sup_{i\in[d]}\sum_{0\le \bm{\gamma}\le \bm{\alpha}}
			\big\|w^{-1/2}D^{\bm{\gamma}}\partial_i f\big\|_{L^2}  \lesssim
			\|f\|_{\cH_w^{m+2}}. \label{eq:weighted.Hm1.gradient}
		\end{align}
		Since $m\ge k-1$ and $k\ge \lambda_d$, we have $m+1\ge \lambda_d>d/2$. Therefore, the Sobolev Embedding Theorem, see, e.g., \cite[Theorem 4.12, Case A, with $\Omega=\R^d$, $n=d$, $p=2$, $j=0$, and $m$ therein taken to be $m+1$]{adams2003sobolev}, together with \eqref{eq:weighted.Hm1.gradient} summed over $0\le |\bm{\alpha}|\le m+1$, gives
		\begin{align}
			\|w^{-1/2}\nabla f\|_{L^\infty}
			&\lesssim
			\sup_{i\in[d]}\big\|w^{-1/2}\partial_i f\big\|_{\cH^{m+1}}  \lesssim
			\|f\|_{\cH_w^{m+2}}. \label{eq:weighted.linfty.gradient}
		\end{align}
		Moreover, Jensen's inequality, Assumption \ref{assum.main}\eqref{assum.drift}, and $\int_{\R^d}w(x)^{-1}\,\dd x<\infty$ give
		\begin{align*}
			\|\K_t f\|_{\cH_w^m}^2
			&=
			\sum_{0\le |\bm{\alpha}|\le m}
			\int_{\R^d}
			\bigg|\int_{\R^d}D^{\bm{\alpha}}b(t,y,\cdot)(x)\cdot \nabla f(y)\,\mu_t(\dd y)\bigg|^2 w(x)^{-1}\,\dd x
			\\
			& \lesssim 
			\int_{\R^d}|\nabla f(y)|^2\,\mu_t(\dd y).
		\end{align*}
		Hence, by Proposition \ref{lem:preliminaries}\eqref{lem:preliminaries:iv},
		\begin{align} \label{eq:Kf.bound}
			\|\K_t f\|_{\cH_w^m}^2 \lesssim 
			\|w^{-1/2}\nabla f\|_{L^\infty}^2
			\int_{\R^d}w(y)\,\mu_t(\dd y) \lesssim 		\|w^{-1/2}\nabla f\|_{L^\infty}^2.
		\end{align}
		Thus, by \eqref{eq:weighted.linfty.gradient}, $\|\K_t f\|_{\cH_w^m}\lesssim \|f\|_{\cH_w^{m+2}}$, uniformly in $t\in[0,T]$. Combining the three estimates yields
		\begin{align*}
			\|\LL_{t,\mu_t}f\|_{\cH_w^m}
			\le
			C\|f\|_{\cH_w^{m+2}},
		\end{align*}
		uniformly in $t\in[0,T]$, as desired.
	\end{proof}

	\begin{proposition}[Pathwise uniqueness]\label{prop.unique}
		Suppose Assumption \ref{assum.main} holds with $k \ge \lambda_d+2$. Let $(\eta,W)$ and $(\eta',W)$ be two $C([0,T];\cH^{-k}) \times C([0,T];\cH^{-(k+1)})$-valued solutions to the fluctuation SPDE, in the sense of Definition \ref{def:SPDE.limit}, both defined on the same filtered probability space $(\Omega,\F,\FF=(\F_t)_{t\in[0,T]},\PP)$. Assume, moreover, that $\eta_0=\eta_0'$, $\PP$-a.s. Then, $\eta=\eta'$, $\PP$-a.s.
	\end{proposition}
	
	\begin{proof}
		Let $\delta:=\eta-\eta' \in C([0,T];\cH^{-k}) $. Then,
		Definition \ref{def:SPDE.limit} yields $\delta\in L^2([0,T];\cH_w^{-(k-1)})$ and, using \eqref{eq: lower.higher} in Lemma \ref{lem:sob}\eqref{weighted-unweighted-pair}, we obtain, for every $f\in C_c^\infty(\R^d)$ and  $t\in[0,T]$,
		\begin{align} \label{diff.SPDE.dym}
			\langle \delta_t,f\rangle_{\cH^{-k},\cH^{k}}
			=
			\int_0^t \langle \delta_s,\LL_{s,\mu_s} f\rangle_{\cH_w^{-(k-1)},\cH_w^{k-1}}\,\dd s.
		\end{align}
		We divide the proof into three steps. We will repeatedly use some facts on weighted Sobolev spaces collected in Lemma \ref{lem:sob}. 
		
		\medskip\noindent\textit{Step 1.}
		We first prove that, for a.e. $t\in[0,T]$,
		\begin{align} \label{eq:energy.identity}
			\|\delta_t\|_{\cH_w^{-(k+1)}}^2
			=
			2\int_0^t
			\langle \delta_s,\LL_{s,\mu_s}\widetilde{\delta}_s\rangle_{\cH_w^{-(k-1)},\cH_w^{k-1}}\,\dd s,
		\end{align}
		where, for a.e. $s\in[0,T]$, Lemma \ref{lem:sob}\eqref{embed.k1.k2} gives $\delta_s\in \cH_w^{-(k+1)}$, and $\widetilde{\delta}_s\in \cH_w^{k+1}$ is the unique element given by the Riesz Representation Theorem such that for all $\phi \in \cH^{k+1}_w$,
		\begin{align} \label{eq:delta.representation}
			\langle \delta_s,\phi\rangle_{\cH_w^{-(k+1)},\cH_w^{k+1}}
			=
			\langle \widetilde{\delta}_s,\phi\rangle_{\cH_w^{k+1}}, \quad \text{and} \quad
			\big\|\widetilde{\delta}_s\big\|_{\cH_w^{k+1}}
			=
			\|\delta_s\|_{\cH_w^{-(k+1)}}.
		\end{align}
		
	To prove \eqref{eq:energy.identity}, by Lemma \ref{lem:sob}\eqref{denseness.smooth.in.neg}, applied with $k+1$ in place of $k$, we may choose an orthonormal basis $\{f_j\}_{j\in\N}$ of $\cH_w^{k+1}$ with $f_j\in C_c^\infty(\R^d)$ for each $j \in \N$. For a.e. $t \in [0,T]$ such that $\delta_t\in \cH_w^{-(k-1)}$, Lemma \ref{lem:sob}\eqref{weighted-unweighted-pair}, together with \eqref{eq:delta.representation}, gives
	\begin{align} \label{eq:coordinate.riesz}
		\langle \delta_t,f_j\rangle_{\cH^{-k},\cH^k}
		=
		\langle \delta_t,f_j\rangle_{\cH_w^{-(k-1)},\cH_w^{k-1}}
		=
		\langle \delta_t,f_j\rangle_{\cH_w^{-(k+1)},\cH_w^{k+1}}
		=
		\langle \widetilde{\delta}_t,f_j\rangle_{\cH_w^{k+1}}.
	\end{align}
		By \eqref{diff.SPDE.dym} and Proposition \ref{prop:L.bound}, the map $t\mapsto \langle \delta_t,f_j\rangle_{\cH^{-k},\cH^k}$ is absolutely continuous, with
		\begin{align} \label{eq:coordinate.derivative}
			\frac{\dd}{\dd t}\langle \delta_t,f_j\rangle_{\cH^{-k},\cH^k}
			=
			\langle \delta_t,\LL_{t,\mu_t}f_j\rangle_{\cH_w^{-(k-1)},\cH_w^{k-1}}
		\end{align}
		for a.e. $t\in[0,T]$. For a.e. $t$ such that $\delta_t\in \cH_w^{-(k+1)}$, Parseval's identity in $\cH_w^{k+1}$ and \eqref{eq:coordinate.riesz} give
		\begin{align}\label{parseval}
			\big\|\widetilde{\delta}_t\big\|_{\cH_w^{k+1}}^2
			=
			\sum_{j=1}^\infty \langle \delta_t,f_j\rangle_{\cH^{-k},\cH^k}^2, \quad \text{and} \quad 
			\widetilde{\delta}_t=\sum_{j=1}^\infty \langle \delta_t,f_j\rangle_{\cH^{-k},\cH^k}f_j \quad \text{in } \cH_w^{k+1}.
		\end{align}
		
		For $N\in\N$, let
		\begin{align*}
			\widetilde{\delta}_t^N:=\sum_{j=1}^N \langle \delta_t,f_j\rangle_{\cH^{-k},\cH^k}f_j, \quad \text{and} \quad
			E_N(t):=\sum_{j=1}^N \langle \delta_t,f_j\rangle_{\cH^{-k},\cH^k}^2.
		\end{align*}
		Then  from \eqref{parseval}, for a.e. $t \in [0,T]$, 
		\begin{align} \label{Parseval.conv}
			\widetilde{\delta}_t^N \to \widetilde{\delta}_t \quad \text{in }  \cH_w^{k+1}, \quad \text{and} \quad E_N(t) \to \|\delta_t\|^2_{\cH^{-(k+1)}_w}.
		\end{align}
		Moreover, $E_N$ is absolutely continuous. Hence, by \eqref{eq:coordinate.derivative} and $E_N(0)=0$, we have
		\begin{align} \label{eq:EN.identity}
			E_N(t)
			=
			2\int_0^t
			\langle \delta_s,\LL_{s,\mu_s}\widetilde{\delta}_s^N\rangle_{\cH_w^{-(k-1)},\cH_w^{k-1}}\,\dd s.
		\end{align}
		For a.e. $s \in [0,T]$, by  Proposition \ref{prop:L.bound} and \eqref{Parseval.conv},
		\begin{align}
			\begin{split}
			\big|\langle \delta_s,\LL_{s,\mu_s}\widetilde{\delta}_s^N\rangle_{\cH_w^{-(k-1)},\cH_w^{k-1}}
			-
			\langle \delta_s,\LL_{s,\mu_s}\widetilde{\delta}_s\rangle_{\cH_w^{-(k-1)},\cH_w^{k-1}}\big| &\le
			\|\delta_s\|_{\cH_w^{-(k-1)}}
			\big\|\LL_{s,\mu_s}(\widetilde{\delta}_s^N-\widetilde{\delta}_s)\big\|_{\cH_w^{k-1}} \\
			&\lesssim
			\|\delta_s\|_{\cH_w^{-(k-1)}}\big\|\widetilde{\delta}_s^N-\widetilde{\delta}_s\big\|_{\cH_w^{k+1}}
			\to 0.
			\end{split}
			\label{eq:EN.integrand.convergence}
		\end{align}
		Moreover, Proposition \ref{prop:L.bound}, the estimate $\|\widetilde{\delta}_s^N\|_{\cH_w^{k+1}}\le \|\widetilde{\delta}_s\|_{\cH_w^{k+1}}=\|\delta_s\|_{\cH_w^{-(k+1)}}$, and Lemma \ref{lem:sob}\eqref{embed.k1.k2} give
		\begin{align}
			\big|\langle \delta_s,\LL_{s,\mu_s}\widetilde{\delta}_s^N\rangle_{\cH_w^{-(k-1)},\cH_w^{k-1}}\big|
			&\le
			\|\delta_s\|_{\cH_w^{-(k-1)}}\big\|\LL_{s,\mu_s}\widetilde{\delta}_s^N\big\|_{\cH_w^{k-1}}
			\lesssim \|\delta_s\|_{\cH_w^{-(k-1)}}\big\|\widetilde{\delta}_s^N\big\|_{\cH_w^{k+1}}
			\lesssim \|\delta_s\|_{\cH_w^{-(k-1)}}^2. \label{eq:EN.integrand.domination}
		\end{align}
		Since $\delta\in L^2([0,T];\cH_w^{-(k-1)})$, the last expression is integrable on $[0,T]$. Letting $N\to\infty$ in \eqref{eq:EN.identity}, \eqref{Parseval.conv}, the Dominated Convergence Theorem, and \eqref{eq:EN.integrand.convergence}--\eqref{eq:EN.integrand.domination} yield \eqref{eq:energy.identity}.
		
		\medskip\noindent\textit{Step 2.} 
		We next show that, for a.e. $t\in[0,T]$,
		\begin{align} \label{eq:Gronwall.inequality}
			\frac{\dd}{\dd t}
			\int_0^t
			\langle \delta_s,\LL_{s,\mu_s}\widetilde{\delta}_s\rangle_{\cH_w^{-(k-1)},\cH_w^{k-1}}\,\dd s
			\lesssim 
			\int_0^t
			\langle \delta_s,\LL_{s,\mu_s}\widetilde{\delta}_s\rangle_{\cH_w^{-(k-1)},\cH_w^{k-1}}\,\dd s.
		\end{align}

		Fix a time $t\in[0,T]$ such that $\delta_t\in\cH_w^{-(k-1)}$. In particular,  $\delta_t\in\cH_w^{-(k+1)} $ and there exists $\widetilde{\delta}_t \in \cH^{k+1}_w$ satisfying  \eqref{eq:delta.representation}. By Lemma \ref{lem:sob}\eqref{denseness.smooth.in.neg}, we can choose $\varphi_p\in C_c^\infty(\R^d)$ such that if we  define $\delta_t^p\in \cH_w^{-(k-1)}$ by
		\begin{align*}
			\langle \delta_t^p,f\rangle_{\cH_w^{-(k-1)},\cH_w^{k-1}}
			=
			\int_{\R^d}\varphi_p(x)f(x)\,\dd x,
			\qquad f\in \cH_w^{k-1},
		\end{align*}
		we have
		\begin{align} \label{eq:delta.n.to.delta}
			\delta_t^p\to\delta_t \quad \text{in } \cH_w^{-(k-1)}.
		\end{align}
		Also, Lemma \ref{lem:sob}\eqref{embed.k1.k2} gives $\delta^p_t\in \cH_w^{-(k+1)}$. Therefore, by the Riesz Representation Theorem, there exists a unique element $\widetilde{\delta}_t^p\in \cH_w^{k+1}$ such that, for all $\phi\in \cH_w^{k+1}$,
		\begin{align} \label{eq: riesz.delta.n}
			\langle \delta_t^p,\phi\rangle_{\cH^{-(k+1)}_w,\cH^{k+1}_w}
			=
			\langle \widetilde{\delta}_t^p,\phi\rangle_{\cH^{k+1}_w},
			\quad  \text{and} \quad 
			\big\|\widetilde{\delta}_t^p\big\|_{\cH^{k+1}_w}
			=
			\big\|\delta_t^p\big\|_{\cH^{-(k+1)}_w}.
		\end{align}
		Moreover, by Lemma \ref{lem:sob}\eqref{embed.k1.k2}, we know $\delta_t^p\to\delta_t$  in $\cH_w^{-(k+1)}$ as well, so the isometry in the Riesz Representation Theorem gives
		\begin{align}\label{eq: tilde.delta.n.conv}
			\widetilde{\delta}_t^p\to\widetilde{\delta}_t \quad \text{in }  \cH_w^{k+1}.
		\end{align}
		Since $\delta_t^p$ is represented by $\varphi_p\in C_c^\infty(\R^d)$, Lemma \ref{lem:sob}\eqref{dual.smooth}, applied with $k+1$ in place of $k$, gives $\widetilde{\delta}_t^p\in \cH_w^{k+3}$. The estimates in the proof of Proposition \ref{prop:L.bound}, applied with $m=k+1$ and $f=\widetilde{\delta}_t^p$, show that the three terms in \eqref{eq: operator.rewrite} all belong to $\cH_w^{k+1}$. Therefore, Lemma \ref{lem:sob}\eqref{weighted-unweighted-pair} and \eqref{eq: riesz.delta.n} give
		\begin{align}
			\begin{split}
			&	\big\langle \delta_t^p,\LL_{t,\mu_t}\widetilde{\delta}_t^p\big\rangle_{\cH_w^{-(k-1)},\cH_w^{k-1}}
			=
			\big\langle \delta_t^p,\LL_{t,\mu_t}\widetilde{\delta}_t^p\big\rangle_{\cH^{-(k+1)}_w,\cH^{k+1}_w}
			=
			\big\langle \widetilde{\delta}_t^p,\LL_{t,\mu_t}\widetilde{\delta}_t^p\big\rangle_{\cH^{k+1}_w}\\
			& = \frac{\sigma^2}{2}
			\big\langle \widetilde{\delta}_t^p,\Delta \widetilde{\delta}_t^p\big\rangle_{\cH_w^{k+1}}
			+
			\big\langle \widetilde{\delta}_t^p,\beta_t\cdot\nabla \widetilde{\delta}_t^p\big\rangle_{\cH_w^{k+1}}
			+
			\big\langle \widetilde{\delta}_t^p,\K_t\widetilde{\delta}_t^p\big\rangle_{\cH_w^{k+1}}.
			\end{split}
			\label{eq:delta.hat.decom}
		\end{align}

		We estimate each of the three terms in \eqref{eq:delta.hat.decom}.  For the diffusion term, we claim that
		\begin{align} \label{eq:uniqueness.pf.diffusion.term.bound}
			\big\langle \Delta \widetilde{\delta}_t^p, \widetilde{\delta}_t^p\big\rangle_{\cH^{k+1}_w}
			+ \big\| \nabla \widetilde{\delta}_t^p \big\|^2_{\cH^{k+1}_w}	
			\lesssim  \big\| \widetilde{\delta}_t^p \big\|_{\cH^{k+1}_w}^2.
		\end{align}
		Here and below, for $f\in \cH_w^{k+2}$, we write
			\begin{align*}
				\|\nabla f\|_{\cH_w^{k+1}}^2
				:=
				\sum_{\ell=1}^d
				\|\partial_\ell f\|_{\cH_w^{k+1}}^2.
		\end{align*} To see this,   first note that
		\begin{align}  \label{eq: delta.p.delta.ibp}
			\big\langle \Delta \widetilde{\delta}_t^p, \widetilde{\delta}_t^p\big\rangle_{\cH^{k+1}_w}
			&=
			\sum_{0 \le |\bm{\alpha}|\le k+1}
			\int_{\R^d}
			D^{\bm{\alpha}} \widetilde{\delta}_t^p(x)\,\Delta D^{\bm{\alpha}} \widetilde{\delta}_t^p(x)\, w^{-1}(x)\,\dd x.
		\end{align} 
		Fix $\bm{\alpha}$ with $0 \le |\bm{\alpha}|\le k+1$ and let $u:=D^{\bm{\alpha}}\widetilde{\delta}_t^p$. Since $\widetilde{\delta}_t^p\in \cH^{k+3}_w$, we have $u\in \cH^2_w$. By Lemma \ref{lem:sob}\eqref{denseness.smooth.in.neg}, applied with $k=2$, we can find $u_m\in C_c^\infty(\R^d)$ such that $u_m\to u$ in $\cH_w^2$. For each $m$, integration by parts  gives
		\begin{align*}
			\int_{\R^d} u_m\,\Delta u_m\,w^{-1}\,\dd x
			=
			-\int_{\R^d}|\nabla u_m|^2\,w^{-1}\,\dd x
			+\frac12\int_{\R^d}u_m^2\,\Delta w^{-1}\,\dd x.
		\end{align*}
		Letting $m \to \infty$ and using $u_m\to u$, $\nabla u_m\to \nabla u$, and $\Delta u_m\to \Delta u$ in $L^2(w^{-1}(x) \, \dd x)$, together with the fact that the weight $w$ in \eqref{eq:w.def} satisfies  $|\Delta w^{-1}(x)|\lesssim w^{-1}(x)$,  we have
		\begin{align*}
			\int_{\R^d} u\,\Delta u\,w^{-1}\,\dd x
			=
			-\int_{\R^d}|\nabla u|^2\,w^{-1}\,\dd x
			+\frac12\int_{\R^d}u^2\,\Delta w^{-1}\,\dd x.
		\end{align*}
		Summing this over  $0 \le |\bm{\alpha}|\le k+1$ and using \eqref{eq: delta.p.delta.ibp} gives
		\begin{align*} 
			\big\langle \Delta \widetilde{\delta}_t^p, \widetilde{\delta}_t^p\big\rangle_{\cH^{k+1}_w}
			=&
			-\sum_{0 \le |\bm{\alpha}|\le k+1}
			\int_{\R^d}
			|D^{\bm{\alpha}} \nabla \widetilde{\delta}_t^p|^2\,w^{-1} \,\dd x\\
			&+
			\frac{1}{2}\sum_{0 \le |\bm{\alpha}|\le k+1}
			\int_{\R^d}
			|D^{\bm{\alpha}}\widetilde{\delta}_t^p|^2\,\Delta w^{-1}\,\dd x. \notag
		\end{align*}
		Using again $|\Delta w^{-1}(x)|\lesssim w^{-1}(x)$, we obtain \eqref{eq:uniqueness.pf.diffusion.term.bound}.
		
		For the second term on the right-hand side of \eqref{eq:delta.hat.decom}, Assumption \ref{assum.main}\eqref{assum.drift} implies that $\beta_t$ and its spatial derivatives up to order $k+1$ are bounded uniformly in $t$. Therefore,
		\begin{align} \label{eq: bound.beta}
			\big|\langle \beta_t\cdot\nabla\widetilde{\delta}_t^p,\widetilde{\delta}_t^p\rangle_{\cH_w^{k+1}}\big|
			\le
			\big\|\beta_t\cdot\nabla\widetilde{\delta}_t^p\big\|_{\cH_w^{k+1}}\big\|\widetilde{\delta}_t^p\big\|_{\cH_w^{k+1}} 	\lesssim \big\|\nabla\widetilde{\delta}_t^p\big\|_{\cH_w^{k+1}}\big\|\widetilde{\delta}_t^p\big\|_{\cH_w^{k+1}}.
		\end{align}

		For the third term on the right-hand side of \eqref{eq:delta.hat.decom}, we first use the following weighted Sobolev estimate. Since $k+1\ge \lambda_d+1>d/2$, the Sobolev Embedding Theorem, see, e.g., \cite[Theorem 4.12, Case A, with $\Omega=\R^d$, $n=d$, $p=2$, $j=0$, and  $m = k+1$]{adams2003sobolev}, together with Leibniz's rule and the bounds $|D^{\bm{\alpha}}w^{-1/2}|\lesssim w^{-1/2}$ , gives
		\begin{align}\label{eq: weighted.sob.est}
			\big\|w^{-1/2}\nabla f\big\|_{L^\infty}
			\lesssim
			\|\nabla f\|_{\cH_w^{k+1}},
			\qquad f\in \cH_w^{k+2}.
		\end{align}
		Next, note that although \eqref{eq:Kf.bound} in the proof of Proposition \ref{prop:L.bound} was first proved for $f\in C_c^\infty(\R^d)$, it extends by a density argument to the present choice $f=\widetilde{\delta}_t^p\in \cH_w^{k+3}$. Indeed, if $f_\ell\in C_c^\infty(\R^d)$ and $f_\ell\to f$ in $\cH_w^{k+3}$, then \eqref{eq: weighted.sob.est} implies
		\begin{align*}
			\big\|w^{-1/2}\nabla(f_\ell-f)\big\|_{L^\infty}
			\lesssim
			\|\nabla(f_\ell-f)\|_{\cH_w^{k+1}}
			\to 0,
		\end{align*}
		which, together with the continuity of $\K_t:\cH_w^{k+3}\to\cH_w^{k+1}$ from the proof of Proposition \ref{prop:L.bound}, allows us to pass to the limit in \eqref{eq:Kf.bound}. Therefore, applying \eqref{eq:Kf.bound} with $m=k+1$ and $f=\widetilde{\delta}_t^p$, and then using \eqref{eq: weighted.sob.est}, yields
		\begin{align} \label{eq:bound.K}
			\big|\langle \K_t\widetilde{\delta}_t^p,\widetilde{\delta}_t^p\rangle_{\cH_w^{k+1}}\big|
			\le
			\big\|\K_t\widetilde{\delta}_t^p\big\|_{\cH_w^{k+1}}\big\|\widetilde{\delta}_t^p\big\|_{\cH_w^{k+1}}
			\lesssim
			\big\|\nabla\widetilde{\delta}_t^p\big\|_{\cH_w^{k+1}}\big\|\widetilde{\delta}_t^p\big\|_{\cH_w^{k+1}}.
		\end{align}
		
		Putting the three bounds \eqref{eq:uniqueness.pf.diffusion.term.bound}, \eqref{eq: bound.beta} and \eqref{eq:bound.K} into \eqref{eq:delta.hat.decom}, we find that there exists $C<\infty$, independent of $t $ and $p$,  such that
			\begin{align} \label{eq:pre.Young.ineq}
				\big\langle \delta_t^p, \LL_{t,\mu_t} \widetilde{\delta}_t^p \big\rangle_{\cH_w^{-(k-1)}, \cH_w^{k-1}}
				+
				\frac{\sigma^2}{2}\big\| \nabla \widetilde{\delta}^p_t \big\|^2_{\cH^{k+1}_w}
				\le
				C\big\| \widetilde{\delta}^p_t \big\|_{\cH^{k+1}_w}^2
				+
				C\big\|\widetilde{\delta}^p_t\big\|_{\cH^{k+1}_w}\,\big\| \nabla \widetilde{\delta}_t^p\big\|_{\cH^{k+1}_w}.
		\end{align}
		Young's inequality gives,  for some constant $C_\sigma<\infty$, independent of $p $ and $t$, 
			\begin{align}\label{eq:Young.cross.term}
				C\big\|\widetilde{\delta}_t^p\big\|_{\cH_w^{k+1}}
				\big\|\nabla\widetilde{\delta}_t^p\big\|_{\cH_w^{k+1}}
				\le
				\frac{\sigma^2}{4}
				\big\|\nabla\widetilde{\delta}_t^p\big\|_{\cH_w^{k+1}}^2
				+
				C_\sigma
				\big\|\widetilde{\delta}_t^p\big\|_{\cH_w^{k+1}}^2.
		\end{align}
		Using \eqref{eq:Young.cross.term} in \eqref{eq:pre.Young.ineq}, 
		\begin{align} \label{eq:quadratic.fonr.inequ}
			\big\langle \delta_t^p, \LL_{t,\mu_t} \widetilde{\delta}_t^p \big\rangle_{\cH_w^{-(k-1)}, \cH_w^{k-1}}
			\le
			(C + C_\sigma)
			\big\|\widetilde{\delta}_t^p\big\|_{\cH_w^{k+1}}^2
			=
			(C + C_\sigma)
			\|\delta_t^p\|_{\cH_w^{-(k+1)}}^2.
		\end{align}
		From  \eqref{eq: tilde.delta.n.conv} and Proposition \ref{prop:L.bound}, we have $\LL_{t,\mu_t}\widetilde{\delta}_t^p\to\LL_{t,\mu_t}\widetilde{\delta}_t$ in $\cH_w^{k-1}$. Combining this with \eqref{eq:delta.n.to.delta}, and using Lemma \ref{lem:sob}\eqref{embed.k1.k2}, we get
		\begin{align*}
			\big\langle \delta_t^p,\LL_{t,\mu_t}\widetilde{\delta}_t^p\big\rangle_{\cH_w^{-(k-1)},\cH_w^{k-1}}
			\to
			\big\langle \delta_t,\LL_{t,\mu_t}\widetilde{\delta}_t\big\rangle_{\cH_w^{-(k-1)},\cH_w^{k-1}},
			\quad \text{and} \quad
			\|\delta_t^p\|_{\cH^{-(k+1)}_w}
			\to
			\|\delta_t\|_{\cH^{-(k+1)}_w}.
		\end{align*}
		Therefore, letting $p \to \infty$ in \eqref{eq:quadratic.fonr.inequ}, gives
		\begin{align} \label{eq:three.term.bound}
			\big\langle \delta_t,\LL_{t,\mu_t}\widetilde{\delta}_t\big\rangle_{\cH_w^{-(k-1)},\cH_w^{k-1}}
			\le
			(C + C_\sigma)
			\|\delta_t\|_{\cH^{-(k+1)}_w}^2.
		\end{align}
		Since the constants above are independent of $t$ and $p$, and since the fixed time $t$ was arbitrary in a set of full measure, \eqref{eq:three.term.bound} holds for a.e. $t \in [0,T]$. Combining this with \eqref{eq:energy.identity} of Step 1, and using the absolute continuity of the integral in \eqref{eq:energy.identity}, yields \eqref{eq:Gronwall.inequality}.

		\medskip\noindent\textit{Step 3.} 
		By \eqref{eq:energy.identity} of Step 1, the absolutely continuous function
		\begin{align*}
			t\mapsto
			\int_0^t
			\big\langle \delta_s,\LL_{s,\mu_s}\widetilde{\delta}_s\big\rangle_{\cH_w^{-(k-1)},\cH_w^{k-1}}\,\dd s
		\end{align*}
		agrees a.e. with $\frac12\|\delta_t\|_{\cH_w^{-(k+1)}}^2$. Hence it is nonnegative, and it vanishes at $t=0$. Gronwall's inequality applied to \eqref{eq:Gronwall.inequality} of Step 2 gives
		\begin{align*}
			\int_0^t
			\big\langle \delta_s,\LL_{s,\mu_s}\widetilde{\delta}_s\big\rangle_{\cH_w^{-(k-1)},\cH_w^{k-1}}\,\dd s
			=0,
			\qquad t\in[0,T].
		\end{align*}
		It follows from \eqref{eq:energy.identity} that $\|\delta_t\|_{\cH_w^{-(k+1)}}=0$, and hence $\delta_t=0$ in $\cH_w^{-(k+1)}$, for a.e. $t\in[0,T]$. Since $\delta_t\in \cH_w^{-(k-1)}$ for a.e. $t\in[0,T]$, and Lemma \ref{lem:sob}\eqref{embed.k1.k2} gives the embeddings $\cH_w^{-(k-1)}\hookrightarrow \cH_w^{-(k+1)}$ and $\cH_w^{-(k-1)}\hookrightarrow \cH^{-k}$, the injectivity of the first embedding implies $\delta_t=0$ in $\cH_w^{-(k-1)}$, for a.e. $t\in[0,T]$. Applying the second embedding then gives $\delta_t=0$ in $\cH^{-k}$, for a.e. $t\in[0,T]$. Since $t\mapsto \delta_t$ is continuous as an $\cH^{-k}$-valued map, it follows that $\delta_t=0$ in $\cH^{-k}$ for all $t\in[0,T]$. Therefore $\eta_t=\eta_t'$ for all $t\in[0,T]$, which proves pathwise uniqueness.
	\end{proof}

	\section{Proof of Theorem \ref{thm:fluctuation.clt}}\label{sec:proof.main}
	We now combine the tightness result from Section \ref{sec:tightness}, the convergence result from Section \ref{sec:conv}, and the uniqueness result from Section \ref{sec:unique} to prove Theorem \ref{thm:fluctuation.clt}. 
	\begin{proof}[{\bf Proof of Theorem \ref{thm:fluctuation.clt}}]
		We first prove part \eqref{main.thm.part.0}. Since $k>d/2$, the Sobolev Embedding Theorem, see, e.g., \cite[Theorem 4.12, Case A, with $\Omega=\R^d$, $n=d$, $p=2$, $j=0$, and $m=k$]{adams2003sobolev}, gives $\cH^k \hookrightarrow C_b(\R^d)$. Thus, for every $x\in\R^d$, the Dirac mass $\delta_x$ belongs to $\cH^{-k}$,  and
		\begin{align*}
			\sup_{x\in\R^d}\|\delta_x\|_{\cH^{-k}}\lesssim 1.
		\end{align*} 
		Consequently, the random variables $Z_i:=\delta_{X_0^i}-\mu_0$ are centered i.i.d. square-integrable $\cH^{-k}$-valued random variables. By the Central Limit Theorem for square-integrable random variables in separable Hilbert spaces, see, e.g., \cite[Subsection 10.2]{ledoux1991probability} for a general Banach-space version, we have part \eqref{main.thm.part.0}.
		
		We next prove parts \eqref{main.thm.part.1} and \eqref{main.thm.part.2}. By Proposition \ref{prop:tight}, the sequence $(\eta^n)_{n\in\N}$ is tight in $C([0,T];\cH^{-k})$. Let $\eta$ be the limit in law of an arbitrary convergent subsequence, and let $\eta_0:=\eta(0)$. Proposition \ref{prop:convergence} then yields a process $W\in C([0,T];\cH^{-(k+1)})$ such that $(\eta,W)$ is a solution of the fluctuation SPDE in the sense of Definition \ref{def:SPDE.limit}, with initial condition $\zeta_0:=\eta_0$. Along the same subsequence, the initial values converge in law to $\eta_0$, while part \eqref{main.thm.part.0} gives their convergence in law to $\eta_0^\ast$. Hence, by uniqueness of limits in distribution, $\eta_0\stackrel{d}{=}\eta_0^\ast$. Moreover, Proposition \ref{prop:convergence} states that $W$ is independent of $\eta_0$ and is a centered Gaussian process in $C([0,T];\cH^{-(k+1)})$ satisfying the covariance formula in Definition \ref{def:SPDE.limit}\eqref{SPDE.def.dym}. Since centered Gaussian laws are determined by their covariance, it follows that, with $W^\ast$ as in the statement of part \eqref{main.thm.part.2}, we have $\Law(W)=\Law(W^\ast)$. The independence just noted then gives $\Law(\eta_0,W)=\Law(\eta_0^\ast)\otimes\Law(W^\ast)$. This gives the existence assertion in \eqref{main.thm.part.1}. The pathwise uniqueness assertion in \eqref{main.thm.part.1} is exactly Proposition \ref{prop.unique}, and the uniqueness in law assertion in \eqref{main.thm.part.1} follows from Remark \ref{rmk:unique.in.law}. Thus, \eqref{main.thm.part.1} is proved.

		Since the preceding argument applies to the limit of any convergent subsequence, every subsequential limit is the first component of a solution pair $(\eta,W)$ with initial condition $\eta_0$ such that $\Law(\eta_0,W)=\Law(\eta_0^\ast)\otimes \Law(W^\ast)$. By the uniqueness-in-law statement in part \eqref{main.thm.part.1}, once the joint law of the initial condition and the driving process is fixed, the law of the first component is fixed. Hence, all subsequential limits of $(\eta^n)_{n\in\N}$ have the same law. Together with tightness, this implies $\eta^n\stackrel{d}{\to}\eta$ in $C([0,T];\cH^{-k})$, where $\eta$ is the unique-in-law first component, in the sense of part \eqref{main.thm.part.1}, of any solution pair $(\eta,W)$ with initial condition $\eta_0$ such that $\Law(\eta_0,W)=\Law(\eta_0^\ast)\otimes \Law(W^\ast)$. This completes the proof of \eqref{main.thm.part.2}.
	\end{proof}

	\appendix

	\section{Matrix relations}
	
	\subsection{Proof of Remark \ref{remark.sufficient.matrix}} \label{sec:proof.remark}
	
	When $\xi$ is symmetric, \eqref{eq:column.sum.assump} follows from \eqref{eq:row.sum.assump}. For the remaining assertions, we do not assume symmetry. Recalling the definition of $\hxi$ in \eqref{hatxi.def}, and using the nonnegativity of $\xi$ and \eqref{eq:row.sum.assump}, we have
	\begin{align} \label{eq:sum.abs.xi.hat}
		\sum_{i,j=1}^n |\hxi_{ij}|
		\le (n-1)\sum_{i,j=1}^n \xi_{ij}+n^2
		\lesssim n^2.
	\end{align}
	Therefore, by \eqref{eq:sufficient.matrix.max},
	\begin{align*}
		\frac{1}{n^{3/2}}\Big(\sum_{i,j=1}^n|\hxi_{ij}|\Big)\max_{i,j\in[n]}\xi_{ij}
		\lesssim
		\sqrt{n}\, \max_{i,j\in[n]}\xi_{ij}
		\to 0.
	\end{align*}
	This proves \eqref{assum:mat.3}.
	
	It remains to prove \eqref{eq:xi_row_sq}. In fact, the stronger convergence to zero holds. Using the nonnegativity of $\xi$, \eqref{eq:row.sum.assump}, and the uniform column-sum bound in \eqref{eq:column.sum.assump}, we have 
	\begin{align*}
		\sum_{i=1}^n\Big(\sum_{j=1}^n(\xi_{ij}^2+\xi_{ji}^2)\Big)^2
		\lesssim 
		\max_{r,s\in[n]}\xi_{rs}^2
		\sum_{i=1}^n\Big(1+\sum_{j=1}^n \xi_{ji}\Big)^2
		\lesssim 
		n\max_{r,s\in[n]}\xi^2_{rs}
		\to 0,
	\end{align*}
	where the last step follows from \eqref{eq:sufficient.matrix.max}. Thus, the left-hand side of \eqref{eq:xi_row_sq}  tends to zero. \hfill \qed
	
	\subsection{Consequences of Assumption \ref{assum.main}\eqref{assum:mat}} 
	
	Recall the definition of $\hxi$ in \eqref{hatxi.def}.
	
	\begin{lemma}\label{lem:assum_i_implies_mat}
		Assumption \ref{assum.main}\eqref{assum:mat} implies, as $n\to\infty$,
		\begin{align}
			&\frac{1}{n}\sum_{i,j=1}^n \xi_{ij}^2 =o(n^{-1/2}), \label{eq:xi_sq_avg}\\
			&\frac{1}{n^3}\sum_{i,j=1}^n \hxi_{ij}^2  =o(n^{-1/2}), \label{eq:xi_hat_F}\\
			&	\frac{1}{n^3} \sum_{i=1}^n \bigg(\sum_{j=1}^n \hxi_{ji}\bigg)^2 =o(1), \label{eq:xi_hat_col_sum}\\
			&	\frac{1}{n^3} \max_{i,j\in[n]}\xi_{ij}
			\sum_{i=1}^n
			\bigg[
			\bigg(\sum_{j=1}^n|\hxi_{ij}|\bigg)^2
			+
			\bigg(\sum_{j=1}^n|\hxi_{ji}|\bigg)^2
			\bigg] =o(n^{-1/2}). \label{eq:xi_hat_l1} 
		\end{align}
	\end{lemma}
	
	\begin{proof}
		We first prove \eqref{eq:xi_hat_F}. From the nonnegativity of $\xi$,
		\begin{align*}
			|\hxi_{ij}|
			\le (n-1)\xi_{ij}+1
			\le n\max_{k,\ell\in[n]}\xi_{k\ell}+1.
		\end{align*}
		Therefore,
		\begin{align*}
			\frac{1}{n^3}\sum_{i,j=1}^n \hxi_{ij}^2
			\le
			\frac{1}{n^2} \max_{k,\ell\in[n]}\xi_{k\ell} \sum_{i,j=1}^n|\hxi_{ij}|
			+
			\frac{1}{n^3}\sum_{i,j=1}^n|\hxi_{ij}|.
		\end{align*}
		The right-hand side is $o(n^{-1/2})$ by \eqref{assum:mat.3} and \eqref{eq:sum.abs.xi.hat}. 
			This proves \eqref{eq:xi_hat_F}.
		
		We now prove \eqref{eq:xi_sq_avg}. Since $\hxi_{ij}=(n-1)\xi_{ij}-1$ and, by \eqref{eq:row.sum.assump}, $\sum_{i,j=1}^n\hxi_{ij}=-n$,
		\begin{align*}
			\frac{1}{n}\sum_{i,j=1}^n \xi_{ij}^2
			&=
			\frac{1}{n(n-1)^2}
			\sum_{i,j=1}^n(\hxi_{ij}+1)^2=
			\frac{1}{n(n-1)^2}\sum_{i,j=1}^n\hxi_{ij}^2
			+
			\frac{n^2-2n}{n(n-1)^2}.
		\end{align*}
		By \eqref{eq:xi_hat_F}, the first term on the right-hand side is $o(n^{-1/2})$, while the second term is $O(n^{-1})=o(n^{-1/2})$. This proves \eqref{eq:xi_sq_avg}.
		
		We now prove \eqref{eq:xi_hat_col_sum}. Using \eqref{eq:row.sum.assump} and the second assertion of \eqref{eq:column.sum.assump}, we obtain
		\begin{align*}
			0&\le
			\limsup_{n\to\infty}
			\frac{1}{n^3}\sum_{i=1}^n\bigg(\sum_{j=1}^n\hxi_{ji}\bigg)^2
			\\&=
			\limsup_{n\to\infty}
			\bigg[
			\Big(1-\frac1n\Big)^2
			\frac{1}{n}\sum_{i=1}^n\bigg(\sum_{j=1}^n\xi_{ji}\bigg)^2
			-
			2\Big(1-\frac1n\Big)
			\frac{1}{n}\sum_{i,j=1}^n\xi_{ji}
			+1
			\bigg]
			\le 1-2+1
			=0.
		\end{align*}
		This proves \eqref{eq:xi_hat_col_sum}.
		
		It remains to prove \eqref{eq:xi_hat_l1}. For each $i \in [n]$, using the nonnegativity of $\xi$ and \eqref{eq:row.sum.assump},
		\begin{align*}
			\sum_{j=1}^n|\hxi_{ij}|
			\le
			(n-1)\sum_{j=1}^n\xi_{ij}+n
			\lesssim n.
		\end{align*}
		Similarly, for each $i \in [n]$, using the nonnegativity of $\xi$ and \eqref{eq:column.sum.assump},
		\begin{align*}
			\sum_{j=1}^n|\hxi_{ji}|
			\le
			(n-1)\sum_{j=1}^n\xi_{ji}+n
			\lesssim n.
		\end{align*}
		Therefore,
		\begin{align*}
			\sum_{i=1}^n
			\bigg[
			\bigg(\sum_{j=1}^n|\hxi_{ij}|\bigg)^2
			+
			\bigg(\sum_{j=1}^n|\hxi_{ji}|\bigg)^2
			\bigg]
			\lesssim 
			n\sum_{i,j=1}^n|\hxi_{ij}|.
		\end{align*}
		Consequently, by \eqref{assum:mat.3},
		\begin{align*}
			\frac{1}{n^3}\max_{i,j\in[n]}\xi_{ij}
			\sum_{i=1}^n
			\bigg[
			\bigg(\sum_{j=1}^n|\hxi_{ij}|\bigg)^2
			+
			\bigg(\sum_{j=1}^n|\hxi_{ji}|\bigg)^2
			\bigg]
			\lesssim 
			\frac{1}{n^{2}}
			\bigg(\sum_{i,j=1}^n|\hxi_{ij}|\bigg)\max_{i,j\in[n]}\xi_{ij}
			=o(n^{-1/2}).
		\end{align*}
		This proves \eqref{eq:xi_hat_l1}.
	\end{proof}

	\section{Proof of \eqref{eq:Hminus.kw.characterization}} \label{sec:pf.hminus-kw}
	
	\begin{proof}

		Recall the definitions of Fourier transforms and Bessel potentials  in Sections \ref{sec:Fourier}--\ref{subsubsec:bessel}.
		
		\medskip\noindent\textit{Step 1.} 
		Let $m\in C^\infty(\R^d; \mathbb{C})$ be such that $D^{\bm{\beta}}m$ is bounded for every multi-index $\bm{\beta}$ with $0\le |\bm{\beta}|\le \lambda_d$.  Let  
		\begin{align*}
			T_m f:=\F^{-1}[m\F f], \quad f \in \SS(\R^d).
		\end{align*}
		We first prove the estimate
		\begin{align}\label{eq:multiplier.minus.N.one.sided}
			\|T_m f\|_{L^2((1+|x|^2)^{-\lambda_d}\,\dd x)}
			\lesssim 
			\|f\|_{L^2((1+|x|^2)^{-\lambda_d}\,\dd x)},
			\quad f\in\SS(\R^d).
		\end{align}
		To see this,  note that
		\begin{align}\label{eq:plus.N.polynomial.equivalence.one.sided}
			\|f\|_{L^2((1+|x|^2)^{\lambda_d}\,\dd x)}^2
			\asymp
			\sum_{0\le |\bm{\beta}|\le \lambda_d}
			\|x^{\bm{\beta}}f\|_{L^2}^2.
		\end{align}
		For every $\bm{\beta}$ with $0\le |\bm{\beta}|\le \lambda_d$, using $\F[x^{\bm{\beta}}g]=(-2\pi i)^{-|\bm{\beta}|}D^{\bm{\beta}}\F g$ and  Leibniz rule, we have
		\begin{align*}
			\F[x^{\bm{\beta}}T_m f](u)
			=
			\sum_{ 0\le \bm{\gamma}\le\bm{\beta}}
			\binom{\bm{\beta}}{\bm{\gamma}}
			(-2\pi i)^{-|\bm{\gamma}|}
			D^{\bm{\gamma}}m(u)
			\F[x^{\bm{\beta}-\bm{\gamma}}f](u), 
		\end{align*}
		where $\binom{\bm{\beta}}{\bm{\gamma}}:=\prod_{j=1}^d\binom{\beta_j}{\gamma_j}$. 
		Taking the inverse Fourier transform gives
		\begin{align}\label{eq:commutator.multiplier.one.sided}
			x^{\bm{\beta}}T_m f
			=
			\sum_{ 0\le \bm{\gamma}\le\bm{\beta}}
			\binom{\bm{\beta}}{\bm{\gamma}}
			(-2\pi i)^{-|\bm{\gamma}|}
			T_{D^{\bm{\gamma}}m}
			(x^{\bm{\beta}-\bm{\gamma}}f).
		\end{align}
		Since $D^{\bm{\gamma}}m$ is bounded, by Plancherel's Theorem,
		\begin{align}  \label{eq:plancheral.bound}
			\|T_{D^{\bm{\gamma}}m}f\|_{L^2}
			=
			\big\|D^{\bm{\gamma}}m\,\F f\big\|_{L^2} \lesssim 
			\|\F f\|_{L^2} = \|f\|_{L^2}.
		\end{align}
		Hence, applying \eqref{eq:plancheral.bound} to \eqref{eq:commutator.multiplier.one.sided}, we have
		\begin{align}\label{eq:commutator.bound.one.sided}
			\|x^{\bm{\beta}}T_m f\|_{L^2}
			\lesssim 
			\sum_{0\le \bm{\gamma}\le\bm{\beta}}
			\|x^{\bm{\beta}-\bm{\gamma}}f\|_{L^2}
			\lesssim 
			\|f\|_{L^2((1+|x|^2)^{\lambda_d}\,\dd x)}.
		\end{align}
		Using \eqref{eq:plus.N.polynomial.equivalence.one.sided} and  \eqref{eq:commutator.bound.one.sided}, we get
		\begin{align}\label{eq:multiplier.plus.N.one.sided}
			\|T_m f\|^2_{L^2((1+|x|^2)^{\lambda_d}\,\dd x)}
			\lesssim  \sum_{0\le |\bm{\beta}|\le \lambda_d}
			\|x^{\bm{\beta}}T_m f\|_{L^2}^2 \lesssim 
			\|f\|^2_{L^2((1+|x|^2)^{\lambda_d}\,\dd x)}.
		\end{align}
		For $f,g\in\SS(\R^d)$, Plancherel's Theorem gives
		\begin{align*}
			\int_{\R^d} T_m f(x) \, \overline{g(x)}\,\dd x
			&=
			\int_{\R^d} m(u) \,  \F f(u) \,  \overline{\F g(u)}\,\dd u \\
			&=
			\int_{\R^d} \F f(u) \,  \overline{\overline{m}(u) \,  \F g(u)}\,\dd u = 
			\int_{\R^d} f(x) \,  \overline{T_{\overline{m}}g(x)}\,\dd x.
		\end{align*}
		Therefore,  by the Cauchy-Schwarz inequality and applying  \eqref{eq:multiplier.plus.N.one.sided} with $\overline{m}$ in place of $m$, 
		\begin{align}\label{eq: Tmfg.bound}
			\begin{split}
				\left|\int_{\R^d} T_m f(x) \,  \overline{g(x)}\,\dd x\right|
				&\le 	\|f\|_{L^2((1+|x|^2)^{-\lambda_d}\,\dd x)}
				\|T_{\overline{m}}g\|_{L^2((1+|x|^2)^{\lambda_d}\,\dd x)} \\
				&\lesssim
				\|f\|_{L^2((1+|x|^2)^{-\lambda_d}\,\dd x)}
				\|g\|_{L^2((1+|x|^2)^{\lambda_d}\,\dd x)}.
			\end{split}
		\end{align}
		Since $\SS(\R^d)$ is dense in $L^2(\R^d)$, applying \eqref{eq: Tmfg.bound} with $g(x)=(1+|x|^2)^{-\lambda_d/2}h(x)$ gives
		\begin{align*}
				\|T_m f\|_{L^2((1+|x|^2)^{-\lambda_d}\,\dd x)}
			&
				=
				\big\|(1+|x|^2)^{-\lambda_d/2}T_m f\big\|_{L^2}
			 \\
			&
				=
				\sup_{\substack{h\in\SS(\R^d)\\ \|h\|_{L^2}\le 1}}
				\left|
				\int_{\R^d}(1+|x|^2)^{-\lambda_d/2}T_m f(x)\,
				\overline{h(x)}
				\,\dd x
				\right|
			 \\
			&
				=
				\sup_{\substack{h\in\SS(\R^d)\\ \|h\|_{L^2}\le 1}}
				\left|
				\int_{\R^d}T_m f(x)\,
				\overline{(1+|x|^2)^{-\lambda_d/2}h(x)}
				\,\dd x
				\right|
			 \\
			&
				\lesssim
				\|f\|_{L^2((1+|x|^2)^{-\lambda_d}\,\dd x)}
				\sup_{\substack{h\in\SS(\R^d)\\ \|h\|_{L^2}\le 1}}
				\big\|(1+|x|^2)^{-\lambda_d/2}h\big\|_{L^2((1+|x|^2)^{\lambda_d}\,\dd x)}
			 \\
			&
				=
				\|f\|_{L^2((1+|x|^2)^{-\lambda_d}\,\dd x)}.
		\end{align*}
		This proves \eqref{eq:multiplier.minus.N.one.sided}.
		
		\medskip\noindent\textit{Step 2.}
		We next prove the estimate
		\begin{align}\label{eq:positive.order.one.sided}
			\|J^k\varphi\|_{L^2(w^{-1}(x)\,\dd x)}
			\lesssim 
			\|\varphi\|_{\cH_w^k},
			\quad \varphi\in C_c^\infty(\R^d).
		\end{align}
		For $\bm{\alpha}$ with $0\le |\bm{\alpha}|\le k$, let
		\begin{align*}
			M_{\bm{\alpha}}(u)
			:=
			(1+|u|^2)^{k/2}
			\frac{\overline{(2\pi iu)^{\bm{\alpha}}}}
			{\sum_{0\le |\bm{\beta}|\le k}|(2\pi iu)^{\bm{\beta}}|^2},
			\quad u\in\R^d.
		\end{align*}
		The denominator is bounded below by $1$, so $M_{\bm{\alpha}}$ is smooth on $\R^d$. Let $0 \le |\bm{\ell}| \le \lambda_d$.  For $|u|\ge 1$, the denominator $\sum_{0\le |\bm{\beta}|\le k}|(2\pi iu)^{\bm{\beta}}|^2 \asymp |u|^{2k}$, while its derivatives of order $|\bm{\ell}|$ are $O(|u|^{2k-|\bm{\ell}|})$. The numerator's derivatives of order $|\bm{\ell}|$ are $O(|u|^{k+|\bm{\alpha}|-|\bm{\ell}|})$. Therefore, 
		\begin{align*}
			|D^{\bm{\ell}}M_{\bm{\alpha}}(u)|
			\lesssim  |u|^{|\bm{\alpha}|-k-|\bm{\ell}|},
			\quad |u|\ge 1.
		\end{align*}
		Since $0 \le |\bm{\alpha}|\le k$, these derivatives are bounded for $|u|\ge 1$, and boundedness  for $|u| \le 1$ follows from smoothness. Moreover, for every $u\in\R^d$,
		\begin{align}\label{eq:ak.resolution.one.sided}
			(1+|u|^2)^{k/2}
			=
			\sum_{0\le |\bm{\alpha}|\le k}
			M_{\bm{\alpha}}(u)(2\pi iu)^{\bm{\alpha}}.
		\end{align}
		Multiplying \eqref{eq:ak.resolution.one.sided} by $\F\varphi(u)$,  using $ \F[D^{\bm{\alpha}}\varphi](u) = (2\pi iu)^{\bm{\alpha}}\F\varphi(u)$ and the definition of $T_{M_{\bm{\alpha}}}$, we get
		\begin{align*}
			(1+|u|^2)^{k/2}\F\varphi(u)
			&=
			\sum_{0\le |\bm{\alpha}|\le k}
			M_{\bm{\alpha}}(u)\F[D^{\bm{\alpha}}\varphi](u) = \sum_{0\le |\bm{\alpha}|\le k}
			\F[T_{M_{\bm{\alpha}}}D^{\bm{\alpha}}\varphi](u).
		\end{align*}
		The left-hand side is $\F[J^k\varphi](u)$,
		thus taking the inverse Fourier transforms gives
		\begin{align*}
			J^k\varphi
			=
			\sum_{0\le |\bm{\alpha}|\le k}
			T_{M_{\bm{\alpha}}}D^{\bm{\alpha}}\varphi.
		\end{align*}
		Using \eqref{eq:multiplier.minus.N.one.sided} from Step 1, together with $ 	\|\cdot \|_{L^2((1+|x|^2)^{-\lambda_d}\,\dd x)}
		\asymp
		\|\cdot\|_{L^2(w^{-1}(x)\,\dd x)}$,
		we obtain
		\begin{align*}
			\|J^k\varphi\|_{L^2(w^{-1}(x)\,\dd x)}
			\lesssim
			\sum_{0\le |\bm{\alpha}|\le k}
			\|D^{\bm{\alpha}}\varphi\|_{L^2(w^{-1}(x)\,\dd x)}
			\lesssim
			C\|\varphi\|_{\cH_w^k}.
		\end{align*}
		This proves \eqref{eq:positive.order.one.sided}.
		
		\medskip\noindent\textit{Step 3.}
		We now finish the proof. Let $h\in\SS'(\R^d)$ be such that $	J^{-k}h\in L^2(w(x)\,\dd x) $. 
		Since $w\ge 1$, we have $J^{-k}h\in L^2$, and hence $h\in\cH^{-k}$. Thus, for every $\varphi\in C_c^\infty(\R^d)$, \eqref{eq:Hk.duality.bessel} gives $ 	\langle h,\varphi\rangle_{\cH^{-k},\cH^k}
		=
		\langle J^{-k}h,J^k\varphi\rangle_{L^2}$.
		By the Cauchy-Schwarz inequality and \eqref{eq:positive.order.one.sided} of Step 2,
		\begin{align} \label{eq:dual.pre.extension}
			\left|\langle h,\varphi\rangle_{\cH^{-k},\cH^k}\right|
			\le
			\|J^{-k}h\|_{L^2(w(x)\,\dd x)}
			\|J^k\varphi\|_{L^2(w^{-1}(x)\,\dd x)}
			\lesssim
			\|J^{-k}h\|_{L^2(w(x)\,\dd x)}
			\|\varphi\|_{\cH_w^k}.
		\end{align}
		Since $\cH_w^k$ is the completion of $C_c^\infty(\R^d)$ with respect to $\|\cdot\|_{\cH_w^k}$,  the functional $\varphi\mapsto \langle h,\varphi\rangle_{\cH^{-k},\cH^k}$
		extends from $C^\infty_c(\R^d)$ uniquely to a bounded linear functional on $\cH_w^k$. Therefore, $h\in\cH_w^{-k}$, and taking the supremum in \eqref{eq:dual.pre.extension} over all $\varphi\in C_c^\infty(\R^d)$ with $\|\varphi\|_{\cH_w^k}\le 1$ gives \eqref{eq:Hminus.kw.characterization}.
	\end{proof}
	
	\section{Proof of Proposition \ref{lem:preliminaries}}
	\label{sec:proof-lem-preliminaries}
	\noindent\textit{Proof of Proposition~\ref{lem:preliminaries} \hyperref[lem:preliminaries:i]{(i)}.}
	We use Assumption~\ref{assum.main}\eqref{assum.drift} here only for the boundedness of the spatial derivatives of $b_0$ and $b$ up to order one, so $k \ge 0$ is enough. For each $n\in\N$, the drift of the system \eqref{eq:X} is Lipschitz in the spatial variables, uniformly in time. Standard SDE theory yields existence and pathwise uniqueness of a strong solution.
	
	For the McKean--Vlasov equation \eqref{eq:MKV}, the same regularity implies that its drift is Lipschitz in the spatial variable and $\W_2$-Lipschitz in the measure variable, uniformly in time. Hence, standard theory for McKean--Vlasov equations, see, e.g., \cite[Theorem 4.21]{carmona2018probabilistic}, implies that \eqref{eq:MKV} has a unique strong solution.

	{\noindent \it \hyperref[lem:preliminaries:iv]{(ii)}.}
	Fix $p\in[1,\infty)$. We first show that
	\begin{align} \label{eq:p.moment.mu0}
		\int_{\R^d} |x|^p \, \mu_0(\dd x )
		< \infty.
	\end{align}
	By Assumption \ref{assum.main}\eqref{assum.init.transport}, $\mu_0$ satisfies the quadratic transport inequality. Since $\W_1\le \W_2$, where $\W_1$ denotes the $1$-Wasserstein distance, the Bobkov-G\"otze characterization \cite[Theorem 3.1]{bobkov1999exponential} implies that there exists $C<\infty$ such that, for every $1$-Lipschitz $F:\R^d\to\R$,
	\begin{align*}
		\int_{\R^d}\exp\bigg(\lambda\bigg(F-\int_{\R^d}F\,\dd \mu_0\bigg)\!\bigg)\,\dd \mu_0
		\le
		e^{C\lambda^2/2}, \quad  \lambda \in \R.
	\end{align*}
	Applying this with $F(x)=|x|$ gives the sub-Gaussianity of $\mu_0$, and hence $\mu_0$ has finite moments of all orders. In particular, this proves \eqref{eq:p.moment.mu0}.
	
	By Assumption \ref{assum.main}\eqref{assum.drift} and \eqref{eq:row.sum.assump}, 
	\begin{align*}
		\sup_{t\in[0,T]}|X_t^i|
		\lesssim 1+  
		|X_0^i| +\sup_{t\in[0,T]}|B_t^i|.
	\end{align*}
	Therefore, using $X_0^i\sim\mu_0$, \eqref{eq:p.moment.mu0}, and the fact that the supremum of a Brownian motion has finite moments of all orders,
	\begin{align*}
			\E\big[\sup_{t\in[0,T]}|X_t^i|^p\big]
			\lesssim 1 + 
			\E\big[|X_0^i|^p\big]
			+\E\big[\sup_{t\in[0,T]}|B_t^i|^p\big] \lesssim 1.
	\end{align*}
	Since $p\in[1,\infty)$ was arbitrary, this proves the first bound in \eqref{eq:uniform.moment.bound.X.Y}. The same argument applied to the McKean--Vlasov equation proves the second bound in \eqref{eq:uniform.moment.bound.X.Y}. Taking $p=2\lambda_d$ gives \eqref{eq:w.momet.statement}.
	
	{\noindent \it \hyperref[lem:preliminaries:ii]{(iii)}.} We use the synchronous coupling. Let $Y^1,\dots,Y^n$ be i.i.d. copies of  $Y$ from \eqref{eq:MKV}, where $Y^i$ is driven by the same Brownian motion $B^i$ as $X^i$, and $Y_0^i=X_0^i$. Using  \eqref{eq:row.sum.assump}, subtracting the two equations, adding and subtracting $b(s,Y_s^i,Y_s^j)$, and the Cauchy-Schwarz inequality gives
	\begin{align*} 
		\begin{split}
			&	\sum_{i=1}^n\E\big[|X_t^i-Y_t^i|^2\big] \\
			&\lesssim
			\int_0^t
			\sum_{i=1}^n
			\E\bigg[
			\Big|
			b_0(s,X_s^i)-b_0(s,Y_s^i)
			+\sum_{j=1}^n\xi_{ij}\big(b(s,X_s^i,X_s^j)-b(s,Y_s^i,Y_s^j)\big)
			\Big|^2
			\bigg]\,\dd s \\
			&\quad +
			\int_0^t
			\sum_{i=1}^n
			\E\bigg[
			\Big|
			\sum_{j=1}^n \xi_{ij}
			\big(
			b(s,Y_s^i,Y_s^j)-\langle \mu_s, b(s,Y_s^i,\cdot)\rangle
			\big)
			\Big|^2
			\bigg]\,\dd s .
		\end{split}
	\end{align*}
	By Assumption \ref{assum.main}\eqref{assum.drift}, Jensen's inequality, \eqref{eq:row.sum.assump}, and \eqref{eq:column.sum.assump}, the first integral is bounded by a uniform multiple of
	\begin{align*}
		\int_0^t\sum_{i=1}^n\E[|X_s^i-Y_s^i|^2]\,\dd s.
	\end{align*}
	For the second integral, Assumption \ref{assum.main}\eqref{assum:mat} gives $\xi_{ii}=0$. Thus, conditional on $Y_s^i$, the inner summands are independent and centered, so that upon squaring, the off-diagonal terms vanish. Since $b$ is bounded by Assumption \ref{assum.main}\eqref{assum.drift},
	\begin{align*}
		\sum_{i=1}^n
		\E\bigg[
		\bigg|
		\sum_{j=1}^n\xi_{ij}
		\big(b(s,Y_s^i,Y_s^j)-\langle \mu_s, b(s,Y_s^i,\cdot)\rangle\big)
		\bigg|^2
		\bigg]
		\lesssim
		\sum_{i,j=1}^n\xi_{ij}^2.
	\end{align*}
	Hence,  Gronwall's inequality and the Cauchy-Schwarz inequality give
	\begin{align} \label{eq: conv.emp.first.leg}
			\sup_{t\in[0,T]}
			\E\bigg[
			\W_2^2
			\Big(
			\frac1n\sum_{i=1}^n\delta_{X_t^i},
			\frac1n\sum_{i=1}^n\delta_{Y_t^i}
			\Big)
			\bigg] \le \frac{1}{n} \sup_{t\in[0,T]}\sum_{i=1}^n\E[|X_t^i-Y_t^i|^2]
			\lesssim
			\frac{1}{n}	\sum_{i,j=1}^n\xi_{ij}^2  \to 0
		\end{align}
		where the last step follows from  \eqref{eq:xi_sq_avg}.
	
	It remains to control the empirical measure of the i.i.d.~copies $Y^1,\dots,Y^n$. Equip $C([0,T];\R^d)$ with the uniform norm, and let $\overline{\W}_{2}$ denote the corresponding quadratic Wasserstein distance:
	\begin{align}\label{eq:path.wass}
		\overline{\W}_{2}(\nu,\nu')
		:=
		\inf_\pi
		\bigg(
		\int_{C([0,T];\R^d)^2}
		\sup_{t\in[0,T]}|\omega_t-\omega_t'|^2\,\pi(\dd \omega,\dd \omega')
		\bigg)^{1/2},
	\end{align}
	where the infimum is over all couplings $\pi$ of $\nu$ and $\nu'$. Since $Y^1,\dots,Y^n$ are i.i.d. copies of $Y$, the empirical path measures converge weakly a.s.~to $\Law(Y)$ by the Law of Large Numbers for Empirical Measures (see, e.g., \cite[Theorem 11.4.1]{dudley2018real}). Moreover, by the  Law of Large Numbers and Proposition \ref{lem:preliminaries}\eqref{lem:preliminaries:iv},
	\begin{align*}
		\frac1n\sum_{i=1}^n\sup_{t\in[0,T]}|Y_t^i|^2
		\to
		\E\bigg[\sup_{t\in[0,T]}|Y_t|^2\bigg], \quad \text{a.s.}
	\end{align*}
	Therefore, by \cite[Theorem 7.12, (iii)$\Rightarrow$(ii)]{villani2003topics},
	\begin{align*}
		\overline{\W}_{2}
		\bigg(
		\frac1n\sum_{i=1}^n\delta_{Y^i},
		\Law(Y)
		\bigg)
		\to0, \quad \text{a.s.}
	\end{align*}
	To upgrade this to convergence in expectation, note that  by the triangle inequality,
	\begin{align*}
		\overline{\W}_{2}^2
		\bigg(
		\frac1n\sum_{i=1}^n\delta_{Y^i},
		\Law(Y)
		\bigg)
		\lesssim
		\frac1n\sum_{i=1}^n\sup_{t\in[0,T]}|Y_t^i|^2
		+
		\E\big[\sup_{t\in[0,T]}|Y_t|^2\big].
	\end{align*}
	The right-hand side is uniformly integrable by Proposition \ref{lem:preliminaries}\eqref{lem:preliminaries:iv} with $p = 4$. Hence,
	\begin{align} \label{eq: path.wass.conv}
		\E\bigg[
		\overline{\W}_{2}^2
		\bigg(
		\frac1n\sum_{i=1}^n\delta_{Y^i},
		\Law(Y)
		\bigg)
		\bigg]\to0.
	\end{align}
	To pass from path space to fixed-time marginals, let $e_t:C([0,T];\R^d)\to\R^d$ denote the coordinate map $e_t(\omega):=\omega_t$, for $t\in[0,T]$. Since $e_t$ is $1$-Lipschitz, the pushforward estimate for Wasserstein distances, see, e.g., \cite[equation (7.1.6)]{ambrosio2005gradient}, gives
	\begin{align*}
		\W_2
		\bigg(
		\frac1n\sum_{i=1}^n\delta_{Y_t^i},
		\mu_t
		\bigg)
		&=
		\W_2
		\bigg(
		(e_t)_\#\frac1n\sum_{i=1}^n\delta_{Y^i},
		(e_t)_\#\Law(Y)
		\bigg) \notag \le
		\overline{\W}_{2}
		\bigg(
		\frac1n\sum_{i=1}^n\delta_{Y^i},
		\Law(Y)
		\bigg).
	\end{align*}
	Together with \eqref{eq: path.wass.conv}, we have
	\begin{align} \label{eq: conv.emp.second.leg}
		\sup_{t\in[0,T]}
		\E\bigg[
		\W_2^2
		\bigg(
		\frac1n\sum_{i=1}^n\delta_{Y_t^i},
		\mu_t
		\bigg)
		\bigg]\to0.
	\end{align}
	The conclusion follows from \eqref{eq: conv.emp.first.leg}, \eqref{eq: conv.emp.second.leg}, and the triangle inequality for $\W_2$.

	{\noindent \it \hyperref[lem:preliminaries:iii]{(iv)}.}  
	By Assumption~\ref{assum.main}\eqref{assum.drift}, $b_0$ and $b$ are Lipschitz, uniformly in $t$. Together with Jensen's inequality, \eqref{eq:row.sum.assump} and \eqref{eq:column.sum.assump}, it is straightforward to show that the drift of $X=(X^1,\dots,X^n)$ is Lipschitz on $(\R^d)^n$, uniformly in $t$, with Lipschitz constant independent of $n$. Moreover, by tensorization of the quadratic transport inequality, see \cite[Proposition 1.9]{gozlan2010transport}, Assumption~\ref{assum.main}\eqref{assum.init.transport} implies that $P_0=\mu_0^{\otimes n}$ satisfies
		\begin{align*}
			\W_2^2(\nu,P_0)
			\le
			\gamma_0 H(\nu\,|\,P_0),
			\quad
			\nu\in\P((\R^d)^n).
		\end{align*}
		By \cite[Proposition 8.11]{gozlan2010transport}, this quadratic transport inequality implies that $P_0$ satisfies a Poincar\'e inequality with constant $\gamma_0/2$, which is independent of $n$. The Poincaré inequality for $P_t$ then follows from \cite[Theorem 4.2]{cattiaux2014semi}.

	{\noindent \it \hyperref[lem:preliminaries:v]{(v)}.}  
	Let
	\begin{align*}
		\widetilde{b}(t,x)
		:=
		b_0(t,x)+\int_{\R^d} b(t,x,y)\,\mu_t(\dd y),
		\qquad (t,x)\in[0,T]\times\R^d.
	\end{align*}
	By Assumption~\ref{assum.main}\eqref{assum.init.transport}, $\mu_0$ satisfies the quadratic transport inequality \eqref{eq:init.transport}. Since $k\in\N$, Assumption~\ref{assum.main}\eqref{assum.drift} implies that $\widetilde{b}(t,\cdot)$ is Lipschitz, uniformly in $t\in[0,T]$, and that $\int_0^T|\widetilde{b}(t,0)|^2\,\dd t<\infty$. Hence, \cite[Proposition C.1]{lacker2023hierarchies} implies that the path measure $\mu_{[T]}:=\Law(Y)$ satisfies a quadratic transport inequality on $C([0,T];\R^d)$. In other words, recalling \eqref{eq:path.wass}, there exists $\gamma_T<\infty$ such that
	\begin{align} \label{eq:path.transport.ineq}
		\overline{\W}_{2}^2(\nu_{[T]},\mu_{[T]})
		\le
		\gamma_T H(\nu_{[T]} \,|\, \mu_{[T]}), \quad  \nu_{[T]}\in\P(C([0,T];\R^d)).
	\end{align}
	
	Fix $t\in[0,T]$ and $\nu\in\P(\R^d)$. If $H(\nu \,|\, \mu_t)=\infty$, \eqref{eq:mu_t_transport} holds trivially. Otherwise, define
	\begin{align*}
		\nu_{[T]}(\dd \omega)
		=
		\frac{\dd \nu}{\dd \mu_t}(\omega_t)\,\mu_{[T]}(\dd \omega).
	\end{align*}
	Then, with $e_t:C([0,T];\R^d)\to\R^d$ denoting the coordinate map $e_t(\omega):=\omega_t$, we have $(e_t)_\#\nu_{[T]}=\nu$, $(e_t)_\#\mu_{[T]}=\mu_t$, $H(\nu_{[T]} \,|\, \mu_{[T]})=H(\nu \,|\, \mu_t)$, and $e_t$ is $1$-Lipschitz. Hence, the pushforward estimate for Wasserstein distances, see, e.g., \cite[equation (7.1.6)]{ambrosio2005gradient}, together with \eqref{eq:path.transport.ineq}, gives
	\begin{align*}
		\W_2^2(\nu,\mu_t)
		\le
		\overline{\W}_{2}^2(\nu_{[T]},\mu_{[T]})
		\le
		\gamma_T H(\nu \,|\, \mu_t).
	\end{align*}
	This proves \eqref{eq:mu_t_transport}.

	\section{Proof of Lemma \ref{lem:max.ent.POC}}\label{sec.poc.pf}
	\medskip\noindent\textit{Proof. Step 1.}  	
	In this step, we  verify the assumptions needed to apply \cite[Theorems 2.8 and 2.11]{lacker2024quantitative}. In the notation of \cite{lacker2024quantitative}, we take $b_0^i=b_0$, $b^{ij}=b$, and $P_0=Q_0=\mu_0^{\otimes n}$. The hypotheses of \cite[Assumption A]{lacker2024quantitative} are satisfied on the time interval $[0,T]$. Indeed, since $T<\infty$, Assumption \ref{assum.main}\eqref{assum.drift}--\eqref{assum.init.transport} puts us in the setting of \cite[Example 2.3]{lacker2024quantitative}, because the coefficients $b_0^i$ and $b^{ij}$ are Lipschitz uniformly in $i,j$, the initial laws have finite second moments, and $Q_0^i=\mu_0$ satisfies the required quadratic transport inequality.  Hence, \cite[Assumption A(i)--(iii)]{lacker2024quantitative} holds. Finally, Assumption \ref{assum.main}\eqref{assum:mat} gives an interaction matrix $\xi$ with nonnegative entries, zero diagonal entries, and row sums equal to one. Thus \cite[Assumption A(iv)]{lacker2024quantitative}, namely the condition \cite[(rows)]{lacker2024quantitative}, also holds.

	Moreover, since the initial conditions are i.i.d. with law $\mu_0$, we have $P_0^v=Q_0^v=\mu_0^{\otimes |v|}$ for every $v\subset[n]$. Hence, the initial chaoticity assumptions in \cite[Theorems 2.8 and 2.11]{lacker2024quantitative} hold with $C_0=0$. The maximum entropy estimate \cite[Theorem 2.8]{lacker2024quantitative} therefore applies. The average entropy estimate \cite[Theorem 2.11]{lacker2024quantitative} additionally assumes that the column sums are bounded by $1$, as in \cite[(columns)]{lacker2024quantitative}. In our setting, the column sum bound \eqref{eq:column.sum.assump} in Assumption~\ref{assum.main}\eqref{assum:mat} holds with constant $C$. As in \cite[Remark 2.1]{lacker2024quantitative}, this only changes the constants in the estimates. Thus, \cite[Theorem 2.11]{lacker2024quantitative} applies, with the implicit constants in $\lesssim$ allowed to depend on $C$.

	\medskip\noindent\textit{Step 2.}
	Next, we identify the independent projection of \eqref{eq:X}. Since the initial positions are i.i.d. with common distribution $\mu_0$, and since \eqref{eq:row.sum.assump} in Assumption \ref{assum.main}\eqref{assum:mat} gives row sums equal to one, \cite[Remark 2.3]{jabin2025mean} or \cite[Remark 2.9]{lacker2026independent} implies that the independent projection is given by $n$ i.i.d. copies of the McKean-Vlasov equation \eqref{eq:MKV}. Therefore, $Q_t^i=\mu_t$ for every $i\in[n]$ and $t\in[0,T]$. Consequently, in the entropy estimates of \cite{lacker2024quantitative}, the reference law is
	$Q_t^v=\bigotimes_{i\in v}Q_t^i=\mu_t^{\otimes |v|}$
	for every $v\subset[n]$.

	\medskip\noindent\textit{Step 3.}
	We now prove part \eqref{lem:max.ent.POC:i} of Lemma \ref{lem:max.ent.POC}. Let $\delta:=\max_{i,j\in[n]}\xi_{ij}$. Since $\xi$ is nonnegative and its row sums are equal to one by Assumption \ref{assum.main}\eqref{assum:mat}, we have $\delta\le 1$. Hence, by \cite[Theorem 2.8]{lacker2024quantitative}, applied with $k=3$, and then by the data processing inequality from path space to the time-$t$ marginal,
	\begin{align*}
		\max_{\substack{v\subset[n]\\ |v|=3}} H(P_t^v \,|\,  Q_t^v)
		\lesssim
		(3\delta+1)(3\delta)^2
		\lesssim
		\delta^2 = 	\max_{i,j\in[n]}\xi_{ij}^2,
	\end{align*}
	where the last step follows from $\delta \lesssim 1$. Using $Q_t^v=\mu_t^{\otimes 3}$ from Step 2, we obtain
	\eqref{lem:max.ent.POC:i}.

	\medskip\noindent\textit{Step 4.}   
	We now prove part \eqref{lem:max.ent.POC:ii}. By \cite[Theorem 2.11]{lacker2024quantitative}, applied with $k=2$, and again using the data processing inequality from path space to time $t$,
	\begin{align*}
		\frac{1}{\binom n2}\sum_{\substack{v\subset[n]\\ |v|=2}}
		H(P_t^v\,|\, Q_t^v)
		\lesssim
		(1+2\delta)
		\bigg(
		\frac{4}{n^2}\sum_{i,j=1}^n \xi_{ij}^2
		+
		\frac{2}{n}\sum_{i=1}^n
		\Big(
		\sum_{j=1}^n(\xi_{ij}^2+\xi_{ji}^2)
		\Big)^2
		\bigg).
	\end{align*}
	Since $\delta\le1$, $1+2\delta \lesssim 1$. Also, Assumption \ref{assum.main}\eqref{assum:mat} gives row sums equal to one and zero diagonal entries. Hence, for each $i \in [n]$, the Cauchy-Schwarz inequality gives
	\begin{align*}
		\sum_{j=1}^n \xi_{ij}^2
		\ge
		\frac{1}{n-1}
		\Big(\sum_{j=1}^n \xi_{ij}\Big)^2
		=
		\frac{1}{n-1}.
	\end{align*}
	Therefore, 
	\begin{align*}
		\frac{1}{n-1} \sum_{j=1}^n(\xi_{ij}^2+\xi_{ji}^2)
		\le
		\Big(\sum_{j=1}^n \xi_{ij}^2\Big) \sum_{j=1}^n(\xi_{ij}^2+\xi_{ji}^2) \le 
		\Big(
		\sum_{j=1}^n(\xi_{ij}^2+\xi_{ji}^2)
		\Big)^2.
	\end{align*}
	Thus, 
	\begin{align*}
		\frac{1}{n^2}\sum_{i,j=1}^n \xi_{ij}^2
		&\le
		\frac{1}{n^2}\sum_{i=1}^n
		\sum_{j=1}^n(\xi_{ij}^2+\xi_{ji}^2) \\& \le
		\frac{n-1}{n^2}\sum_{i=1}^n
		\Big(
		\sum_{j=1}^n(\xi_{ij}^2+\xi_{ji}^2)
		\Big)^2 \le
		\frac{1}{n}\sum_{i=1}^n
		\Big(
		\sum_{j=1}^n(\xi_{ij}^2+\xi_{ji}^2)
		\Big)^2.
	\end{align*}
	It follows that
	\begin{align*}
		\frac{1}{\binom n2}\sum_{\substack{v\subset[n]\\ |v|=2}}
		H(P_t^v\,|\, Q_t^v)
		\lesssim
		\frac{1}{n}\sum_{i=1}^n
		\Big(
		\sum_{j=1}^n(\xi_{ij}^2+\xi_{ji}^2)
		\Big)^2.
	\end{align*}
	Finally, using $Q_t^v=\mu_t^{\otimes2}$ from Step 2, we obtain \eqref{lem:max.ent.POC:ii}. \qed

	\section{A compact embedding result and facts about Sobolev spaces}
	The following  embedding result is a consequence of \cite[Theorem 1.2]{benci1978second} if $d>2$. For a general $d\in\N$, we give a proof for completeness.
	
	\begin{lemma}\label{lem:compact.embed}
		Let $\ell\ge 1$ and $M>0$. Define
		\begin{align*}
			\K_{M,\ell}
			=
			\bigg\{
			f\in L^2_{\mathrm{loc}}(\R^d):
			\int_{\R^d}(1+|u|^2)^{-\ell+1}|f(u)|^2\,\dd u
			+
			\int_{\R^d}(1+|u|^2)^{-\ell}|\nabla f(u)|^2\,\dd u
			\le M
			\bigg\}.
		\end{align*} Then, $\K_{M,\ell}$ is compact in $ 	L^2\big((1+|u|^2)^{-\ell}\,\dd u\big)$.
	\end{lemma}
	
	\begin{proof}
		Let $\{f_n\}_{n\in\N}\subset \K_{M,\ell}$. For each $R>0$, let $B_R$ denote the open ball of radius $R$ in $\R^d$. For  $u\in B_R$, the bounds $	(1+|u|^2)^{-\ell+1}\ge (1+R^2)^{-\ell+1}$ and $		(1+|u|^2)^{-\ell}\ge (1+R^2)^{-\ell} $ imply 
		\begin{align*}
			\int_{B_R}|f_n(u)|^2\,\dd u
			\le
			(1+R^2)^{\ell-1}M,
			\qquad
			\int_{B_R}|\nabla f_n(u)|^2\,\dd u
			\le
			(1+R^2)^\ell M.
		\end{align*}
		Thus $\{f_n\}_{n\in\N}$ is bounded in $\cH^1(B_R)$ for every $R>0$. By weak compactness in $\cH^1(B_m)$, the Rellich-Kondrachov Theorem, and a diagonal extraction over $m\in\N$, there exist a subsequence, still denoted by $\{f_n\}_{n\in\N}$, and a function $f\in \cH^1_{\mathrm{loc}}(\R^d)$ such that, for every $m\in\N$, $f_n\to f$ strongly in $L^2(B_m)$ and weakly in $\cH^1(B_m)$. By restriction, the same convergences hold on $B_R$ for every $R>0$.
		
		Since $f_n \to  f$ weakly in $\cH^1(B_R)$, we have $f_n \to f$ weakly in $L^2(B_R)$ and $\nabla f_n \to  \nabla f$ weakly in $L^2(B_R;\R^d)$. Multiplication by the bounded functions $(1+|u|^2)^{(-\ell+1)/2}$ and $(1+|u|^2)^{-\ell/2}$ preserves these weak convergences. Therefore, by the weak lower semicontinuity of the $L^2(B_R) \times L^2(B_R; \R^d)$ norm, we have
		\begin{align*}
			&\int_{B_R}(1+|u|^2)^{-\ell+1}|f(u)|^2\,\dd u
			+
			\int_{B_R}(1+|u|^2)^{-\ell}|\nabla f(u)|^2\,\dd u \\
			&\le
			\liminf_{n\to\infty}
			\bigg[
			\int_{B_R}(1+|u|^2)^{-\ell+1}|f_n(u)|^2\,\dd u
			+
			\int_{B_R}(1+|u|^2)^{-\ell}|\nabla f_n(u)|^2\,\dd u
			\bigg]
			\le M.
		\end{align*}
		Letting $R\to\infty$ and using the Monotone Convergence Theorem gives $f\in\K_{M,\ell}$.
		
		It remains to prove that $f_n\to f$ in $L^2((1+|u|^2)^{-\ell} \, \dd u)$. For every $g\in\K_{M,\ell}$ and every $R>0$, the bound $	(1+|u|^2)^{-\ell}
		\le
		(1+R^2)^{-1}(1+|u|^2)^{-\ell+1} $, valid for $|u| > R$, implies that
		\begin{align*}
			\int_{\{|u|>R\}}(1+|u|^2)^{-\ell}|g(u)|^2\,\dd u
			\le
			\frac{1}{1+R^2}
			\int_{\R^d}(1+|u|^2)^{-\ell+1}|g(u)|^2\,\dd u
			\le
			\frac{M}{1+R^2}.
		\end{align*}
		Applying this estimate to $g=f_n$ and $g=f$, and using $(a-b)^2\le 2a^2+2b^2$, we obtain
		\begin{align*}
			\int_{\R^d}(1+|u|^2)^{-\ell}|f_n(u)-f(u)|^2\,\dd u
			\le
			\int_{B_R}(1+|u|^2)^{-\ell}|f_n(u)-f(u)|^2\,\dd u
			+
			\frac{4M}{1+R^2}.
		\end{align*}
		For fixed $R>0$, the convergence $f_n\to f$ in $L^2(B_R)$ implies that the first term tends to $0$ as $n\to\infty$. Therefore,
		\begin{align*}
			\limsup_{n\to\infty}
			\int_{\R^d}(1+|u|^2)^{-\ell}|f_n(u)-f(u)|^2\,\dd u
			\le
			\frac{4M}{1+R^2}.
		\end{align*}
		Letting $R\to\infty$, we conclude that $f_n\to f$ in $L^2((1+|u|^2)^{-\ell} \, \dd u)$, as desired.
	\end{proof}
	
	\subsection{Facts about  Sobolev spaces}
	Recall the definitions of the Sobolev spaces given in Section \ref{sec:sob.space}.
	\begin{lemma} \label{lem:sob}
		\begin{enumerate}[(i)]
			\item \label{embed.k1.k2}For any $k, k_1, k_2 \in \N $ with $k_1 \le k_2 $, we have the continuous embeddings
			\begin{align*}
				&L^2\hookrightarrow L^2(w^{-1}(x)\,\dd x),\quad
				L^2(w(x)\,\dd x)\hookrightarrow L^2,\quad
				\cH^{k}\hookrightarrow \cH^{k}_w,\\
				&\cH_w^{k_2}\hookrightarrow \cH_w^{k_1},\quad
				\cH^{-k}_w\hookrightarrow \cH^{-k},\quad
				\cH^{k_2}\hookrightarrow \cH^{k_1},\quad
				\cH^{-k_1}_w\hookrightarrow \cH^{-k_2}_w.
			\end{align*}
			\item \label{weighted-unweighted-pair}
			For any $k_1, k_2 \in \N$ with $k_1\le k_2$, $u\in \cH_w^{-k_1}$, and $f\in \cH_w^{k_2}$, we have
			\begin{align} \label{eq: lower.higher}
				\langle u,f\rangle_{\cH^{-k_1}_w,\cH_w^{k_1}}
				=
				\langle u,f\rangle_{\cH^{-k_2}_w,\cH_w^{k_2}}.
			\end{align}
			If in addition, $f\in \cH^{k_2}$, then
			\begin{align} \label{eq:weighted.unweighted}
				\langle u,f\rangle_{\cH^{-k_1}_w,\cH_w^{k_1}}
				=
				\langle u,f\rangle_{\cH^{-k_1},\cH^{k_1}}
				=
				\langle u,f\rangle_{\cH^{-k_2},\cH^{k_2}}
				=
				\langle u,f\rangle_{\cH^{-k_2}_w,\cH_w^{k_2}}.
			\end{align}
			\item\label{lem:sob:vii} For any $k \in \N$, $C_c^\infty(\R^d)$ is dense in  $\cH^{k}_w$ and 
			\begin{align} \label{eq: dual.set}
				\bigg\{f \mapsto \int_{\R^d} \varphi (x) f(x) \, \dd x : \varphi \in C^\infty_c(\R^d)\bigg\}
			\end{align}
			is dense in  $\cH^{-k}_w$. \label{denseness.smooth.in.neg}
			\item\label{dual.smooth} For any $k\in\N$ and $u\in C_c^\infty(\R^d)$, let $u\in \cH_w^{-k}$ denote the continuous linear functional
			\begin{align*}
				\varphi\mapsto \int_{\R^d}u(x)\varphi(x)\,\dd x, \quad \varphi \in \cH^k_w.
			\end{align*}
			Let $\widetilde{u}\in \cH_w^k$ be the element given by the Riesz Representation Theorem, namely
			\begin{align} \label{eq:riesz.r.t.}
				\langle u,\varphi\rangle_{\cH_w^{-k},\cH_w^k}
				=
				\langle \widetilde{u},\varphi\rangle_{\cH_w^k},
				\quad \varphi\in \cH_w^k.
			\end{align}
			Then $\widetilde{u}\in \cH_w^m$ for every  $\N \ni m\ge k$.
		\end{enumerate}
	\end{lemma}
	\noindent\textit{Proof.} {\it \hyperref[embed.k1.k2]{(i)}.}	 This is clear from the definitions of the spaces.
	
	{\noindent \it \hyperref[eq: lower.higher]{(ii)}.}
	By part \eqref{embed.k1.k2}, we have the continuous embeddings $\cH_w^{k_2}\hookrightarrow \cH_w^{k_1}$ and $\cH_w^{-k_1}\hookrightarrow \cH_w^{-k_2}$. Thus, both pairings in \eqref{eq: lower.higher} are well-defined. Moreover, the functional induced by $u$ on $\cH_w^{k_2}$ is just the restriction of the original functional on $\cH_w^{k_1}$. Since $f$ is the same underlying function in both spaces, \eqref{eq: lower.higher} follows.
	
	Similarly, if, in addition, $f\in \cH^{k_2}$, then by part \eqref{embed.k1.k2}, we have the continuous embeddings $\cH^{k_2}\hookrightarrow \cH^{k_1}\hookrightarrow \cH_w^{k_1}$ and $\cH^{k_2}\hookrightarrow \cH_w^{k_2}\hookrightarrow \cH_w^{k_1}$. By duality, these imply $\cH^{-k_1}_w\hookrightarrow \cH^{-k_1}\hookrightarrow \cH^{-k_2}$ and $\cH^{-k_1}_w\hookrightarrow \cH^{-k_2}_w$. Hence, all four pairings in \eqref{eq:weighted.unweighted} are well-defined. The functionals induced by $u$ on $\cH^{k_1}$, $\cH^{k_2}$, and $\cH_w^{k_2}$ are restrictions of the same original functional on $\cH_w^{k_1}$. Since $f$ is the same underlying function in each of these spaces, all four pairings in \eqref{eq:weighted.unweighted} are equal.

	{\noindent \it \hyperref[lem:sob:vii]{(iii)}.}  The density of $C_c^\infty(\R^d)$ in $\cH_w^k$ holds by definition, since $\cH_w^k$ is the completion of $C_c^\infty(\R^d)$. Next, by the Riesz Representation Theorem, the Riesz map $R:\cH_w^k\to \cH_w^{-k}$, defined by
	\begin{align*}
		\langle R\varphi,f\rangle_{\cH_w^{-k},\cH_w^k}
		=
		\langle \varphi,f\rangle_{\cH_w^k},
		\qquad \varphi,f\in \cH_w^k,
	\end{align*}
	is an isometric isomorphism. Therefore, $R(C_c^\infty(\R^d))$ is dense in $\cH_w^{-k}$. For $\varphi, f \in C_c^\infty(\R^d)$, integration by parts gives
	\begin{align} \label{eq:riesz.i.b.p}
		\langle R\varphi,f\rangle_{\cH_w^{-k},\cH_w^k}
		=
		\langle \varphi,f\rangle_{\cH_w^k}
		=
		\int_{\R^d} L_w\varphi(x)f(x)\,\dd x,
	\end{align}
	where the operator $L_w : C^\infty_c(\R^d) \to C^\infty_c(\R^d)$ is defined by 
	\begin{align} \label{eq:Lw.def}
		L_w\varphi
		=
		\sum_{0 \le |\boldsymbol{\alpha}|\le k}
		(-1)^{|\boldsymbol{\alpha}|}
		D^{\boldsymbol{\alpha}}
		\big(w^{-1}D^{\boldsymbol{\alpha}}\varphi\big)
		\in C_c^\infty(\R^d).
	\end{align}
	Since both sides of \eqref{eq:riesz.i.b.p} define continuous linear functionals of $f$ on $\cH_w^k$, the identity \eqref{eq:riesz.i.b.p} extends to all $f\in \cH_w^k$. Hence, every element of $R(C_c^\infty(\R^d))$ belongs to the set \eqref{eq: dual.set}.
	Since $R(C_c^\infty(\R^d))$ is dense in $\cH_w^{-k}$, the latter set is dense in $\cH_w^{-k}$.
	
	{\noindent \it \hyperref[denseness.smooth.in.neg]{(iv)}.}
	We need some additional notation. Let $\D(\R^d):=C_c^\infty(\R^d)$, and let $\D'(\R^d):=\D(\R^d)'$ be its continuous dual, the space of distributions. For $S\in\D'(\R^d)$ and $\varphi\in\D(\R^d)$, we write $\langle S,\varphi\rangle_{\D',\D}$ for the duality pairing. For $\varphi\in\cH_w^k$, define $L_w\varphi\in\D'(\R^d)$ by
	\begin{align} \label{eq:Lw.def.general}
		\langle L_w \varphi ,f\rangle_{\D',\D}
		=
		\langle \varphi,f\rangle_{\cH_w^k},
		\quad f\in C_c^\infty(\R^d).
	\end{align}
	Note that this definition is consistent with \eqref{eq:Lw.def} when $\varphi\in C_c^\infty(\R^d)$. 
	
	\medskip\noindent\textit{Step 1.}
	We first show that $L_w\widetilde{u}=u$ in $\D'(\R^d)$.  Taking $\varphi=\widetilde{u}$ in \eqref{eq:Lw.def.general} and using \eqref{eq:riesz.r.t.}, we get
	\begin{align*}
		\langle L_w\widetilde{u},f\rangle_{\D',\D}
		=
		\langle \widetilde{u},f\rangle_{\cH_w^k}
		=
		\langle u,f\rangle_{\cH_w^{-k},\cH_w^k}
		=
		\int_{\R^d}u(x)f(x)\,\dd x,
		\quad f\in C_c^\infty(\R^d).
	\end{align*}
	Hence, $L_w\widetilde{u}=u$ in $\D'(\R^d)$.
	
	\medskip\noindent\textit{Step 2.}
	Next, let $v:=w^{-1/2}\widetilde{u}$. We claim that $v\in \cH^k$, that is, $D^{\boldsymbol{\alpha}}v\in L^2(\R^d)$ for every multi-index $\boldsymbol{\alpha}$ with $0\le |\boldsymbol{\alpha}|\le k$. We prove this by induction over $|\boldsymbol{\alpha}|$. For $|\boldsymbol{\alpha}|=0$, this follows from $\widetilde{u}\in \cH_w^k$. Suppose now that $1\le m\le k$, and that $D^{\boldsymbol{\beta}}v\in L^2(\R^d)$ for all multi-indices $\boldsymbol{\beta}$ with $0\le |\boldsymbol{\beta}|\le m-1$. Let $\boldsymbol{\alpha}$ be such that $|\boldsymbol{\alpha}|=m$. Since $\widetilde{u}=w^{1/2}v$, the Leibniz rule gives, in the sense of distributions and for some constants $c_{\boldsymbol{\alpha},\boldsymbol{\gamma}}$,
	\begin{align*}
		D^{\boldsymbol{\alpha}}v
		=
		w^{-1/2}D^{\boldsymbol{\alpha}}\widetilde{u}
		-
		\sum_{0 \neq \boldsymbol{\gamma}\le \boldsymbol{\alpha}}
		c_{\boldsymbol{\alpha},\boldsymbol{\gamma}}
		\big(w^{-1/2}D^{\boldsymbol{\gamma}}w^{1/2}\big)
		D^{\boldsymbol{\alpha}-\boldsymbol{\gamma}}v.
	\end{align*}
	The first term on the right-hand side belongs to $L^2(\R^d)$, because $\widetilde{u}\in \cH_w^k$ and $|\boldsymbol{\alpha}|\le k$. For each term in the sum, we have $|\boldsymbol{\alpha}-\boldsymbol{\gamma}|\le m-1$, so $D^{\boldsymbol{\alpha}-\boldsymbol{\gamma}}v\in L^2(\R^d)$ by the induction hypothesis. Moreover, $w^{-1/2}D^{\boldsymbol{\gamma}}w^{1/2}$ is bounded on $\R^d$. Therefore, every term on the right-hand side belongs to $L^2(\R^d)$. Hence, $D^{\boldsymbol{\alpha}}v \in L^2(\R^d)$. This completes the induction and proves $v\in \cH^k$.
	
	\medskip\noindent\textit{Step 3.} Recalling the operator $L_w$ in \eqref{eq:Lw.def},  define $P : C^\infty_c (\R^d) \to C^\infty_c(\R^d)$ by
	\begin{align} \label{eq:P.operator.def}
		P\varphi
		=
		w^{1/2}L_w(w^{1/2}\varphi) = \sum_{0 \le |\boldsymbol{\alpha}|\le k}
		(-1)^{|\boldsymbol{\alpha}|}
		w^{1/2}
		D^{\boldsymbol{\alpha}}
		\big(
		w^{-1}D^{\boldsymbol{\alpha}}(w^{1/2}\varphi)
		\big).
	\end{align}
	Expanding the right-hand side by the Leibniz rule shows that $P$ is a differential operator of order $2k$. More precisely, there exist coefficients $a_{\boldsymbol{\gamma}}\in C_b^\infty(\R^d)$ such that
	\begin{align} \label{eq: P.varphi}
		(P\varphi)(x)
		=
		\sum_{0\le |\boldsymbol{\gamma}|\le 2k}
		a_{\boldsymbol{\gamma}}(x)D^{\boldsymbol{\gamma}}\varphi(x),
		\quad x\in\R^d,\quad \varphi\in C_c^\infty(\R^d).
	\end{align}
	\medskip\noindent\textit{Step 4.} 
	We next show that $P$ is elliptic in the sense of \cite[equation (4.1), p.~17]{taylor2023partial2}. First, we show that $P\in OPS^{2k}_{1,0}(\R^d)$, as defined on \cite[p.~3]{taylor2023partial2}. From \eqref{eq: P.varphi}, we see that the total symbol of the differential operator $P$ is given by
	\begin{align} \label{eq:P.expansion}
		p(x, \zeta) := \sum_{0 \le |\bm{\gamma}|\le 2k} a_{\bm{\gamma}} (x) \zeta^{\bm{\gamma}}, \quad x, \zeta \in \R^d,
	\end{align}
	where $\zeta^{\bm{\gamma}} := \zeta_1^{\gamma_1} \cdots \zeta_d^{\gamma_d}$ for $\zeta = (\zeta_1, \dots, \zeta_d) $ and $\bm{\gamma} = (\gamma_1, \dots, \gamma_d)$. For any multi-indices $\bm{\alpha}$ and ${\bm{\nu}}$, since $D_x^{\bm{\nu}} a_{\bm{\gamma}}$ is bounded and $D_\zeta^{\bm{\alpha}}\zeta^{\bm{\gamma}}$ is either zero or a polynomial of degree at most $|\bm{\gamma}|-|\bm{\alpha}|$, we have
	\begin{align*}
		\big|D^{\bm{\nu}}_x D^{\bm{\alpha}}_\zeta p(x, \zeta)\big|
		\lesssim
		(\sqrt{1+|\zeta|^2})^{2k-|\bm{\alpha}|}.
	\end{align*}
	Hence, $P\in OPS^{2k}_{1,0}(\R^d)$. Moreover, from \eqref{eq:P.operator.def}, its principal highest-order part amounts to $\sum_{|\bm{\alpha}|=k}(-1)^{|\bm{\alpha}|}D^{2\bm{\alpha}}$, and, with the symbol convention in \eqref{eq:P.expansion}, the corresponding principal symbol is $ 	(-1)^k\sum_{|\bm{\alpha}| = k} \zeta^{2 \bm{\alpha}}$. 
	This is the degree $2k$ part of the right-hand side of \eqref{eq:P.expansion}. All remaining terms have degree at most $2k-1$ in $\zeta$. Writing
	\begin{align*}
		\wh{p}(x,\zeta)
		:=
		p(x,\zeta)-(-1)^k\sum_{|\bm{\alpha}|=k}\zeta^{2\bm{\alpha}},
	\end{align*}
	and using that the coefficients $a_{\bm{\gamma}}(x)$ are bounded uniformly in $x\in\R^d$, we have $|\wh{p}(x,\zeta)|\lesssim 1+|\zeta|^{2k-1}$.
	Together with the Multinomial Theorem, which implies $|\zeta|^{2k}\lesssim\sum_{|\bm{\alpha}|=k}\zeta^{2\bm{\alpha}}$,  we see that there exist $C, R<\infty$ such that, for all $|\zeta|\ge R$,
	\begin{align*}
		|p(x,\zeta)| \ge \sum_{|\bm{\alpha}| = k } \zeta^{2 \bm{\alpha}} - |\wh{p}(x, \zeta)| 
		&\ge
		\frac{1}{C}|\zeta|^{2k}
		-
		C(1+|\zeta|^{2k-1})
		\gtrsim
		|\zeta|^{2k}
		\gtrsim
		(\sqrt{1+|\zeta|^2})^{2k}.
	\end{align*}
	Consequently, $|p(x,\zeta)^{-1}|\lesssim(\sqrt{1+|\zeta|^2})^{-2k}$ for all $|\zeta|\ge R$. Thus, $P$ is elliptic.

	\medskip\noindent\textit{Step 5.} 
	We next show that $v\in \cH^m$ for every $m \in \N$. Since $L_w\widetilde{u}=u$ in $\D'(\R^d)$ by Step 1 and $\widetilde{u}=w^{1/2}v$, we have $Pv=w^{1/2}u$ in $\D'(\R^d)$. Let $f=w^{1/2}u$. Since $u\in C_c^\infty(\R^d)$, we have $f\in \cH^s$ for every $s\in\R$.
	
	Because $P$ is elliptic by Step 4, the parametrix construction in \cite[equations (4.11)--(4.12)]{taylor2023partial2} provides $q(x,D)\in OPS^{-2k}_{1,0}(\R^d)$ and a smoothing operator $r(x,D)\in OPS^{-\infty}(\R^d)$ such that this distributional solution satisfies
	\begin{align*} 
		v=q(x,D)f-r(x,D)v.
	\end{align*}
	By \cite[Proposition 5.5, p.~20]{taylor2023partial2}, these operators satisfy, for every $s\in\R$ and $N\in\N$,
	\begin{align}\label{eq:smoothing}
		q(x,D):\cH^{s}\to \cH^{s+2k},
		\quad
		r(x,D):\cH^{s}\to \cH^{s+N}.
	\end{align}
	
	By Step 2, $v\in \cH^k$. Suppose that $v\in \cH^{k+j}$ for some $j\ge 0$. Then, taking $s=k+j$ and $N=1$ in \eqref{eq:smoothing}, $r(x,D)v\in \cH^{k+j+1}$. Also, since $f\in \cH^{k+j+1-2k}$, \eqref{eq:smoothing} gives $q(x,D)f\in \cH^{k+j+1}$. Therefore, $v=q(x,D)f-r(x,D)v\in \cH^{k+j+1}$. By induction, $v\in \cH^m$ for every $m\in\N$.

	\medskip\noindent\textit{Step 6.} 
	Finally, using $\widetilde{u}=w^{1/2}v$,  for any $m\in\N$ and $0 \le |\bm{\alpha}|\le m$, Leibniz rule gives
	\begin{align*}
		w^{-1/2}D^{\bm{\alpha}}\widetilde{u}
		=
		w^{-1/2}D^{\bm{\alpha}}(w^{1/2}v)
		=
		\sum_{{\bm{0}} \le \bm{\gamma} \le \bm{\alpha}} c_{{\bm{\alpha}},{\bm{\gamma}}}\,(w^{-1/2}D^{\bm{\gamma}} w^{1/2})\,D^{\bm{\alpha-\gamma}}v.
	\end{align*}
	The coefficients $w^{-1/2}D^{\bm{\gamma}} w^{1/2}$ are bounded and $D^{\bm{\alpha-\gamma}}v\in L^2(\R^d)$ since $v\in \cH^{m}$ by Step 5, hence $w^{-1/2}D^{\bm{\alpha}}\widetilde{u}\in L^2$ for all $\bm{\alpha}$ with $0 \le |\bm{\alpha}|\le m$. This is exactly $\widetilde{u}\in \cH_w^m$. \qed

	\section{A  mollification result}
	Recall the definition of  $w$ in \eqref{eq:w.def}. Let $(\rho_\varepsilon)_{\varepsilon>0}$ be an even symmetric mollifier on $\R^d$ with $\rho_\varepsilon\ge0$, $\int_{\R^d}\rho_\varepsilon(x)\,\dd x=1$, and $\operatorname{supp}(\rho_\varepsilon)\subset B_\varepsilon$, where we recall $B_\varepsilon$ is the open ball with radius $\varepsilon$ in $\R^d$.

	\begin{lemma}\label{lem:weighted-L2-mollification}
		Let  $h : \R^d \to \R$ be such that  $h\in\cH^1$ and $h,\nabla h\in L^2(w(x)\,\dd x)$. Then:
		\begin{align}\label{eq:weighted.L2.mollification:h}
			h*\rho_\varepsilon\to h
			 \text{ and } 	\nabla(h*\rho_\varepsilon)\to\nabla h
			\quad\text{in }L^2(w(x)\,\dd x).
		\end{align}
	\end{lemma}
	
	\begin{proof}
		Since $w(x)=1+|x|^{2\lambda_d}$, for every $x,z\in\R^d$ with $|z|\le1$,
		\begin{align*}
			w(x+z)
			=
			1+|x+z|^{2\lambda_d}
			\lesssim
			1+|x|^{2\lambda_d}
			=
			w(x).
		\end{align*}
		Thus, for every $|z|\le1$,
		\begin{align}\label{eq:weighted.L2.mollification:translation.bound}
			\int_{\R^d}|h(x-z)|^2w(x)\,\dd x
			&=
			\int_{\R^d}|h(y)|^2w(y+z)\,\dd y
			\lesssim
			\|h\|_{L^2(w(x)\,\dd x)}^2.
		\end{align}
		
		We next prove that
		\begin{align}\label{eq:weighted.L2.mollification:translation.continuity}
			\|h(\cdot-z)-h\|_{L^2(w(x)\,\dd x)}
			\to0
			\quad\text{as }z\to0.
		\end{align}
		First, we claim that $C_c^\infty(\R^d)$ is dense in $L^2(w(x)\,\dd x)$. Let $f\in L^2(w(x)\,\dd x)$. Since $f\mathbf{1}_{B_R}\to f$ in $L^2(w(x)\,\dd x)$ as $R\to\infty$, it is enough to approximate $f\mathbf{1}_{B_R}$ for fixed $R$. Since $w\ge1$, we have $f\mathbf{1}_{B_R}\in L^2(B_R)$. Since $C_c^\infty(B_R)$ is dense in $L^2(B_R)$, there exists a sequence $g_m\in C_c^\infty(B_R)$ such that $g_m\to f\mathbf{1}_{B_R}$ in $L^2(B_R)$. Since $w$ is bounded above on $B_R$, the same convergence holds in $L^2(w(x)\,\dd x)$. This proves that $C_c^\infty(\R^d)$ is dense in $L^2(w(x)\,\dd x)$. Returning to the proof of \eqref{eq:weighted.L2.mollification:translation.continuity}, let $g\in C_c^\infty(\R^d)$. By \eqref{eq:weighted.L2.mollification:translation.bound}, for $|z|\le1$,
		\begin{align}\label{eq:weighted.L2.mollification:translation.approx}
			\begin{split}
				&\, \|h(\cdot-z)-h\|_{L^2(w(x)\,\dd x)} \\
				&\le
				\|(h-g)(\cdot-z)\|_{L^2(w(x)\,\dd x)}
				+
				\|g(\cdot-z)-g\|_{L^2(w(x)\,\dd x)}
				+
				\|g-h\|_{L^2(w(x)\,\dd x)}
				\\
				&\lesssim
				\|h-g\|_{L^2(w(x)\,\dd x)}
				+
				\|g(\cdot-z)-g\|_{L^2(w(x)\,\dd x)}.
			\end{split}
		\end{align}
		Since $g\in C_c^\infty(\R^d)$, as $z \to 0$,
		$\|g(\cdot-z)-g\|_{L^2(w(x)\,\dd x)}\to0$.
		Combining this with the density of $C_c^\infty(\R^d)$ in $L^2(w(x)\,\dd x)$ and \eqref{eq:weighted.L2.mollification:translation.approx}, we obtain \eqref{eq:weighted.L2.mollification:translation.continuity}.
		
		We now prove \eqref{eq:weighted.L2.mollification:h}. By Minkowski's integral inequality, 
		\begin{align}\label{eq:weighted.L2.mollification:minkowski}
			\|h*\rho_\varepsilon-h\|_{L^2(w(x)\,\dd x)}
			\le
			\int_{\R^d}\rho_\varepsilon(z)
			\|h(\cdot-z)-h\|_{L^2(w(x)\,\dd x)}
			\,\dd z.
		\end{align}
		Since $\operatorname{supp}(\rho_\varepsilon)\subset B_\varepsilon$ and $\int_{\R^d}\rho_\varepsilon(z)\,\dd z=1$, \eqref{eq:weighted.L2.mollification:minkowski} implies
		\begin{align*}
			\|h*\rho_\varepsilon-h\|_{L^2(w(x)\,\dd x)}
			\le
			\sup_{|z|<\varepsilon}
			\|h(\cdot-z)-h\|_{L^2(w(x)\,\dd x)}.
		\end{align*}
		The right-hand side  tends to $0$ by \eqref{eq:weighted.L2.mollification:translation.continuity}. This proves the first convergence in \eqref{eq:weighted.L2.mollification:h}.  The proof of the second convergence in \eqref{eq:weighted.L2.mollification:h} proceeds along the same lines by applying the argument to each $\partial_\ell h$, $\ell=1,\dots,d$, and noting that differentiation commutes with convolution. 
	\end{proof}

	\section{A weighted Pinsker inequality}
	
	We will use the following weighted Pinsker-type estimate, in the spirit of \cite{bolley2005weighted}.
	
	\begin{lemma}\label{lem:weightedPinsker}
		Let $P$ and $Q$ be probability measures on a measurable space $\mathcal{X}$. Let $\psi\ge0$ be such that $\int_{\mathcal{X}} \psi^2\,\dd(P+Q)<\infty$. Then, with $|P-Q|$ denoting the total variation measure of $P-Q$, we have
		\begin{align} \label{eq:weightedPinsker}
			\bigg(\int_{\mathcal{X}} \psi\,\dd|P-Q|\bigg)^2
			\le
			2\,H(P\,|\,Q)\int_{\mathcal{X}} \psi^2\,\dd(P+Q).
		\end{align}
	\end{lemma}
	
	\begin{proof}
		If $H(P\,|\,Q)=\infty$, \eqref{eq:weightedPinsker} holds trivially. We may therefore assume that $P\ll Q$, and write $r=\dd P/\dd Q$. Then, $\dd|P-Q|=|r-1|\,\dd Q$. Using  $|r-1|=(\sqrt{r}+1)|\sqrt{r}-1|$ and the Cauchy-Schwarz inequality,
		\begin{align} \label{eq:weighted.0}
			\int_{\mathcal{X}} \psi\,\dd|P-Q|
			&\le
			\bigg(\int_{\mathcal{X}} \psi^2(\sqrt{r}+1)^2\,\dd Q\bigg)^{1/2}
			\bigg(\int_{\mathcal{X}}(\sqrt{r}-1)^2\,\dd Q\bigg)^{1/2}.
		\end{align}
		For the first term on the right-hand side, since $(\sqrt{r}+1)^2\le 2(r+1)$,
		\begin{align} \label{eq:weighted.1}
			\int_{\mathcal{X}} \psi^2(\sqrt{r}+1)^2\,\dd Q
			\le
			2\int_{\mathcal{X}} \psi^2(r+1)\,\dd Q
			=
			2\int_{\mathcal{X}} \psi^2\,\dd(P+Q).
		\end{align}
		For the second term on the right-hand side of \eqref{eq:weighted.0}, we use the elementary inequality
		\begin{align}\label{eq:elem.ineq}
			(\sqrt{r}-1)^2\le r\log r-r+1,
			\qquad r\ge0,
		\end{align}
		with the convention $0\log0:=0$.   Indeed, for $r>0$, setting $y=\sqrt{r}$ shows that the difference between the right-hand side and the left-hand side of \eqref{eq:elem.ineq} is $2y(y\log y-y+1)\ge0$, and for $r=0$, both sides are equal to $1$. 
		Hence,
		\begin{align} \label{eq:weighted.2}
			\int_{\mathcal{X}}(\sqrt{r}-1)^2\,\dd Q
			\le
			\int_{\mathcal{X}}(r\log r-r+1)\,\dd Q
			=
			H(P\,|\,Q),
		\end{align}
		because $\int_{\mathcal{X}}(r-1)\,\dd Q=0$. Putting \eqref{eq:weighted.1} and  \eqref{eq:weighted.2} back into \eqref{eq:weighted.0} gives \eqref{eq:weightedPinsker}.
	\end{proof}
	
	\bibliographystyle{amsplain} 
	\bibliography{biblio}
	
\end{document}